\documentclass{amsart}
\usepackage[utf8]{inputenc}
\usepackage{biblatex}
\usepackage{amsmath}
\usepackage{amsthm}
\usepackage{comment}
\usepackage{mathtools}
\usepackage[margin=1.25in]{geometry}
\usepackage{xcolor}
\usepackage{hyperref}
\usepackage{thmtools}
\usepackage{thm-restate}
\usepackage{cleveref}

\newcommand{\coker}{\operatorname{coker}}

\newtheorem{theorem}{Theorem}[section]
\newtheorem{lemma}[theorem]{Lemma}
\newtheorem{proposition}[theorem]{Proposition}

\newtheorem{claim}{Claim}[theorem]

\theoremstyle{definition}
\newtheorem{definition}[theorem]{Definition}

\theoremstyle{remark}
\newtheorem{remark}[theorem]{Remark}
\newcommand{\R}{\mathbb{R}}
\newcommand{\N}{\mathbb{N}}

\newcommand{\nl}{\newline}
\newcommand{\st}{\text{ s.t. }}
\newcommand{\eps}{\epsilon}
\newcommand{\rmk}{\text{Remark: }}
\newcommand{\BE}{\textnormal{BE}}

\newcommand{\Pf}{\textbf{Proof: } }
\newcommand{\Ind}{\textnormal{Ind}}
\newcommand{\Ric}{\text{Ric}}
\newcommand{\C}{\mathbb{C}}
\newcommand{\Leu}{L_{\eps, u}}

\newcommand{\QAC}{Q_{AC_{\eps}} }
\newcommand{\QBE}{Q_{\text{BE}_{\eps}}}
\newcommand{\sn}{\text{sn}}
\newcommand{\cn}{\text{cn}}
\newcommand{\dn}{\text{dn}}
\newcommand{\bg}{\overline{g}}

\newcommand{\bw}{\overline{w}}
\newcommand{\btau}{\overline{\tau}}
\newcommand{\brho}{\overline{\rho}}
\newcommand{\bkappa}{\overline{\rho}}

\newtheorem{Remark}{Remark}

\newtheorem{corollary}[theorem]{Corollary}
\renewcommand{\eps}{\varepsilon}

\numberwithin{equation}{section}

\title{Geometric Variations of an Allen--Cahn Energy on Hypersurfaces}

\author{Jared Marx-Kuo \and \'{E}rico Melo Silva}
\date{April 2023}
\begin{document}

\maketitle

\begin{abstract}
\noindent We introduce an Allen--Cahn type functional, $\text{BE}_{\eps}$, that defines an energy on separating hypersurfaces, $Y$, of closed Riemannian Manifolds. We establish $\Gamma$-convergence of $\text{BE}_{\eps}$ to the area functional, and compute first and second variations of this functional under hypersurface pertrubations. We then compute an explicit expansion for the variational formula as $\epsilon \to 0$. A key component of this proof is the invertibility of the linearized Allen--Cahn equation about a solution, on the space of functions vanishing on $Y$. We also relate the index and nullity of $\text{BE}_{\eps}$ to the Allen--Cahn index and nullity of a corresponding solution vanishing on $Y$. We apply the second variation formula and index theorems to show that the family of $2p$-dihedrally symmetric solutions to Allen--Cahn on $S^1$ have index $2p - 1$ and nullity $1$.\end{abstract}
\tableofcontents
%
\section{Introduction.}
\noindent Let $(M^n, g)$ be a Riemannian manifold and $u \in W^{1,2}(M)$. The Allen--Cahn energy is given by 
\begin{equation} \label{ACEnergy} 
E_{\eps}(u) = \int_M \frac{\eps}{2} |\nabla u|^2 + \frac{1}{\eps} W(u)
\end{equation}
where $W(u) = \frac{(1 - u^2)^2}{4}$ is a double-well potential. Critical points of this functional satisfy the Allen--Cahn equation
\begin{equation} \label{ACEquation}
\eps^2 \Delta_g u = W'(u)
\end{equation}
There is a well-known correspondence between zero sets of solutions to \eqref{ACEquation} and minimal surfaces. Modica and Mortola (\cite{modica1985gradient} \cite{mortola1977esempio}) showed that the Allen--Cahn energy functional $\Gamma$-converges to perimeter. In particular a large body of literature has emerged in recent years using min-max properties of the Allen--Cahn energy to construct min-max minimal hypersurfaces due to Guaraco, Gaspar, Chodosh, Mantoulidis, and many others (\cite{guaraco2018min} \cite{gaspar2018allen} \cite{chodosh2020minimal}). Under certain geometric constraints, Wang and Wei ([\cite{wang2019second}, Thm 1.1]) showed that the level sets of a sequence of stable solutions to \eqref{ACEquation}, $\{u_{\eps_i}\}$, converge to a minimal surface with good regularity.\nl \nl
\noindent One drawback of the Allen--Cahn energy \eqref{ACEnergy} is that a priori, it is difficult to control the topology of the zero set of minimizers, and thus the topology of the limiting minimal surfaces. This motivates us to define an Allen--Cahn energy functional \textit{on hypersurfaces themselves}. For $Y^{n-1} \subseteq M^n$, separating with $M = M^+ \sqcup_Y M^-$, let $u_{\eps}^{\pm}$ be the non-negative (positive) minimizers of equation \eqref{ACEnergy} on $M^{\pm}$ over all functions vanishing on $Y = \partial M^{\pm}$. For $Y$ satisfying some geometric constraints, $u_{\eps}^{\pm}$ are non-zero. We then define a \textbf{Balanced Energy} functional on hypersurfaces as 
\begin{align*}
\BE_{\eps}(Y) &:= E_{\eps}(u_{\eps}^+, M^+) + E_{\eps}(u_{\eps}^-, M^+)
\end{align*}
and investigate the variational properties of $\BE_{\eps}$ on the space of separating hypersurfaces. This includes explicit formula for the first and second variation, as well as connections between the Morse index and nullity of this function to the normal Allen--Cahn index and nullity on solutions. \nl \nl
In section \S \ref{FirstVariationSection}, we show that if $Y$ is critical for $\BE_{\eps}$, then $u_{\eps}^{\pm}$ can be glued together to form a solution to \eqref{ACEquation} on $M$. In particular, we have the existence of critical points for $\BE_{\eps}$ from the following theorem of Pacard and Ritore (See figure \ref{fig:aclevelsetconvergence} for a visualization):
\begin{restatable}[Pacard--Ritor\'e, Theorem 1.1]{thmm}{PRTheorem} \label{PRTheorem}
Assume that $(M, g)$ is an $n$-dimensional closed Riemannian manifold and $Y^{n-1} \subseteq M^n$ is a two-sided, nondegenerate minimal hypersurface. Then there exists $\eps_0 > 0$ such that $\forall \eps < \eps_0$, there exist solutions, $u_{\eps}$, to equation \eqref{ACEquation} such that $u_{\eps}$ converges to $+1$ (resp. -1) on compact subsets of $(\Omega^+)^o$ (resp. $(\Omega^c)^o$) and 
\[
E_{\eps}(u_{\eps}) \xrightarrow{\eps \to 0} \frac{1}{\sqrt{2}} A(Y)
\]
for $A(Y)$ the $n-1$-dimensional area of $Y$
\end{restatable}
\begin{figure}[h!]
\centering
\includegraphics[scale=0.3]{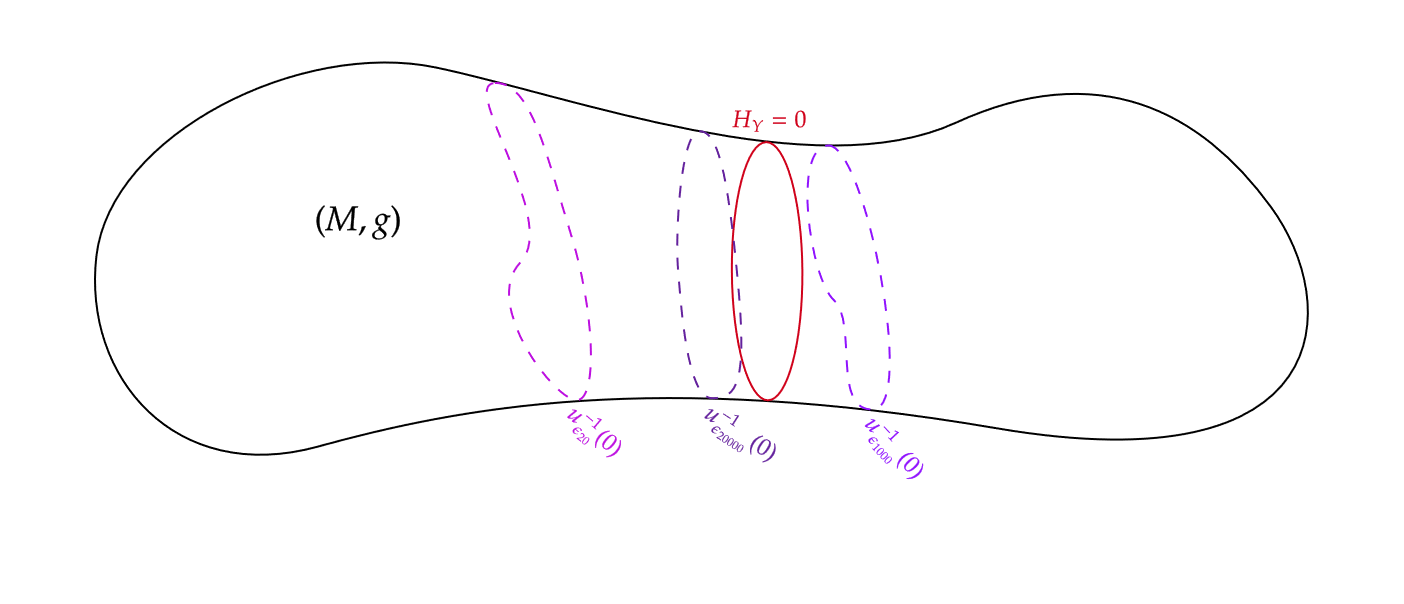}
\caption{Illustration of level set convergence to a minimal hypersurface}
\label{fig:aclevelsetconvergence}
\end{figure}
%
%
\noindent We are inspired by \cite{guaraco2019min}, Ex. 19, in which the author uses energy minimization and symmetry to glue solutions on $3$ subdomains of $S^n$ together. Energy minimization is also considered in the research of ``ground states" of the Allen--Cahn energy in \cite{caju2020ground}. We also draw inspiration from an analogous approach in the theory of extremal eigenvalue problems for the Laplacian, as presented in Henrot \cite{henrot2006extremum} on variations of the first eigenvalue. \nl \nl
We also mention a similar theory of variations of the Allen--Cahn energy using ``inner variations" of Allen--Cahn solutions as developed by Le \cite{le2011second} and Gaspar \cite{gaspar2020second}, among others. In this setting, variations of the form $u_{\eps} \circ \Phi_t$ are considered, for $\Phi_t$ a one parameter family of diffeomorphisms. Such variations are geometrically similar to variations of $\BE_{\eps}$, since they will move the nodal set of $u_{\eps}$. However, variations of $\BE_{\eps}$ move the nodal set \textit{and} simultaneously re-minimizes the Allen--Cahn energy on both sides, so the settings are different in both the setup and formulae. In particular, the first and second variations of $\BE_{\eps}$ can be framed as Dirichlet-to-Neumann problems and expanded asymptotically in $\eps$ to acquire precise estimates on subleading terms.
\subsection{Acknowledgements}
The authors would like to thank Otis Chodosh for suggesting this problem, as well as many insightful conversations over the course of this work. The first named author would also like to thank Rafe Mazzeo for his advice and interest in the problem.
\section{Results and Overview}
\noindent Given $Y^{n-1} \subseteq M$, sufficiently regular and separating so that 
\[
M = \Omega^+ \sqcup_Y \Omega^-
\]
we define $u^{\pm}_{Y, \eps}$ to be the positively (negatively) signed minimizers of \eqref{ACEnergy} on $\Omega^{\pm}$ with $0$ dirichlet condition on $Y$, and $U_{\eps, Y}$ to be the combined, pasted function defined on all of $M$. We then informally define the ``\textit{Balanced Energy}" on $Y$ as
\[
\BE_{\eps}(Y):= E_{\eps}(u^+_{Y,\eps}, \Omega^+) + E_{\eps}(u^-_{Y,\eps}, \Omega^-)
\]
(see \ref{BEDefinition} for full details). This is one way to associate an energy to a large class of hypersurfaces in $M$. We will often suppress the $\{Y, \eps \}$ subscripts and just write $u^{\pm}$. \nl \nl 
\noindent We note that $\BE_{\eps} \xrightarrow{\Gamma} \mathcal{P}$,the perimeter functional for sets of finite perimeter, which agrees with the area functional on smooth boundaries. This uses the standard $\Gamma$-convergence result for the Allen--Cahn energy due to Modica-Mortola \cite{mortola1977esempio}.
\begin{restatable}{thmm}
{GammaConSimple}\label{GammaConSimple}
    Let $\Omega$ be a set of finite perimeter. Then, we have
    \begin{enumerate}
        \item If $\Omega_{n} \xrightarrow{L^{1}} \Omega$ and $\partial \Omega_{n}$ is a smooth hypersurface, then $\textnormal{BE}_{\varepsilon}(\partial \Omega_{n}) \to \frac{1}{\sqrt{2}} \mathcal{P}( \Omega)$.
        \item There is a sequence of smooth hypersurfaces $\Omega_{n}$ such that $\Omega_{n} \xrightarrow{L^{1}} \Omega$ with 

        \[
        \limsup_{\varepsilon \to 0^{+}} \textnormal{BE}_{\varepsilon}(\partial \Omega_{n}) = \frac{1}{\sqrt{2}} \mathcal{P}( \Omega).
        \]
    \end{enumerate}
\end{restatable}
\noindent  For a more precise statement see \S \ref{GammaConvergenceSection}. We then compute the first variation of $\BE_{\eps}$ as a function of the hypersurface $Y$.
\begin{restatable}{thmm}{FirstVariationTheorem}\label{FirstVariationTheorem}
Let $\{Y_t\}$ a one-parameter family of hypersurfaces in $(M^n, g)$ with 
\[
Y_t = \exp_p (f(p) \nu(p))
\]
and $Y_0 = Y$. Then 
\begin{equation} \label{FirstVariationFormula}
\frac{d}{dt} \BE_{\eps}(Y_t) \Big|_{t = 0} = \BE_{\eps}'(f \nu) \Big|_Y = \frac{\eps}{2} \int_Y f [ (u_{\nu}^+)^2 - (u_{\nu}^-)^2]
\end{equation}
\end{restatable}
\noindent \rmk \; Equation \eqref{FirstVariationFormula} shows that critical points of \eqref{FirstVariationFormula} correspond to solutions of \eqref{ACEquation} on all of $M$, since $u_{\eps, Y}$ is now $C^1$ and hence smooth across its nodal set $u_{\eps, Y}^{-1}(0) = Y$. \nl \nl 
If we use basic results from \cite{marx2023dirichlet} we can show 
\begin{restatable}{corr}{FirstVarAnalogy}\label{FirstVarAnalogy}
For $Y$ a fixed $C^{3,\alpha}$ surface, there exists an $\eps_0$ such that for all $\eps < \eps_0$
\[
\BE_{\eps}'(f \nu) \Big|_Y = \frac{1}{2 \sqrt{2}} \langle H_Y + E, f \rangle 
\]
where $|E| \leq K \eps $ for $K$ dependent on $Y$.
\end{restatable}
\begin{Remark}
\cref{FirstVarAnalogy} can be thought of as an extension of $\Gamma$-convergence of the Allen--Cahn energy to the area functional, with the first variation converging to the mean curvature.
\end{Remark}
\noindent Having derived a notion of first variation for $\BE_{\eps}$, we can formulate the second variation. We first prove the following theorems, relating the AC-index and AC-nullity to the $\BE_{\eps}$-index and $\BE_{\eps}$-nullity.
\begin{restatable}{thmm}{EqualityOfIndicesTheorem}\label{EqualityOfIndicesTheorem}
Let $Y$ be any surface with bounded geometry \eqref{boundedGeometry} with $||A_Y||_{C^{k-2}} \leq C$ and $r_{\text{inj}}(Y) > \rho$. There exists an $\eps_0 = \eps_0(C, \rho)$ so that for all $\eps < \eps_0$, if $Y$ is a critical point for $\BE_{\eps}$, then
\begin{equation} \label{IndexEquation}
\textnormal{Ind}_{AC_{\eps}}(u_{Y, \eps}) = \textnormal{Ind}_{\BE_{\eps}}(Y)
\end{equation}
\end{restatable}
%
%
\begin{restatable}{thmm}{NullityBound}\label{NullityBound} 
Let $Y$ be any surface with bounded geometry \eqref{boundedGeometry} constants $C, \rho > 0$. There exists an $\eps_0 = \eps_0(C, \rho)$ so that for all $\eps < \eps_0$, if $Y$ is a critical point for $\BE_{\eps}$, then
\begin{equation} \label{NullityEquation}
\textnormal{Null}_{AC_{\eps}}(u_{Y,\eps}) = \textnormal{Null}_{\BE_{\eps}}(Y)
\end{equation}
\end{restatable}
\noindent Thus, for solutions which are energy minimizers, the index and nullity can be computed over the space of \textit{geometric} variations, i.e. variations of the null set. \nl \nl 
\noindent We can also explicitly state the second variation formula at a critical point. We consider solutions $\dot{u}_{\eps,Y}^{\pm} : \Omega^{\pm} \to \R $ to the linearized Allen--Cahn equation at $u$: 
\begin{align} \label{LinearizedSystem}
\begin{cases}
\eps^2 \Delta_g \dot{u}_{\eps, Y}^{\pm} &= W''(u_{\eps, Y}) \dot{u}_{\eps, Y}^{\pm} \text{ in } \Omega^{\pm},\\
\dot{u}_{\eps, Y}^{\pm} \Big|_Y &= -f u_{\nu}.
\end{cases}
\end{align}
We will often suppress the $\{\eps, Y\}$ subscripts out of convenience. The existence and uniqueness of a solution to the system \eqref{LinearizedSystem} will be a consequence of proposition \ref{NoLinearizedKernel}. With this, we can state the formula for the second derivative of $\BE_{\eps}$ at a critical point.
\begin{restatable}{thmm}{SecondVariationTheorem} \label{SecondVariationTheorem}
Let $\{Y_t\}$ a one-parameter family of hypersurfaces in $(M^n, g)$ with $Y_0 = Y$ as in theorem \ref{FirstVariationTheorem}. Further assume that $Y$ is a critical point. Then 
\begin{equation} \label{SecondVariationFormula}
\frac{d^2}{dt^2} \BE_{\eps}(Y_t) \Big|_{t = 0} = \eps \int_Y f u_{\nu} [ \dot{u}^{+}_{\nu} - \dot{u}^-_{\nu}]
\end{equation}
for $\dot{u}^{\pm}$ as in \eqref{LinearizedSystem}.
\end{restatable}
\noindent We then asymptotically expand $\dot{u}^{\pm}_{\nu}$ as a function of $\eps$ and prove that the equation \eqref{SecondVariationFormula} actually reflects a Jacobi operator.
\begin{restatable}{thmm}{SecondVarFormulaSmallEpsSobolevThm} \label{SecondVarFormulaSmallEpsSobolevThm}
Let $\{Y_t\}$ a one-parameter family of hypersurfaces in $(M^n, g)$ with $Y_0 = Y$ as in theorem \ref{FirstVariationTheorem}. Assume that $Y$ is a a $C^{2,\alpha}$ critical point with bounded geometry (\eqref{boundedGeometry}). Then 
\begin{align} \label{SecondVariationFormulaSobolevSmallEpsEqn}
\frac{d^2}{dt^2} \BE_{\eps}(Y_t) \Big|_{t = 0} &= \frac{\sqrt{2}}{3} \int_Y f u_{\nu}[-2 (\Delta_Y + \dot{H}_0) ](f) + E(f) \\ \nonumber
&= \frac{2 \sqrt{2}}{3}\int_Y  |\nabla f|^2 - \dot{H}_Y f^2 + E \\ \nonumber 
&= \frac{2 \sqrt{2}}{3} [D^2A(f) + E(f)]
\end{align}
where
\[
|E(f)| \leq K \eps^{1/2} ||f||_{H^1}^2
\]
for $K = K(C, \rho)$ as in \eqref{boundedGeometry}.
\end{restatable}
\begin{Remark}
Equation \eqref{SecondVariationFormulaSobolevSmallEpsEqn} is a further extension of $\Gamma$-convergence of the Allen--Cahn energy to the area functional, but now with the second variation converging to the second variation of area. This contrasts with the results of Gaspar \cite{gaspar2020second} and Le \cite{le2011second}, where there is an extra term that obstructs the convergence of the second inner variation of the Allen--Cahn energy to the second variation of the area functional. Our result corresponds to the situation where the extra term in the second inner variation vanishes, and theorem \ref{SecondVarFormulaSmallEpsSobolevThm} additionally shows a precise control of the size of the error in $\eps$.
\end{Remark}
\noindent To prove \cref{SecondVarFormulaSmallEpsSobolevThm} we establish the following theorem on the linearized operator $\eps^2 \Delta_g - W''(u)$ (the content of which is contained in proposition \ref{NoLinearizedKernel}):
\begin{restatable}{thmm}{LinearizedOperatorInverseThm}\label{LinearizedOperatorInverseThm}
The operator $\Leu = \eps^2 \Delta_g - W''(u): H^1_{\eps, 0}(M^{\pm}) \to H^{-1}_{\eps}(M^{\pm})$ has a bounded inverse.
\end{restatable}
\noindent In particular, \cref{SecondVarFormulaSmallEpsSobolevThm} tells us that the second variation of $\BE_{\eps}$ for a critical point is $O(\eps)$ away from the second variation of area for a minimal surface. This leads to the following applications using the geometry of $Y$
\begin{restatable}{corr}{RicciStability}
\label{RicciStability}
Fix $C, \rho, c_0 > 0$ and $0 \leq \gamma < 1/2$. Suppose $Y$ is a $C^{2,\alpha}$ hypersurface with bounded geometry coefficients \eqref{boundedGeometry}, $C, \rho$. There exists an $\eps_0 = \eps_0(C, \rho, c_0, \gamma)$ such that if $Y$ is a $\BE_{\eps}$ critical point for $\eps < \eps_0$ and 
\begin{equation} \label{StabilityIntegralCondition}
\frac{1}{\text{Vol}(Y) \eps^{1/2 - \gamma}}\int_Y \Ric_{g}(\nu, \nu) + ||A_Y||^2 \geq c_0
\end{equation}
then $Y$ is $\BE_{\eps}$-unstable.
\end{restatable}
\begin{remark}
Corollary \ref{RicciStability} should be compared to the classic Allen--Cahn stability result which says that when $M$ is ricci positive, no stable solutions. When $Y$ has bounded geometry and $\eps$ is sufficiently small, the integral condition \eqref{StabilityIntegralCondition} is weaker than $M$ being ricci positive.
\end{remark}
\noindent Similarly, we produce a parallel result to that of minimal surfaces, proved by Fischer-Colbrie--Schoen:
\begin{restatable}{thmm}{FischerColbrieMimic}
\label{FischerColbrieMimic}
Let $M^3$ be a complete oriented $3$-manifold of non-negative scalar curvature. Let $Y^2$ be a $C^{2,\alpha}$, oriented complete, \textbf{compact}, stable critical point for $\BE_{\eps}$ in $M$ . Let $S$ denote the scalar curvature of $M$ at any point. Then $Y$ is conformally equivalent to a $(S^2, g_{\text{round}})$ or $Y$ is ``almost" a totally geodesic flat torus $T^2$, i.e.
\[
||A_Y||_{L^2(Y)}^2 \leq K \eps^{1/2}
\]
for some $K > 0$ independent of $\eps$ for $K$ sufficiently small. If $S > 0$ on $M$, then $Y$ is conformally equivalent to $(S^2, g_{\text{round}})$. If both:
\begin{align*} 
S &\geq \kappa_0 > 0 \\
\text{Vol}(Y) & \geq v_0 > 0
\end{align*}
both independent of $\eps$, then $Y$ must be conformally equivalent to $S^2$.
\end{restatable}
\noindent We then produce results on the Morse index of Allen--Cahn level sets. In the highly symmetric case, we have:
\begin{restatable}{corr}{MinimalIndexEqualsACIndex} \label{MinimalIndexEqualsACIndex}
Suppose $\{u_{\eps}\}$ is a sequence of Allen--Cahn solutions with $Y = u_{\eps}^{-1}(0)$ for all $\eps$. If $Y$ is minimal and $u_{\eps}$ are each critical points of \eqref{BEEnergies}, then there exists an $\eps_0 > 0$ sufficiently small, such that for all $\eps < \eps_0$
\begin{equation}
\Ind_{AC_{\eps}}(u_{\eps})  = \Ind_{\BE_{\eps}}(u_{\eps})  \geq \Ind_{\text{Min}}(Y)
\end{equation}
%
\end{restatable}
%
\noindent Separately, we use the $\BE$ frame work and variation formula to compute the Morse index of solutions to \eqref{ACEquation} on $S^1$. We first define $u_{\eps, 2p}$ to be the unique solution vanishing at $Y_{2p} := \{i/(2p)\}_{i = 0}^{2p-1}$ on $S^1 \cong [0,1]$ (see \S \ref{S1MorseIndex})
\begin{restatable}{prop}{CircleSecondVarProp} \label{CircleSecondVarProp}
Let $p > 0$, and $u_{\eps, 2p}$ the solution vanishing on $Y_{2p}$. Consider a variation of the form
\[
a_i(t):= \frac{i}{2p} + f(i/(2p)) t
\]
for $f: \{0, 1, \dots, 2p -1\} \to \R$. Then
\begin{equation} \label{S1SecondVarRevamped}
BE''(f) \Big|_{Y_{2p}} = \eps  c^2 v(\eps)   \sum_{i = 0}^{2p-1} [f(i/(2p)) - f((i+1)/2p)]^2
\end{equation}
for $v(\eps) < 0$ for all $\eps > 0$
\end{restatable}
\noindent Equation \eqref{S1SecondVarRevamped} allows us to classify the Morse index of all solutions to \eqref{ACEquation} on $S^1$ as $\eps \to 0$. 
\begin{restatable}{thmm}{CircleMorseIndexTheorem}
\label{S1MorseIndexTheorem}
Fix $p > 0$. There exists an $\eps_p$ such that $\forall \eps < \eps_p$, the function, $u_{\eps, 2p}$, has Allen--Cahn Morse index $2p-1$ and nullity $1$. The nullity is given by rotations of $S^1$ and every other non-trivial variation produces a strictly negative variation.
\end{restatable}
\noindent We note that formulas \eqref{FirstVariationFormula}, \eqref{SecondVariationFormula} \eqref{SecondVariationFormulaSobolevSmallEpsEqn} are local to $Y$, thought of as the nodal set of a solution, $u_{\eps}$, to \eqref{ACEquation}. The first and second variation for \eqref{ACEnergy} are integrals over all of $M$, despite most of the geometry of $u_{\eps}$ being contained in a neighborhood of $Y$. The locality to $\BE_{\eps}$ can be an advantage that gives more information about  Allen--Cahn solutions, as in \cref{S1MorseIndexTheorem}. In this way, we feel that the direct connection of $\BE_{\eps}$ to the surface gives a new, geometric way to investigate a smaller class of Allen--Cahn functions and variations. \nl \nl
Finally, we note that while our visuals and notation seem to be adapted for $Y$ connected (and $M^{\pm}$ each connected), our results apply to $Y$ with finitely many components. As long as $Y$ is separating, one can make sense of the decomposition $M = M^+ \sqcup_Y M^-$ and define energy minimizers on each component of $M^{\pm}$ which are positively (resp. negatively) signed on the interior of each component. Non-zero minimizers are guaranteed with a bound on $\lambda_1$ on each component of $M^{\pm}$, as in theorem \ref{BO}. This can also be achieved by bounding the distance between components of $Y$ from below. Assuming this, we can reproduce our expansions of $u_{\eps}(s,t)$, $\dot{u}(s,t)$ locally on each component of $Y$ and prove the same theorems.
%
%
%
\subsection{Paper Organization}
The paper is organized as follows: 
\begin{itemize}
\item Section \S \ref{PrelimSection} sets up the machinery of $\BE_{\eps}$.

\item In section \S \ref{GammaConvergenceSection}, we prove that the $\BE_{\eps}$ functional $\Gamma$-converges to the area functional up to a scaling factor. This closely mimics the original proof by Modica--Mortola in \cite{mortola1977esempio}

\item In section \S \ref{NonExistenceSection} we show that \textit{absolute minimizers} of $\BE_{\eps}$ do \textit{not} exist on compact manifold $(M^n, g)$ for $\eps$ sufficiently small. In fact $\inf_Y \BE_{\eps}(Y)$ over all hypersurfaces is $0$ when $n \geq 2$. On $S^1$, $\inf_Y \BE_{\eps}(Y)$ has a non-trivial lower bound.

\item In \S \ref{FirstVariationSection}, we prove theorem \ref{FirstVariationTheorem}. In \S \ref{SecondVariationSection}, we prove theorem \ref{SecondVariationTheorem} similarly

\item In \S \ref{EqualityOfIndices}, we prove that if $Y \leftrightarrow u_{\eps, Y}$ is a critical point of $\BE_{\eps} \leftrightarrow E_{\eps}(M)$, then the $\BE_{\eps}$ Morse index and nullity of $Y$ equals the Allen--Cahn Morse index and nullity of $u_{\eps, Y}$. The equalities, along with an explicit expansion of the dirichlet-to-neumann operator for the linearized Allen--Cahn system, allows us to compute the Morse index and nullity of all solutions to \eqref{ACEquation} on $S^1$

\item In \S \ref{SecondVariationSmallEpsSection}, \S \ref{SecondVarProofSection}, we compute an asymptotic expansion of \eqref{SecondVariationFormula} in terms of $\eps$. In particular, this shows that the second variation is $O(\eps)$-close to the second variation of area for minimal surfaces
%
\item In \S \ref{ApplicationsOfSecondVariation}, we give some applications of equation \eqref{SecondVariationFormula}, in particular showing that no stable $\BE_{\eps}$-hypersurfaces exist when there is a Ricci lower bound, or when there are bounds on $||A_Y||_{L^2(Y)}$ and $\text{Vol}(Y)$. We also prove an analogous version of Fischer-Colbrie-Schoen's theorem on stable minimal surfaces in $3$-dimensional PSC manifolds, our \cref{FischerColbrieMimic}, for $\BE_{\eps}$-stable in $3$-dimensional PSC manifolds. Finally, we show that for $Y$ minimal with $\{u_{\eps}\}$ a sequence of Allen--Cahn solutions and $u_{\eps}^{-1}(0) = Y$ that the index of $Y$ as a minimal hypersurface equals the index Allen--Cahn index of $u_{\eps}$ for all $\eps$ sufficiently small.

\item In section \S \ref{S1MorseIndex}, we prove \cref{S1MorseIndexTheorem} and compute the Morse index of \textit{all} Allen--Cahn solutions on $S^1$ for all $\eps$ sufficiently small. This culminates in equation \eqref{S1SecondVarRevamped} which explicitly reflects the geometry of the variations. In particular, this shows a rigidity result, i.e. all variations lead to non-positive second variation, with null directions given by rotations.


\end{itemize}

\section{Notation and Preliminaries.} \label{PrelimSection}
\noindent Throughout we will denote by $(M, g)$ an arbitrary closed Riemannian manifold. 
We will write $\Omega \subset M$ to denote an arbitrary set of finite perimeter which has smooth boundary, $\partial \Omega$, and we will write $E \subset M$ for sets which may not be smooth, and we identify sets which differ except up to a set of measure $0$

\begin{itemize}
\item If $\Omega \subset M$ is a smooth domain, then $\lambda_{1}(\Omega)$ denotes the first eigenvalue of the Laplacian $\Delta_{g}$ on $\Omega$ with dirichlet conditions on $\partial \Omega$

\item $E_{\varepsilon}(u; E)$ denotes the $\varepsilon$-Allen--Cahn energy, \eqref{ACEnergy}, of the function $u \in W^{1,2}(M)$ restricted to the Borel set $E$.

\item $\mathcal{C}(M)$ denotes the collection of Borel sets of finite perimeter.

\item $\mathcal{C}^{\infty}(M)$ are those sets of finite perimeter with smooth boundary.

\item $\mathcal{C}_{\varepsilon}^{\infty}(M) \coloneqq \{ \Omega \in \mathcal{C}^{\infty}(M) \big\vert \min\{\lambda_{1}(\Omega)^{-\frac{1}{2}}, \lambda_{1}(\Omega^{c})^{-\frac{1}{2}}\} > \varepsilon \}$. Note that $\bigcup_{\varepsilon > 0} \mathcal{C}^{\infty}_{\varepsilon}(M) = \mathcal{C}^{\infty}(M)$.

\item We denote by $|\Omega|$ the measure of a set with the Borel measure induced by the volume form on $(M, g)$.

\item $|\Omega \Delta \Omega'|$ denotes the measure of the symmetric difference of two sets, i.e. $|\Omega \Delta \Omega'| = |(\Omega \backslash \Omega') \cup (\Omega' \backslash \Omega)|$.

\item If $\Omega$ is a smooth domain, $\nu$ will denote the outer unit normal to $\partial \Omega$.

\item Let $\{E_{j}\} \subset \mathcal{C}(M)$. We say that $E_{j} \underset{L^{1}}{\to} E$ if $|E_{j} \Delta E| \to 0$ as $j \to \infty$.

\item Let $\{\Omega_{j}\} \subset \mathcal{C}^{\infty}(M)$ we will say that $\Omega_{j} \underset{C^{\infty}}{\to} \Omega$ if $\partial \Omega_{j} \to \partial \Omega$ as graphs in $C^{\infty}$. 

\item $\BE_{\eps}(Y)$ denotes the \textbf{Balanced Energy} (see definition \ref{BEDefinition}) of a hypersurface $Y^{n-1} \subseteq M^n$ when well defined
\end{itemize}
\noindent Basic properties of sets of finite perimeter are developed in \cite{maggi2012sets}. In particular it will be important to note that for any $E \in \mathcal{C}(M)$ there is a sequence $\{\Omega_{j}\} \subset \mathcal{C}^{\infty}(M)$ so that $\Omega_{j} \underset{L^{1}}{\to} E$. \nl \nl
\noindent We will denote by $Y \subset M$ a $2$-sided, $C^{2}$ (potentially more regular when noted), closed hypersurface. $Y$ divides $M$ into two (possibly not connected) sets, which we will denote by $\Omega^{+} \sqcup_Y \Omega^{-}$ or $M^+ \sqcup M^-$. When needed, we'll assume that $Y$ has ``bounded geometry" in the following sense: there exists a $C = C(k), \rho > 0$ (both independent of $\eps$) such that if $Y$ is a $C^{k,\alpha}$ surface, then 
\begin{align} \label{boundedGeometry}
||A_Y||_{C^{k-2,\alpha}} &\leq C \\
r_{\text{inj}}(Y) & \geq \rho
\end{align}
Here, $r_{\text{inj}}(Y)$ denotes the radius of injectivity for the exponential map. The lower bound guarantees a uniform tubular neighborhood in which fermi coordinates are defined. For convenience we also define 
\[
\dot{H}_Y = |A_Y|^2 + \Ric_g(\nu,\nu)
\]
where $\Ric_g$ is the ricci tensor of the ambient metric on $M$.
\nl \nl    
\noindent Let $g(t) = \tanh(t/\sqrt{2})$ denote the heteroclinic solution to \eqref{ACEquation} on $\R$. We define $\dot{g}(t) = g'(t)$ and so on for higher derivatives. Similarly, let $W(u) = \frac{(1 - u^2)^2}{4}$. We will denote $w(t), \rho(t), \kappa(t), \tau(t)$ the following functions on $[0, \infty)$ which solve the below ODEs
\begin{align} \label{ODEEquations}
\ddot{w} - W''(g) w &= \dot{g} \\ \nonumber
\ddot{\rho} - W''(g) \rho &= \dot{w} \\ \nonumber
\ddot{\tau} - W''(g) \tau &= t \dot{g} \\ \nonumber
\ddot{\kappa} - W''(g) \kappa &= g w 
\end{align}
all subject to the condition that 
\[
f: [0, \infty) \to \R, \qquad f(0) = 0, \qquad \lim_{t \to \infty} f(t) = 0
\]
The functions, $w(t), \rho(t), \kappa(t), \tau(t)$, are then uniquely defined (see \cite{marx2023dirichlet}, Appendix 7.6). \nl \nl 
We also define the following constants
%
\begin{align} \label{constants}
\sigma_0 &:= \int_0^{\infty} \dot{g}^2 = \frac{\sqrt{2}}{3} \\
\sigma &:= \dot{g}(0) =\sqrt{2}^{-1} 
\end{align}
Let 
\begin{align} \label{rescaledNorms}
||f||_{C^{k,\alpha}_{\eps}} &:= \sum_{i = 0}^k \eps^i ||D^i f||_{\infty} + \eps^{k + \alpha} [D^k f]_{\alpha} \\
||f||_{H^k_{\eps}}^2 &:= \sum_{i = 0}^k \eps^{2i} \eps ||D^i f||_{L^2}^2 \\
||f||_{H^{-k}_{\eps}} &= \sup_{g \in H^k_{\eps,0}} \frac{1}{||g||_{H^k_{\eps}}} \int_M f g
\end{align}
be the rescaled H\"older and Sobolev norms, respectively. Here $H^k_{\eps,0}$ refers to functions in $H^k_{\eps}$ with $0$ trace on the boundary (in context, this will be $M^{\pm}$ so that $\partial M^{\pm} = Y$). In the same vein, we notate a rescaled function as 
\begin{align*}
f&: \R^+ \to \R  \\
f_{\eps}(t) &:= f(t/\eps)
\end{align*}
We also define
\[
\bg(t):= \chi_{-\omega \ln(\eps)}(t) g(t) + [1 - \chi_{-\omega\ln(\eps)}(t)]
\]
where $\chi_{-\omega \ln(\eps)}$ goes from $1$ to $0$ on $[-\omega \ln(\eps), -2 \omega \ln(\eps)]$ where $\omega > 5$. Then we note that 
\begin{equation} \label{ApproxHeteroclinic}
[\partial_t^2 - W''(\bg(t))] \bg(t) = E \qquad ||E||_{H^k} \leq C(k) \eps^{\omega}
\end{equation}
Note that we'll use the same notation of $\bw, \btau, \overline{\kappa}$ to denote cutoffs of $w, \tau, \kappa$ but to $0$ since those functions decay exponentially naturally. E.g.
\begin{align*}
\bw(t)&:= [1 - \chi_{-\omega\ln(\eps)}(t)]g(t) \\
[\partial_t^2 - W''(\bg(t))] \bw(t) &= E \qquad ||E||_{H^k} \leq C(k) \eps^{\omega}
\end{align*}
%
%
We recall the following result from \cite{marx2023dirichlet}, Theorem 1.5 (and implicitly \cite{mantoulidis2022variational}): For $Y$ a $C^{4,\alpha}$ hypersurface with bounded geometry (equation \ref{boundedGeometry}), we have that $u_{\eps}^+: M^+ \to \R$ expands as 
\begin{align} \label{uEpsExpansion4}
u_{\eps}^+(s,t) &= \bg_{\eps}(t) + \eps H_Y(s) \bw_{\eps}(t) + \phi \\ \nonumber
||\phi||_{C^{2,\alpha}_{\eps}(M)}& \leq K \eps^2
\end{align}
where $s$ denotes Fermi coordinates on $Y$ and $t$ is the signed distance to $Y$. If $Y$ is $C^{5,\alpha}$, the expansion becomes
\begin{align} \label{uEpsExpansion5}
u_{\eps}^+(s,t) &= \bg_{\eps}(t) + \eps H_Y(s) \bw_{\eps}(t) + \eps^2 \left[\dot{H}_Y(s) \btau_{\eps}(t) + H_Y^2(s) \left(\brho_{\eps}(t) + \frac{1}{2} \overline{\kappa}_{\eps}(t)\right) \right] + \phi\\ \nonumber
||\phi||_{C^{2,\alpha}_{\eps}(M)}& \leq K \eps^3
\end{align}
\noindent Finally for $u: M \to \R$, we define $\Leu := \eps^2 \Delta_g - W''(u)$. 
%
%
\subsection{$\text{BE}_{\varepsilon}$ energies} \label{BEEnergies}

In \cite{brezis1986remarks}, Brezis and Oswald showed the following result.

\begin{theorem}[Brezis-Oswald \cite{brezis1986remarks}]
\label{BO}
Suppose $\Omega \subset M$ is an open domain with smooth boundary. If $\varepsilon < \lambda_{1}(\Omega)^{-\frac{1}{2}}$, then

\begin{equation}
\begin{cases}
\Delta_{g} u = \frac{W'(u)}{\varepsilon^{2}} &\text{ in } \Omega,\\
u > 0 & \text{ in } \Omega,\\
u = 0 & \text{ on } \partial \Omega,
\end{cases}
\end{equation}
\noindent has a unique solution.
\end{theorem}
\noindent Therefore, for any $\Omega \in \mathcal{C}^{\infty}_{\varepsilon}(M)$, one has the existence of a unique positive solution $u^{+}_{\Omega, \varepsilon}$ to the $\varepsilon$-Allen--Cahn equation and a unique negative solution $u^{-}_{\Omega^{c}, \varepsilon}$ on $\Omega^{c}$. We define
\begin{equation} \label{BrokenPhaseTransition}
u_{\Omega, \varepsilon} = \begin{cases}
u_{\Omega, \varepsilon}^{+} & \text{ on } \overline{\Omega},\\
u_{\Omega^{c}, \varepsilon}^{-} & \text{ on } \Omega^{c}.
\end{cases}
\end{equation}
\noindent We call $u_{\Omega, \varepsilon}$ a \textbf{broken }$\varepsilon$-\textbf{phase transition} (or simply a broken phase transition if reference to $\varepsilon$ is unimportant). We may also reference $u_{\Omega, \varepsilon}$ as $u_{Y, \eps}$ when $Y = \partial \Omega$ plays a more important role. It follows by construction that $u_{\Omega, \varepsilon} \in C^{0}(M) \cap W^{1, 2}(M)$. In fact, it solves the Allen--Cahn equation outside of $\partial \Omega$. It can be seen that elliptic regularity implies that $u_{\Omega, \varepsilon}$ is a solution to the $\varepsilon$-Allen--Cahn equation if and only if $\nabla_{\nu} u^{-}_{\Omega, \varepsilon} = \nabla_{\nu} u^{+}_{\Omega, \varepsilon}$.\par

\begin{definition} 
If $\Omega \in \mathcal{C}^{\infty}_{\varepsilon}(M)$ we define the balanced energy of $\Omega$ as
\begin{equation} \label{BEDefinition}
\text{BE}_{\varepsilon}(\Omega) \coloneqq E_{\varepsilon}(u_{\Omega, \varepsilon})
\end{equation}
We will often have cause to write $\text{BE}_{\varepsilon}(\partial \Omega)$ to denote the same thing. We will also write $\text{BE}_{\varepsilon}^{+}(\Omega) \coloneqq E_{\varepsilon}(u^{+}_{\Omega, \varepsilon})$ and similarly $\text{BE}_{\varepsilon}^{-}(\Omega) = E_{\varepsilon}(u^{-}_{\Omega, \varepsilon})$.
\end{definition}

\section{$\Gamma$-convergence of $\text{BE}_{\varepsilon}$ to the Perimeter Functional.} \label{GammaConvergenceSection}
\noindent The goal of this section is to show that the $\text{BE}_{\varepsilon}$ functional defined on $\mathcal{C}^{\infty}_{\varepsilon}(M)$ converges to the perimeter functional, $\mathcal{P}$, on $\mathcal{C}(M)$ as $\epsilon \to 0$. 

\begin{theorem}[$\Gamma$-convergence] \label{Gamma} 

Let $\Omega \in \mathcal{C}(M)$ be an arbitrary Caccioppoli set.

\begin{enumerate}
\item (Pre-compactness of $\textnormal{BE}_{\varepsilon}$-bounded sets): Let $\{\Omega_{\varepsilon_{k}}\}$ be a sequence of Cacciopoli sets in $\mathcal{C}^{\infty}_{\varepsilon_{k}}(M)$ such that $\textnormal{BE}_{\epsilon_{k}}(\Omega_{\epsilon_{k}})$ is bounded and $\varepsilon_{k} \to 0$ as $k \to \infty$. Then there is some $\Omega'$ so that $\Omega_{\epsilon_{k}} \xrightarrow[L^{1}]{} \Omega'$.

\item (lim-inf inequality): If $\Omega_{\varepsilon_{k}} \underset{L^{1}}{\to} \Omega$, then

\[
\liminf_{\varepsilon_{k} \to 0} \textnormal{BE}_{\varepsilon_{k}}(\Omega_{\varepsilon_{k}}) \geq \frac{1}{\sqrt{2}} \mathcal{P}(\Omega).
\]

\item (Recovery sequence): For every $\Omega \in \mathcal{C}(M)$ there is a sequence $\Omega_{\varepsilon_{k}}(M) \in \mathcal{C}^{\infty}_{\varepsilon_{k}}(M)$ so that $\Omega_{\epsilon_{k}} \underset{L^{1}}{\to} \Omega$ and

\[
\limsup_{\epsilon_{k} \to 0} \textnormal{BE}_{\epsilon_{k}}(\Omega_{\epsilon_{k}}) \leq \frac{1}{\sqrt{2}} P(\Omega).
\]
\end{enumerate}
\end{theorem}

\noindent Strictly speaking what we show is the following: for any set of finite perimeter $\Omega \in \mathcal{C}(M)$, set $\overline{\BE_{\eps}}(\Omega) = + \infty$ if $\Omega \not\in \mathcal{C}^{\infty}_{\eps}(M)$ and $\overline{\BE_{\eps}}(\Omega) = \BE_{\eps}(\Omega)$ otherwise. Then the functional $\overline{\BE_{\eps}}$, which is now defined on all of $\mathcal{C}(M)$, $\Gamma$-converges to the perimeter functional under the usual topology induced by convergence of sets of finite perimeter (identified up to a set of measure 0). This does not affect the argument otherwise, so we will proceed without distinguishing $\overline{\BE_{\eps}}$ and $\BE_{\eps}$. To prove this, we rely fundamentally on the $\Gamma$-convergence theory of Modica-Mortola \cite{mortola1977esempio}. We recall these below:

\begin{proposition}[BV-compactness] \label{BV-compactness}
Let $u_{k} \in W^{1,2}(M)$ be a sequence of functions so that  $E_{\epsilon_{k}}(u_{k})$ is uniformly bounded in $k$. Then up to subsequence 

\[
u_{k} \underset{L^{1}}{\to} u \in \textnormal{BV}(M; \{-1, 1\})
\]
\end{proposition}

\begin{theorem}[$\Gamma$-convergence for the Allen--Cahn energies] \label{GammaConvergeTheorem}

Let $\Omega \in \mathcal{C}(M)$ be an arbitrary Caccioppoli set.

\begin{enumerate}
\item (liminf inequality): Suppose $u_{k} \in W^{1, 2}(M)$ and $E_{\varepsilon_{k}}(u_{k})$ are uniformly bounded. If $u_{k} \underset{L^{1}}{\to} u \in \text{BV}(M; \{-1, 1\})$ then
\[
\liminf_{\epsilon_{k} \to 0} E_{\epsilon_{k}} u_{k} \geq \sigma_{0} \mathcal{P}(\{u = 1\}).
\]
\item (limsup inequality): There is a sequence $u_{k} \in W^{1, 2}(M)$ so that
\[
u_{k} \underset{L^{1}}{\to} \chi_{\Omega} - \chi_{\Omega^{c}},
\]
satisfying,
\[
\limsup_{\epsilon_{k}}E_{\epsilon_{k}}(u_{k}) \leq \sigma \mathcal{P}(\Omega).
\]
\noindent Moreover, if we assume further that if $\Omega \in \mathcal{C}^{\infty}(M)$, then we can construct the $u_{k}$ so that they satisfy $\{u_{k} = 0\} = \partial \Omega$.
\end{enumerate}
\end{theorem}
\noindent To prove theorem \ref{Gamma} we first need the following basic lemma.

\begin{lemma} \label{PosPartConvergence}
If $f_{n} \underset{L^{1}}{\to} f$. Let $f_n^+ = \max(f_n, 0)$, then $f_{n}^{+} \underset{L^{1}}{\to} f^{+}$.
\end{lemma}

\begin{proof}
Write $f_{n} = f_{n}^{+} - f_{n}^{-}$ and $f = f^{+} - f^{-}$. Then if $f_{n}^{+} > 0$ we have two cases.

\begin{enumerate}
\item $f^{+} > 0$, in which case

\[
|f_{n} - f| = |f_{n}^{+} - f^{+}|.
\]

\item $f^{-} > 0$, in which case

\[
|f_{n} - f| = f_{n}^{+} + f^{-} \geq f_{n}^{+} = |f_{n}^{+} - f^{+}|,
\]

since $f^{+} = 0$ whenever $f^{-} > 0$.
\end{enumerate}
\noindent In either case we find 
\[
|f_{n} - f| \geq |f_{n}^{+} - f^{+}|
\]
Note that the above also holds if $f_n^+ = 0$, so
\[
\int |f_{n}^{+} - f^{+}| \leq \int |f_{n} - f| \to 0.
\]
\end{proof}

\begin{proof}[Proof of Theorem \ref{Gamma}]

The result follows from the $\Gamma$-convergence result of Modica-Mortola.

\begin{enumerate}

\item (Pre-compactness): 
Let $\Omega_{k}$ be a bounded sequence for $\text{BE}_{\epsilon}$ in $\mathcal{C}^{\infty}(M)$. That is to say there is a constant $C > 0$
\[
\sup_{k}\text{BE}_{\epsilon_{k}}(\Omega_{k}) \leq C
\]   
\noindent We wish to construct a Caccioppoli set $\Omega \in \mathcal{C}(M)$ so that $|\Omega_{k} \Delta \Omega| \to 0$ up to possibly taking a subsequence in $k$. For each $k$, Let $u_{k} \in W^{1, 2}(M)$ denote the broken $\varepsilon_{k}$ phase-transition associated to $\Omega_{k}$. That is to say,
\[
u_{k} = \begin{cases}
u_{\Omega_{k}}^{+} \text{ in } \Omega_{k},\\
u_{\Omega_{k}}^{-} \text{ in } \Omega_{k}^{c}.
\end{cases}
\]
\noindent  Note that by taking a subsequence without relabeling, we may assume $\varepsilon_{k}$ is sufficiently small that this broken phase-transition because $\varepsilon_{k} \to 0$ as $k \to \infty$. \nl \nl 
\noindent It follows from the definition of $\Omega_{k}$ that
\[
E_{\epsilon_{k}}(u_{k}) = \text{BE}_{\epsilon_{k}}(\Omega_{k}).
\]
\noindent Boundedness of $\Omega_{k}$ with respect to the Balanced energies $\text{BE}_{\epsilon}$ implies that the Allen--Cahn energies for $u_{k}$ controlled uniformly in $k$. Therefore theorem \ref{BV-compactness} applies to the sequence $\{u_{k}\}$ and we find a function $u \in \text{BV}(M ; \{-1, 1\})$ so that up to taking a subsequence.
\[
u_{k} \underset{L^{1}}{\to} u.
\]
\noindent Set $\Omega' = \{u > 0\}$. We wish to show that $|\Omega_{k} \Delta \Omega'| \to 0$, so we will estimate the measures of $\Omega' \backslash \Omega_{k}$ and $\Omega_{k} \backslash \Omega$. Observe that $\Omega_{k} = \{u_{k} > 0\}$ by definition of the functions $u_{k}$ and that $u = \chi_{\Omega'} - \chi_{(\Omega')^{c}}$. \nl \nl 
\noindent By Lemma \ref{PosPartConvergence}, Theorem \ref{GammaConvergeTheorem}(2), and the fact that $u = \chi_{\Omega'} - \chi_{(\Omega')^{c}}$ it follows that $u_{k}^{+} \underset{L^{1}}{\to} u^{+}$, i.e.
\[
\int_{M} |u_{k}^{+} - \chi_{\Omega'}|  \to 0.
\]   
\noindent So moreover by restricting the above to $\Omega_{k}^{c}$,
\[
\int_{\Omega_{k}^{c}} \chi_{\Omega'} = |\Omega_{k}^{c} \cap \Omega'| = |\Omega' \backslash \Omega_{k}| \to 0.
\]
\noindent To show that $|\Omega_{k} \backslash \Omega'| \to 0$ as $k \to \infty$, we repeat the same with $u_k^-$ and $\chi_{(\Omega')^{c}}$:  
\[
\int_{M} |u_{k}^{-} - \chi_{(\Omega')^c}|  \to 0.
\]  
by the same lemma applied to the negative part. In particular
\[
\int_{M} |u_{k}^{-} - \chi_{(\Omega')^c}| \geq \int_{\Omega_k} |u_k^- - \chi_{(\Omega')^c}| = \int_{\Omega_k} |\chi_{(\Omega')^c}| = |\Omega_k \cap (\Omega')^c| = |\Omega_k \backslash \Omega'|
\]
\noindent It follows that $|\Omega_{k} \backslash \Omega'| \to 0$.
\item (liminf inequality): Suppose we have a sequence $\Omega_{k} \in \mathcal{C}^{\infty}_{\eps_{k}}(M)$ of Caccioppoli sets in $M$ which converge to $\Omega \in \mathcal{C}(M)$. Let $u_{k}$ be the piecewise Brezis-Oswald solutions as before. Then because
\[
E_{\epsilon_{k}}(u_{k}) = \text{BE}_{\epsilon_{k}}(\Omega_{k}) = O(1)
\]
\noindent we can apply BV-compactness to the sequence $u_{k}$ to find a function $u$ as before. By the exact same arguments as the previous it follow that
\[
|\Omega \Delta \{u > 0\}| = 0.
\]
\noindent By the Modica-Mortola liminf inequality
\[
\liminf_{k \to \infty} \text{BE}_{\epsilon_{k}}(\Omega_{k}) = \liminf_{k\to \infty} E_{\epsilon_{k}}(u_{k}) \geq \sigma_{0} \mathcal{P}(\Omega).
\]
\item (limsup inequality): Suppose that $\Omega \in \mathcal{C}^{\infty}(M)$. Let $\Omega_{k} = \Omega$ for each k. Then $u_{k}$ is a sequence of functions in $C^{0}(M) \cap W^{1,2}(M)$ which vanish precisely along $\partial \Omega$. 
For the sake of comparison we shall construct functions $g_{k}$ whose Allen--Cahn energy is strictly greater than the $u_{k}$ and which also vanish along $\partial \Omega_{k}$. Our task will be accomplished by comparing their energies using the variational characterization of Brezis-Oswald solutions. \nl\nl
\noindent Let $\rho > 0$ be a constant taken sufficiently small so that on the tubular neighborhood $B_{\rho_{0}}(\partial \Omega)$ the signed distance function $\text{d}: B_{\rho_{0}}(\partial \Omega) \to \mathbb{R}$ is smooth. Let $\chi: \mathbb{R} \to \mathbb{R}$ denote a smooth, even cutoff function satisfying $0 \leq \chi \leq 1$, with $\chi \equiv 1$ on $(-\rho_{0}/4, \rho_{0}/4)$ and $\chi \equiv 0$ outside of $(-\rho_{0}/2, \rho_{0}/2)$. Define $\tilde{g}_{k} : \mathbb{R} \to \mathbb{R}$ to be
\[
\tilde{g}_{k}(t) = g_{\varepsilon_{k}}(t) \chi(t) + (1 - \chi(t))\text{sgn}(t).
\]
\noindent And now we define $g_{k}: M \to \mathbb{R}$ by
\[
    g_{k}(x) = \tilde{g}_{k}(\text{d}(x)),
\]
\noindent extending smoothly to $1$ and $-1$ outside of $B_{\rho_{0}(M)}$. Now we show that the functions $g_{k}$ satisfy the limsup inequality. Note that by the coarea formula and the fact that $|\nabla \text{d}(x)| \equiv 1$ where $g_{k}$ is nonzero
\begin{align*}
E_{\varepsilon}(g_{k}) &= \int_{M} \frac{\varepsilon_{k}}{2}|\nabla g_{k}|^{2} + \frac{W(g_{k})}{\varepsilon_{k}}\\
&= \int_{B_{\rho_{0}(M)}} \frac{\varepsilon_{k}}{2}|\nabla g_{k}|^{2} + \frac{W(g_{k})}{\varepsilon_{k}}\\
&= \int_{|r| \leq \rho_{0}/4} \left(\frac{\varepsilon_{k}|\tilde{g}_{k}'(r)|^{2}}{2} + \frac{W(\tilde{g}_{k}(r))}{\varepsilon_{k}} \right) \mathcal{H}^{n - 1}(\{\text{d}(x) = r\}) dr\\
&\quad + \int_{M \backslash B_{\rho_{0}/4}(M)}  \frac{\varepsilon_{k}}{2}|\nabla g_{k}|^{2} + \frac{W(g_{k})}{\varepsilon_{k}}\\
&= I_{1} + I_{2}.
\end{align*}
\noindent $I_{2} = o(1)$ as $k \to \infty$ by the exponential decay of the heteroclinic to the values $-1, 1$ at infinity. Noting that, for the heteroclinic $g_{\varepsilon_{k}}$, the energy density 
\[
\left( \frac{\varepsilon_{k} |g'_{\varepsilon_{k}}|^{2}}{2} + \frac{W(g_{\varepsilon_{k}}}{\varepsilon_{k}} \right)dr \to \delta_{0},
\]
\noindent where $\delta_{0}$ denotes the Dirac mass centered at $0$, and since $\tilde{g}_{k}$ agrees with the heteroclinic in the region where $|r| \leq \rho_{0}/4$, it follows that
\begin{align*}
I_{1} &= \int_{|r| \leq \rho_{0}/4} \left( \frac{\varepsilon_{k} |\tilde{g}'_{\varepsilon_{k}}|^{2}}{2} + \frac{W(\tilde{g}_{\varepsilon_{k}}}{\varepsilon_{k}} \right) \mathcal{H}^{n - 1}(\{\text{d}(x) = r\})dr \to \mathcal{H}^{n - 1}(\partial \Omega),
\end{align*}
\noindent as $k \to \infty$. \nl \nl
\noindent To conclude, we note that
\[
E_{\epsilon_{k}}(u_{k}) = E_{\epsilon_{k}}(u_{k}; \Omega) + E_{\epsilon_{k}}(u_{k}; \Omega^{c}) \leq E_{\epsilon_{k}}(g_{k}; \Omega) + E_{\epsilon_{k}}(g_{k};\Omega^{c}) = E_{\epsilon_{k}}(g_{k}).
\]
\noindent The result now follows because
\[
\limsup_{k \to \infty} \textnormal{BE}_{\epsilon_{k}}(\Omega) = \limsup_{k\to \infty} E_{\epsilon_{k}}(u_{k}) \leq \limsup_{k \to \infty} E_{\epsilon_{k}}(g_{k}) \leq \sigma_{0} \mathcal{P}(\Omega).
\]
\noindent If we now suppose that $\Omega \in \mathcal{C}(M)$ is arbitrary, and we make no assumption on regularity, by theorem 13.8 in \cite{maggi2012sets} we can find a sequence $\Omega^{j}$ of smooth Caccioppoli sets so that $\Omega^{j} \to \Omega$ in the flat topology which moreover satisfies
\[
\mathcal{P}(\Omega^{j}) = \mathcal{P}(\Omega) + o(1), \text{ as } j \to \infty.
\]
\noindent Then if we set $\Omega_{k} = \Omega^{k}$ we get
\[
\limsup_{k \to \infty} \textnormal{BE}_{\epsilon_{k}} (\Omega_{k}) \leq \sigma_{0} \mathcal{P}(\Omega^{k}) + o(1) = \sigma_{0} \mathcal{P}(\Omega) + o(1) \text{ as } k \to \infty.
 \]
\end{enumerate}

\end{proof}

\section{First Variation.} \label{FirstVariationSection}
\noindent We consider a smooth family of hypersurfaces $\{Y_t\}$ and corresponding domains, $\{\Omega_t\}$ such that $\Omega =: \Omega_0$ and $\partial \Omega = Y$. Let $\{u_t\}$ be minimizers of \eqref{ACEnergy} on $\Omega_t$ for some $\eps > 0$. Doing the analogous for $\{\Omega_t^c\}$ and $\{u_t^c = u_t\}$, we are interested in the quantity
\begin{equation} \label{BETimeDefinition}
\BE(t) = \int_{\Omega_t} \eps \frac{|\nabla u_t|^2}{2} + \frac{1}{\eps}W(u_t) + \int_{\Omega_t^c} \frac{|\nabla u_t^-|^2}{2} + \frac{1}{\eps}W(u_t^c)
\end{equation}
as a function of $t$. We will use the labels $u_t = u_t^+$ and $u_t^c = u_t^-$ interchangeably throughout. Similarly, $\nu = \nu^+$ and $\nu^c = \nu^- := -\nu$. The goal is to prove \cref{FirstVariationTheorem}
\FirstVariationTheorem*
\subsection{Set up}
We decompose $M = \Omega_t \sqcup_{Y_t} \Omega_t^c$. Let $u_t: \Omega_t \to \R$ be the minimizers of \eqref{ACEnergy}. We can define
\[
u: N_{\rho}(\Omega) \times (-\delta, \delta) \to \R
\]
where 
\[
u(p,t) = u_t(p)
\]
Note that the right hand side is not defined for all $p \in N_{\rho}(\Omega)$, but since $u_t$ satisfies a dirichlet condition on $\partial \Omega_t$, and if we assume some uniform regularity of $\partial \Omega_t$, then we can define tubular neighborhoods for each $\partial \Omega_t$ and define odd extensions for $u_t$ on $N_{\rho_t}(\Omega_t)$. By bounded geometry assumptions  \eqref{boundedGeometry}, we have that for $t < t_0$ sufficiently small there exists a $\rho_0 = \rho_0(t_0) > 0$ so that
\[
\bigcap_{0 \leq t < t_0} N_{\rho_t}(\Omega_t) \supseteq N_{\rho_0}(\Omega)
\]
so that $u(\cdot,t)$ can be defined $N_{\rho_0}(\Omega)$ for all $t$. For example, this occurs if 
\begin{align*}
F_t&: \partial \Omega \to \partial \Omega_t  \\
F_t(p) &= \exp_p(f(p) t \nu(p)) \\
||f||_{C^{2,\alpha}} & \leq K
\end{align*}
%
%
Then we have that 
\begin{align*}
\eps^2 \Delta_g u(p,t) &= W'(u(p,t)) \\
p_t \in \partial \Omega_t, & \quad  u(p_t, t) \equiv 0 \\
\eps^2 \Delta_g \dot{u}(p,0) &= W''(u(p,0)) \dot{u}(p,0) \\
\dot{u}(p,0) &= - f u_{\nu}(p,0)
\end{align*}
where $\nu$ is the normal pointing to the interior of $\Omega$. Note that $u$ is differentiable on the interior by pulling back $Y_t \to Y$ and expanding this diffeomorphism locally. See \S \ref{SmoothDependenceDomains} for details.
\subsection{First variation computation}
We define 
\[
G_{\eps}(S, f, X) := \int_S \frac{\eps}{2} |X|^2 + \frac{1}{\eps} W(f)
\]
So that 
\[
E_{\eps}(u_t, \Omega_t) = G_{\eps}(\Omega_t, u_t, \nabla u_t)
\]
which frames 
\begin{align*}
\frac{d}{dt} E_{\eps}(u_t, \Omega_t) &= \frac{d}{dt} G_{\eps}(\Omega_t, u_t, \nabla u_t) \\
&= \int_{\partial \Omega} -f \left[\frac{\eps}{2} |\nabla u|^2 + \frac{1}{\eps} W(u)\right] \\
&+ \int_{\Omega} \eps \langle \nabla u, \nabla \dot{u} \rangle + \frac{1}{\eps} W'(u) \dot{u} \\
&= A + B
\end{align*}
here the sign choice of $-f$ as opposed to $f$ is because $\nu$ points inwards to $\Omega$. We evaluate 
\[
A = \int_{\partial \Omega} -f \left[ \frac{\eps}{2} u_{\nu}^2 + \frac{1}{4 \eps} \right]
\]
And for $B$, we have 
\begin{align*}
B &= \int_{\Omega} \eps \langle \nabla u, \nabla \dot{u} \rangle + \frac{1}{\eps} W'(u) \dot{u} \\
&= \int_{\partial \Omega} \eps u_{\nu^c} \dot{u} \\
&= \int_{\partial \Omega} \eps u_{\nu}^2 f
\end{align*}
having noted that $\dot{u} \Big|_{\partial \Omega} = -f u_{\nu}$ and $u_{\nu^c} = - u_{\nu}$. Thus 
\begin{align*}
\frac{d}{dt} E_{\eps}(u_t, \Omega_t) &= A + B \\
&= \int_{\partial \Omega} \left[\frac{\eps}{2} u_{\nu}^2 - \frac{1}{4 \eps}\right] f
\end{align*}
Repeating this on $\Omega^c$, noting that $f \to -f$ with the choice of $\nu^c$ as the inward normal to $\Omega^c$, we get that 
\[
\frac{d}{dt} \BE(t) = \frac{\eps}{2} \int_{\partial \Omega}\left[ u_{\nu}^2 - (u^c_{\nu})^2 \right] f 
\]
ending the proof of theorem \ref{FirstVariationTheorem}. \qed
\begin{Remark}
While we've formulated $\BE_{\eps}$ for $(M, g)$ compact, equation \eqref{FirstVariationFormula} is defined for $Y = \partial \Omega$ compact and $M$ non-compact. Moreover, while we assume that $f \in C^{2,\alpha}(Y)$, we can formally extend the first variation formula to $f \in L^2(Y)$ by the final formula.
\end{Remark}
\noindent We note that critical points must have $u_{\nu} = u_{\nu^c}$ everywhere on the boundary. From equation \eqref{uEpsExpansion5}, we have 
\begin{align} \label{MatchingNormal}
u_{\nu} &= \frac{1}{\eps \sqrt{2}} + \dot{w}(0) H_Y + \eps[ \dot{\tau}(0) \dot{H}_Y + H_Y^2 \alpha] + O(\eps^2) \\ \nonumber
u_{\nu} - u^c_{\nu} &= 2 \dot{w}(0) H_Y + \eps H_Y^2 \alpha + O(\eps^2)
\end{align}
which assuming $H_Y = O(1)$ and computing $\dot{w}(0)$ gives \cref{FirstVarAnalogy}
\FirstVarAnalogy*
\noindent We further note:
\begin{restatable}{corr}{FirstVarMeanCurvature} \label{FirstVarMeanCurvature}
Let $Y$ be a critical point of $\BE_{\eps}$. If $Y \in C^{3,\alpha}$, then $||H_Y||_{C^0} \leq K \eps$ and $[H_Y]_{\alpha} \leq K \eps^{1 - \alpha}$. If $Y$ is $C^{4,\alpha}$, then $||H_Y||_{C^{\alpha}} = O(\eps^{2-\alpha})$
\end{restatable} 
\noindent Alternatively, we see that if $Y$ is a critical point of $\BE_{\eps}$, then $u_{\Omega, \varepsilon}$ is $C^1$ across $\partial \Omega$ and hence a smooth solution to \eqref{ACEquation} on all of $M$, as remarked in \S \ref{BEEnergies}. Again such solutions exist near minimal surfaces by \cref{PRTheorem}, and these solutions are exactly the energy minimizers with $0$ dirichlet condition on $Y$ by theorem \ref{BO}. This gives the following
\begin{restatable}{corr}{CritPointExistence}
Suppose $Y^{n-1} \subseteq M^n$ a nondegenerate minimal hypersurface. Then there exist critical points for $\BE_{\eps}$, $Y_{\eps}$, with $Y_{\eps} \to Y$ converging graphically as $\eps \to 0$
\end{restatable}
%
%
\noindent In general, the Pacard--Ritore solutions are the prime examples of critical points of $\BE_{\eps}$ on $(M, g)$ generic, though the set of critical points of $\BE_{\eps}$ may be larger. 
\section{Second Variation.} \label{SecondVariationSection}
\noindent Consider the linearized system, $\dot{u}$, as in \eqref{LinearizedSystem} to a family of $\{u_t\}$. These linearizations are defined on both $\Omega$ and $\Omega^c$ and are labelled $\dot{u}, \dot{u}^c$ respectively. We'll use $\dot{u}$ interchangeably with $\dot{u}^+$ and $\dot{u}^c$ interchangeably with $\dot{u}^-$. \nl \nl
In this section we aim to show the following \cref{SecondVariationTheorem}
\SecondVariationTheorem*
\noindent \Pf The first variation formula actually holds for any $t$:
\[
\frac{d}{dt} \BE(t) \Big|_{t = s} = \frac{\eps}{2} \int_{\partial \Omega_s} \left[u_{\nu_s}(p,s)^2 - u^c_{\nu_s}(p,s)^2 \right] f(p,s) dA
\]
where $\nu_t$ is an inward normal to $\partial \Omega_t$ and $f(p,t)$ is the map such that 
\begin{align*}
F_{s,r}&: \partial \Omega_s \xrightarrow{\cong} \partial \Omega_r  \\
p \in \partial \Omega_s, \quad F_{s,r}(p) &= \exp_{p, \partial \Omega_s }(f(p,s) r \nu_s(p))
\end{align*}
and each $f(p,s)$ is defined on $\partial \Omega_s$, or equivalently $f: \partial \Omega \times (-\delta, \delta)$ by pulling back each $\partial \Omega_s$ to $\partial \Omega$ by a diffeomorphism. We take the first variation formula and pull it back by $F_s$: 
\begin{equation} \label{FirstVarSParameterFamily}
\frac{d}{dt} \BE(t) \Big|_{t = s} = \frac{\eps}{2} \int_{\partial \Omega} F_s^*\left(\left[u_{\nu_s}(\cdot,s)^2 - u^c_{\nu_s}(\cdot,s)^2 \right]\right)(p) F_s^*(f(\cdot,s))(p) F_s^*(dA)
\end{equation}
for $p \in \partial \Omega$ a parameter. With $t = 0$ a critical point, the derivative of equation \eqref{FirstVarSParameterFamily} is given by 
\begin{align*}
\frac{d^2}{dt^2} \BE(t) \Big|_{t = 0} &= \frac{\eps}{2} \int_{\partial \Omega} \frac{d}{ds} F_s^*\left(\left[u_{\nu_s}(\cdot,s)^2 - u^c_{\nu_s}(\cdot,s)^2 \right]\right)(p) \Big|_{s = 0} \; f(p) dA
\end{align*}
We can write the differentiated part as
\begin{align*}
F_s^*\left(\left[u_{\nu_s}(\cdot,s)^2 - u^c_{\nu_s}(\cdot,s)^2 \right]\right)(p) &= u_{\nu_s}(F_s(p),s)^2 - u^c_{\nu_s}(F_s(p),s)^2 \\
& =\langle (\nabla u)(F_p(s), s), \nu_s \rangle^2 - \langle (\nabla u^c)(F_p(s), s), \nu_s \rangle^2  
\end{align*}
differentiating the individual terms, we get
\begin{align*}
\frac{d}{ds} \langle (\nabla u)(F_p(s), s), \nu_s \rangle^2 \Big|_{s = 0} &= 2 u_{\nu} \left[ \dot{\nu}(u) + \dot{u}_{\nu} + f u_{\nu \nu}\right] 
\end{align*}
And we conclude 
\[
\frac{d^2}{dt^2} \BE(t) \Big|_{t = 0} = \eps \int_{\partial \Omega} f u_{\nu} [\dot{u}_{\nu} - \dot{u}^c_{\nu}]
\]
where $u_{\nu} = u_{\nu}(p,0)$ for $p \in \partial \Omega$. In the first line, we've noted that 
\begin{align*}
\frac{d}{ds}\langle (\nabla u)(F_p(s), s), \nu_s \rangle &= \frac{d}{ds} \nu_s^i(\partial_i u)(F_p(s), s) \\
&= \dot{\nu}^i (\partial_i u)(p,0) + \nu^i [\nabla (\partial_i u) \cdot \dot{F}_p] + \nu^i (\partial_i \dot{u})(p,0) \\
&= \dot{\nu}(u) + \nu^i [\nabla (\partial_i u) \cdot (f \nu)] + \nu(\dot{u}) 
\end{align*}
if we choose fermi coordinates off of $\partial \Omega$ so that $\nu = \partial_t$, the middle expression becomes 
\[
f u_{tt} = f u_{\nu \nu}
\]
But note that $u_{\nu \nu} = u^c_{\nu \nu}$, since at a critical point, we have a solution to Allen--Cahn which is necessarily smooth. Thus $f u_{\nu \nu} - f u^c_{\nu \nu} = 0$. Similarly
\[
\dot{\nu}(u) - \dot{\nu}(u^c)\Big|_{\partial \Omega} = 0
\]
Since $u$ is smooth across the boundary. This gives our final expression
\begin{align*}
\frac{d^2}{dt^2} \BE(t) \Big|_{t = 0} &= \eps \int_{\partial \Omega} f u_{\nu} [\dot{u}_{\nu} - \dot{u}^c_{\nu}]
\end{align*}
this concludes the proof of theorem \ref{SecondVariationTheorem}. \qed \nl \nl      
\noindent \rmk \; Again, we note that the second variation formula can also be formally defined when $M$ is non-compact or when $\partial \Omega $ non-compact, assuming that $f \in C_c^{\infty}(\partial \Omega)$.
\section{$\BE_{\eps}$ Indices and Nullity} \label{EqualityOfIndices}
\noindent Recall that $Y$ is a critical point for $\BE_{\eps}$, then $u_{Y, \eps}$ is a solution to Allen--Cahn and an energy minimizer on each component of $M \backslash Y$. In this section we prove the following two theorems
\EqualityOfIndicesTheorem*
\NullityBound*
\noindent To prove the theorems, we first establish two propositions
\begin{proposition} \label{GraphicalityProp}
Let $Y$ a critical point for $\BE_{\eps}$ with bounded geometry \eqref{boundedGeometry} constants $C, \rho$. If $u_r = u_{\eps, Y} + r \phi$ for $ \phi \in C^1(M)$, then there exists an $r_0 = r_0(C, \rho) > 0$ such that for all $r < r_0$, $Y_r := u_r^{-1}(0)$ is graphical over $Y$.
\end{proposition}
\noindent \Pf We argue as follows: we want to solve 
\[
0 = u_{\eps, Y} + r \phi
\]
Without loss of generality, assume $||\phi||_{C^0} = 1$. Let $s$ be a coordinate for $Y$ and $t$ denote the signed distance from $Y$. Parameterize each $u_r$ as 
\[
u_r = u(r,s,t) = u_{\eps,Y}(s,t) + r \phi(s,t)
\] 
We want to solve for 
\[
F(s,r) \;\; \st \;\; u(r,s, t = F(s,r)) = 0
\]
with $F(s,0) = 0$. First note that 
\[
u(r = 0, s = s_0, t = 0) = 0
\]
for any $s_0$. Further note that 
\[
\partial_t u \Big|_{r = 0, s = s_0, t = 0} = \partial_t u_{\eps, Y}(s, t) \Big|_{s = s_0, t = 0} = \frac{1}{\eps \sqrt{2}} + O(1) \neq 0
\]
by \eqref{uEpsExpansion4}. By the implicit function theorem, we get $\delta > 0$ and $F(s,r)$ which is differentiable in $s$ and $r$ such that 
\[
\text{dist}_Y(s,s_0) + |r| < \delta \implies u(r, s, t = F(s,r)) = 0
\]
and this is the unique zero in this neighborhood. Note that $F$ is defined locally, but we can repeat this argument for every $s_0 \in Y$. Taking a finite cover of $Y$, we find $\delta_0 > 0$ and $F$, a function, such that
\[
\exists F: Y \times (-\delta_0, \delta_0) \; \st\; u(r,s,F(s,r)) = 0
\]
and $F$ is continuous in both $s$ and $r$. This finishes the proof of graphicality. Note that for large $r$, the level set of $u_r$ may have more than one component. However, using that $||\phi||_{C^0(M)} = 1$ and taking $r < \delta_0^*$ for a potentially smaller $\delta_0^* \leq \delta_0$, we can guarantee that $u_r$ only has one component and it is given by $\{(s, t = F(s,r)\}$. This follows because $u_{\eps}(s,t)$ is exponentially converging to $ \pm 1$ away from $t = 0$. \qed \nl \nl
\noindent Let $f \in C^2(Y)$. For $Y_r = \exp_Y (f r \nu)$, let $w_{r,\eps}(f)$ be the corresponding $\BE_{\eps}$-phase transition for this family of surfaces. We can define 
\[
\dot{w}_{\eps} = \partial_r w_{r,\eps}(f) \Big|_{r = 0}
\]
If $Y$ has bounded geometry \eqref{boundedGeometry}, then for $\eps < \eps_0(C, \rho)$, $\dot{w}_{\eps}$ solves \eqref{LinearizedSystem} with the corresponding boundary condition.
\begin{lemma}
Let $Y$ have bounded geometry \eqref{boundedGeometry} with constants $C, \rho > 0$. Define $Y_{1,r} = \exp_Y (f_1 r \nu)$, $Y_{2,r} = \exp_Y (f_1 r \nu)$, and $w_{r,\eps}(f_1)$, $w_{r,\eps}(f_2)$ the corresponding $\BE$-phase transitions. Then there exists an $\eps_0 = \eps_0(C, \rho)$ such that for all $\eps < \eps_0$
\[
\alpha \dot{w}_{\eps}(f_1) + \beta\dot{w}_{\eps}(f_2) = \dot{w}_{\eps}(\alpha f_1 + \beta f_2)
\]
\end{lemma}
\noindent \Pf This follows by the form of \eqref{LinearizedSystem} as $\alpha \dot{w}_{\eps}(f_1) + \beta \dot{w}_{\eps}(f_2)$  and $\dot{w}_{\eps}(\alpha f_1 + \beta f_2)$ both satisfy the linearized PDE and have the same boundary conditions for $\eps$ sufficiently small. The solution is unique by proposition \ref{NoLinearizedKernel} (or \cite{marx2023dirichlet}, Thm 1.4).\qed 
\subsection{Proof of \cref{EqualityOfIndicesTheorem}}
\noindent Note that any variation of surfaces is a variation in $H^1(M)$, and so 
\[
\frac{d^2}{dr^2} \BE_{\eps}(Y_r) \Big|_{r = 0} = \frac{d^2}{dr^2} E_{\eps}(w_{r,\eps}) \Big|_{r = 0}
\]
where $w_{r,\eps}$ the $\BE_{\eps}$-phase transition corresponding to $Y_r$. This
means that any $\BE_{\eps}$ variation with negative second derivative is an AC variation with negative second derivative. To see that 
\begin{equation} \label{IndexLowerBound}
\text{Ind}_{\BE_{\eps}}(Y) \leq \text{Ind}_{AC_{\eps}}(u_{\eps, Y})
\end{equation}
we need to show that the map from hypersurface variations of $Y$ (to first order)is injective into the space of $H^1(M)$ functions which are first order variations of the form $u_{r, \eps} = u_{\eps, Y} + r \phi$ for $\phi: M \to \R$. However, given $f_1, f_2: Y \to \R$ two variations of $Y$, the corresponding variations of $Y$ are $Y_{r,i} = \exp_Y(f_i \cdot r)$, and the corresponding variations of $u_{\eps, Y}$ are 
\begin{align*}
\dot{w}_{\eps}(f_i) &:= \frac{d}{dr} u_{\eps, Y_{r,i}} \Big|_{r = 0}
\end{align*}
for $i=1,2$. Since we assume $f_1 \neq f_2$ and we know that $\dot{w}(f_i) \Big|_{Y} = - f_i u_{\nu}$ by the dirichlet condition in equation \eqref{LinearizedSystem}, we see that $\dot{w}_{\eps}(f_1) \neq \dot{w}_{\eps}(f_2)$. This establishes injectivity and hence the lower bound on index, equation \eqref{IndexLowerBound}.\nl \nl
To see equality of indices, let $u_{r,\eps} = u_{\eps, Y} + r \phi$ a $C^1(M)$ family of functions. From proposition \ref{GraphicalityProp}, we can write
\[
u_r = w_{r,\eps} + \psi_{r,\eps}
\]
where $w_{r,\eps}$ is the $\BE_{\eps}$-phase transition on $Y_r = u_{r,\eps}^{-1}(0)$, and $\psi_{r,\eps}\Big|_{Y_r} = 0 $ for all $r$. Since $w_{r,\eps}$ is $C^1$ with respect to $r$ by \S \ref{SmoothDependenceDomains}, we readily see that $\psi_{r,\eps}$ is $C^1$ in $r$. Thus we can define 
\begin{align*}
\dot{w}_{\eps} &:= \partial_r w_{r,\eps}\Big|_{r = 0}: M \to \R \\
\dot{\psi}_{\eps} &:= \partial_r \psi_{r, \eps}\Big|_{r = 0}: M \to \R
\end{align*}
For the purposes of computing the second variation at a critical point of $\BE_{\eps}$, it suffices to consider the family of functions given by:
\[
u_{r, \eps} = u_{\eps, Y} + r [\dot{w}_{\eps} + \dot{\psi}_{\eps}] 
\]
We then compute
\begin{align*}
\frac{d^2}{dr^2} E_{\eps}(u_{r,\eps}) \Big|_{r = 0} &:= Q(\dot{w}_{\eps} + \dot{\psi}_{\eps}, \dot{w}_{\eps} + \dot{\psi}_{\eps}) \\
&= \int_M \eps |\nabla (\dot{w}_{\eps} + \dot{\psi}_{\eps})|^2 + \frac{W''(u_{\eps, Y})}{\eps} (\dot{w}_{\eps} + \dot{\psi}_{\eps})^2 \\
&= Q(\dot{w}_{\eps}, \dot{w}_{\eps}) + Q(\dot{\psi}_{\eps}, \dot{\psi}_{\eps}) + 2 \int_M \eps \langle \nabla \dot{w}_{\eps}, \nabla \dot{\psi}_{\eps} \rangle + \frac{W''(u_{\eps, Y})}{\eps} \dot{w}_{\eps} \dot{\psi}_{\eps}
\end{align*}
We note that 
\[
p \in Y, \;\; p(r) \in Y_r \rightarrow \psi_{r,\eps}(p(r)) = 0 \implies \dot{\psi}_{\eps}(p) = - \nabla \psi_{0,\eps} \cdot \dot{p} = 0
\]
for any path $p(r) \in Y_r$ with $p(0) = p$. The last line follows since $\psi_{0,\eps} |_Y \equiv 0$. Then because $\dot{w}_{\eps}$ satisfies \eqref{LinearizedSystem}, we have
\begin{align*}
\int_M \eps \langle \nabla \dot{w}_{\eps}, \nabla \dot{\psi}_{\eps} \rangle + \frac{W''(u_{\eps, Y})}{\eps} \dot{w}_{\eps} \dot{\psi}_{\eps} &= \int_Y \eps \dot{\psi}_{\eps} [\dot{w}_{\eps, \nu^+} - \dot{w}_{\eps, \nu^-}] - \int_M  \eps^{-1} [ \eps^2 \Delta_g \dot{w}_{\eps} - W''(u_{\eps, Y}) \dot{w}_{\eps} ] \dot{\psi}_{\eps} \\
&= 0
\end{align*}
%
\noindent Thus 
\begin{align*}
\frac{d^2}{dr^2} E_{\eps}(u_{\eps, r}) \Big|_{r = 0} &= Q(\dot{w}_{\eps}, \dot{w}_{\eps}) + Q(\dot{\psi}_{\eps}, \dot{\psi}_{\eps}) \\
&= \frac{d^2}{dr^2} \BE_{\eps}(Y_r) \Big|_{r = 0} + Q(\dot{\psi}_{\eps}, \dot{\psi}_{\eps})
\end{align*}
We now note that 
\begin{align} \label{EnergyIncrease}
E_{\eps}(u_{\eps, Y} + r \dot{\psi}) & = Q(\dot{\psi}_{\eps}, \dot{\psi}_{\eps}) \\
&\geq E_{\eps}(u_{\eps, Y}) \\ \nonumber
\implies Q(\dot{\psi}_{\eps}, \dot{\psi}_{\eps}) & \geq 0
\end{align}
precisely because $u_{\eps, Y}$ consists of piecewise energy minimizers vanishing on $Y$ and $u_{\eps, Y} + r \dot{\psi}_{\eps}$ vanishes on $Y$. This means that 
\begin{align*}
\frac{d^2}{dr^2} E_{\eps}(u_{\eps, r}) \Big|_{r = 0} -  \frac{d^2}{dr^2} \BE_{\eps}(Y_r) \Big|_{r = 0} &= Q(\dot{\psi}_{\eps}, \dot{\psi}_{\eps}) \\
& \geq 0
\end{align*}
In particular
\[
\frac{d^2}{dr^2} E_{\eps}(u_r) \Big|_{r = 0} < 0 \implies \frac{d^2}{dr^2} \BE_{\eps}(Y_r) \Big|_{r = 0} < 0
\]
Which shows that any variation of $u_r$ contributing AC-index gives rise to a variation of $Y$ that contributes BE-index. In particular, if $\dot{u}_1, \dots, \dot{u}_k \in C^1(M)$ span a vector space of dimension $k$, equal to the AC-Morse index of $u_{\eps, Y}$, i.e. 
\[
\forall \{c^i\}_{i = 1}^k, \qquad \frac{d^2}{dr^2} E_{\eps}(u_{\eps, Y} + r c^i \dot{u}_i) = Q(c^i \dot{u}_i, c^j \dot{u}_j) \leq 0
\]
with equality only when $c^i = 0$ for all $i$. Then we have 
\begin{align*}
\forall i, \quad \dot{u}_i &= \dot{w}_i + \dot{\psi}_i \\
\frac{d^2}{dr^2} E_{\eps}(u_{\eps, Y} + r c^i \dot{u}_i) &= Q(c^i \dot{w}_i + c^i \dot{\psi}_i, c^i \dot{q}_i + c^i \dot{\psi}_i) \\
&= Q( c^i \dot{w}_i, c^j \dot{w}_j) + Q(c^i \dot{\psi}_i, c^j \dot{\psi}_j) \\
&= \frac{d^2}{dr^2} \BE_{\eps}(Y_r) \Big|_{r = 0} + Q(c^i \dot{\psi}_i, c^j \dot{\psi}_j) 
\end{align*} 
Here, $\dot{w}_i$ are again the first order variations of the $\BE_{\eps}$ phase transitions on $Y^i_r = (u_{\eps, Y} + r \dot{u}_i)^{-1}(0)$ and $\dot{\psi}_i$ are the remainders vanishing on $Y$. Note that we've adapted Einstein notation and noted that $c^i \dot{u}_i$ solves the linearized Allen--Cahn equation. Moreover $Y_r$ is the family of surfaces coming from the combined perturbation
\[
f = c^i \dot{u}_i |_Y \in C^1(Y)
\] 
Since $c^i \dot{\psi}_i|_Y = 0$, we again see that 
\begin{align} \label{RemainderTermPos}
Q(c^i \dot{\psi}_i, c^j \dot{\psi}_j) &\geq 0 \\ \nonumber
\frac{d^2}{dr^2} E_{\eps}(u + r c^i \dot{u}_i) < 0 & \implies \frac{d^2}{dr^2} \BE_{\eps}(Y_{r c^i \dot{w}_i}) \Big|_{r = 0} < 0
\end{align}
Note that equation \eqref{RemainderTermPos} first tells us that the $\{w_i\}$ are linearly independent, for if not, we could arrange for $\{c^i\}$ so that 
\begin{align*}
c^i \dot{u}_i &= c^i \dot{w}_i + c^i \dot{\psi}_i \\
&= c^i \dot{\psi}_i 
\end{align*}
but 
\[
\frac{d^2}{dr^2} E_{\eps}(u_{\eps, Y} + r c^i \dot{u}_i)  = Q(c^i \dot{\psi}_i, c^j \dot{\psi}_j) \geq 0
\]
a contradiction to the assumption that this second variation is stricly negative unless $c^i \equiv 0$. This verifies that 
\[
\frac{d^2}{dr^2} \BE_{\eps}(Y_{r c^i \dot{w}_i}) \Big|_{r = 0} < 0
\]
for all $\{c^i\}$ non-zero sequences. This gives us the other bound,
\[
\text{Ind}_{\BE_{\eps}}(Y) \geq \text{Ind}_{AC_{\eps}}(u_{\eps, Y})
\]
ending the proof. \qed

\subsection{Proof of \cref{NullityBound}}
\subsubsection{Set up}
From here on, we will suppress the $\eps$ subscripts for notational convenience, since the proof applies for any $\eps > 0$ sufficiently small. To prove this theorem, we first define the corresponding symmetric, bilinear form to the second variation. For any $f \in C^2(Y)$, let $\dot{w}(f)$ denote the $C^0$ and piecewise $C^1$ function:
\begin{align*}
\dot{w}(f)(p) &:= \begin{cases}
	\dot{w}^+(f)(p) & p \in M^+ \\
	\dot{w}^-(f)(p) & p \in M^- \\
	-f(p) u_{\nu}(p) & p \in Y
	\end{cases} \\
\eps^2 \Delta_g \dot{w}^{\pm}(f) &= W''(u) \dot{w}^{\pm}(f) \quad \text{on } M^{\pm}
\end{align*}
Note that if $Y$ is a critical point of $\BE_{\eps}$, then $u_{\nu}$ is well defined on $Y$. We now define 
\begin{align*}
f, h &\in C^2(Y) \\
Q_{\BE_{\eps}}(f,h) &= -\int_Y \dot{w}(f) \eps[\dot{w}^+(h)_{\nu} - \dot{w}^+(h)_{\nu}] \\
&= \eps \int_Y f u_{\nu} [\dot{w}^+(h)_{\nu} - \dot{w}^+(h)_{\nu}]
\end{align*}
So that the second variation of $\BE$ along $f$ is $Q_{\BE_{\eps}}(f,f)$. To see that is bilinear, note that equation \eqref{LinearizedSystem} is linear in its dirichlet condition. To see symmetry, we write 
\begin{align*}
Q_{\BE_{\eps}}(f,h) &= -\int_Y \dot{w}(f) \eps[\dot{w}^+(h)_{\nu} - \dot{w}^+(h)_{\nu}] \\
&= -\int_Y \dot{w}^+(f) \eps \dot{w}^+(h)_{\nu} - \dot{w}^-(f) \eps \dot{w}^-(h)_{\nu} 	\\
&=  -\int_Y \dot{w}^+(f) \eps \dot{w}^+(h)_{\nu^+} + \dot{w}^-(f) \eps \dot{w}^-(h)_{\nu^-} \\
&= \eps \int_{M^+} \dot{w}^+(f) \Delta_g \dot{w}^+(h) - \langle \nabla \dot{w}^+(f), \dot{w}^+(h) \rangle \\
& \quad + \eps \int_{M^-} \dot{w}^-(f) \Delta_g \dot{w}^-(h) - \langle \nabla \dot{w}^-(f), \dot{w}^-(h) \rangle \\
&= \eps \int_M - \langle \nabla \dot{w}(f), \nabla \dot{w}(h) \rangle + W''(u) \dot{w}(f) \dot{w}(h)
\end{align*}
The second line follows since $\dot{w}^{\pm}$ coincide on $Y$, and in the third line, we recalled $\nu = \nu^+$ points towards $M^+$ while $\nu^- = - \nu$ points towards $M^-$. Clearly, the final expression is symmetric in $f$ and $h$. With this, the associated linearized operator to the second variation of $\BE_{\eps}$ is 
\begin{align*}
\delta \dot{w}_{\nu}(\cdot)&: C^2(Y) \to \R \\
\delta \dot{w}_{\nu}(f) &:= \dot{w}^+(f)_{\nu} - \dot{w}^-(f)_{\nu} 
\end{align*}
With this, we recall/define
\begin{align*}
\text{Null}_{AC_{\eps}}(u_{\eps, Y}) &= \dim \ker [\eps^2 \Delta_g - W''(u)] \\
\text{Null}_{\BE_{\eps}}(Y) &= \dim \ker [\delta \dot{w}_{\nu}]
\end{align*}
where both kernels are considered over $C^2$. 
\subsubsection{Proof}
\noindent As before, any variations of the form $u_r^i = u_{\eps, Y} + r \phi^i$ can be rewritten as 
\[
u_r^i = u_{\eps, Y} + r [\dot{w}^i + \dot{\psi}^i]
\]
for $i = 1,2$, with $\dot{w}^i$ a solution to \eqref{LinearizedSystem} and $\dot{\psi}^i|Y = 0$. Moreover, define $f^i$ such that $\dot{w}^i \Big|_Y = -f^i u_{\nu}$. We compute
\begin{align*}
\QAC(\dot{w}^1 + \dot{\psi}^1,\dot{w}^2 + \dot{\psi}^2) &= \QAC(\dot{w}^1, \dot{w}^2) + \QAC(\dot{\psi}^1, \dot{w}^2) + \QAC(\dot{w}^1, \dot{\psi}^2) + \QAC(\dot{\psi}^1, \dot{\psi}^2) \\
&=\QAC(\dot{w}^1, \dot{w}^2) + \QAC(\dot{\psi}^1, \dot{\psi}^2)  \\
&= \QBE(f^1, f^2) + \QAC(\dot{\psi}^1, \dot{\psi}^2)
\end{align*}
In the second line, we've noted as in the proof of \cref{EqualityOfIndicesTheorem} that
\[
\QAC(\dot{\psi}^1, \dot{w}^2) = \QAC(\dot{w}^1, \dot{\psi}^2) = 0
\]
by integration by parts and $\dot{\psi}^i \Big|_Y = 0$. \nl \nl
If $\dot{u}^1 = \dot{w}^1 + \dot{\psi}^1$ is AC-null, then setting $\dot{\psi}^2 = 0$, we see that $f^1$ is $\BE_{\eps}$-null. This means that 
\[
\text{Null}_{AC_{\eps}}(u_{\eps, Y}) - \text{Null}_{\BE_{\eps}}(Y) \geq 0
\]
Similarly, if $f^2 = 0$, then $\dot{w}(f^2) = 0$ by proposition \ref{NoLinearizedKernel} (or \cite{marx2023dirichlet}, Thm 1.4)  and so we have that
\begin{align*}
\QAC(\dot{\psi}^1, \cdot) &: C^2_0(M^+) \cup C^2_0(M^-) \to \R \\
\QAC(\dot{\psi}^1, \cdot) &= 0
\end{align*}
where $C^k_0(M^{\pm})$ denotes $C^k$ functions vanishing on $Y = \partial M^{\pm}$. In particular, this means that $\dot{\psi}^1$ satisfies the linearized allen-cahn equation on $M^{\pm}$, but again vanishing on $Y$. Again, by proposition \ref{NoLinearizedKernel}, we have that $\dot{\psi}^1 \equiv  0$. Thus, if $\dot{u}^1$ lies in the nullity of $\QAC$, we have that $\dot{u}^1 = \dot{w}(f^1)$ for some $f^1$ in the nullity of $\QBE$. This finishes the proof. \qed  
\section{Second Variation for $\eps \to 0$} \label{SecondVariationSmallEpsSection}
\noindent In this section we aim to prove \cref{SecondVarFormulaSmallEpsSobolevThm} on the asymptotic expansion of the second variation formula in terms of $\eps$. 
\SecondVarFormulaSmallEpsSobolevThm*
\subsection{Kernel of $\Leu$}
\noindent To prove the Sobolev version of the second variation formula, we first prove \cref{LinearizedOperatorInverseThm}, i.e. $\Leu: H^1_{\eps, 0}(M^+) \to H^{-1}_{\eps}(M^+)$ is invertible. This amounts to proving that there is no kernel:
\begin{proposition} \label{NoLinearizedKernel}
Let $Y = \partial M^+$ with bounded geometry constants \eqref{boundedGeometry} $C, \rho > 0$. For $\phi \in H^1_{\eps, 0}(M^+)$ we have for all $\eps < \eps_0 = \eps_0(C, \rho)$ sufficiently small that
\[
||\phi||_{H^1_{\eps}(M^+)} \leq K ||\Leu(\phi)||_{H^{-1}_{\eps}(M^+)}
\]
for some $K > 0$ independent of $\eps$ \end{proposition}
\noindent An analogous result in the H\"older setting was proved in (\cite{marx2023dirichlet}, Theorem 1.4) and \cref{NoLinearizedKernel} immediately gives invertibility of $\Leu$ (see e.g. \cite{gilbarg1977elliptic}, Theorem 6.15), as well as solvability of the linearized system. \eqref{LinearizedSystem}, for $f \in H^1(Y)$. To show this proposition, we first establish the following mimic of elliptic regularity:
\begin{proposition} \label{EllipticRegularityMimic}
For $\phi: M^+ \to \R$ with $\phi \in H_{0,\eps}^1(M^+)$, we have that for all $\eps > 0$
\[
\eps ||\nabla \phi||_{L^2(M^+)} \leq \frac{5}{2} (||\Leu(\phi)||_{H^{-1}_{\eps}(M^+)} + ||\phi||_{L^2(M^+)})
\]
\end{proposition}
\noindent \Pf \; We have 
\begin{align} \nonumber
\eps^2 \int_{M^+} |\nabla \phi|^2 &= - \int_{M^+} (\eps^2 \Delta_g \phi) \phi \\ \label{IBPEquality}
&= - \int_{M^+} (\Leu \phi) \phi - \int_{M^+} W''(u) \phi^2 \\ \nonumber
\Big| \int_{M^+} (\Leu \phi) \phi\Big| & \leq ||\Leu \phi||_{H^{-1}_{\eps}(M^+)} ||\phi||_{H^1_{\eps}(M^+)} \\ \nonumber
& \leq \frac{1}{2} \left( ||\Leu \phi||_{H^{-1}_{\eps}(M^+)}^2 +  ||\phi||_{H^1_{\eps}(M^+)}^2 \right) \\ \nonumber
\implies \eps^2 ||\nabla \phi||_{L^2(M^+)}^2 & \leq \frac{1}{2} ||\Leu \phi||_{H^{-1}_{\eps}(M^+)}^2  + \frac{5}{2} ||\phi||_{L^2(M^+)}^2
\end{align}
having used that $|W''(u)| \leq 2$. Taking a square root finishes this. \qed
\begin{corollary} \label{H1InitialBound}
We have 
\[
||\phi||_{H^1_{\eps}} \leq K (||\Leu(\phi)||_{H^{-1}_{\eps}(M^+)} + ||\phi||_{L^2(M^+)})
\]
\end{corollary}
\noindent To prove \ref{NoLinearizedKernel}, we see it suffices to show 
\begin{proposition} \label{L2Bound}
For $Y = \partial M^+$ with bounded geometry \eqref{boundedGeometry} constants $C, \rho > 0$, there exists an $\eps_0 = \eps_0(C, \rho) > 0$, such that for all $\eps < \eps_0$ and $\phi \in H^1_{\eps, 0}(M^+)$
\[
||\phi||_{L^2(M^+)} \leq K ||\Leu(\phi)||_{H^{-1}_{\eps}(M^+)}
\]
for some $K > 0$ independent of $\eps$.
\end{proposition}
\noindent First note the following poincare inequality near the boundary: Let $V_{\kappa \eps} \subseteq M^+$ correspond to the one-sided normal neighborhood of $Y$, given by $Y \times [0, \kappa \eps)$ in fermi coordinates. Then
\begin{lemma} \label{LocalPoincare}
Let $\phi \in H^1_{\eps}(V_{\kappa \eps})$ and $\phi(s,0) \equiv 0$ in the trace sense. Then for all $\beta > 0$, there exists an $\eps_0(\beta)$ such that for $\eps < \eps_0(\beta)$, we have
\[
||\phi||_{L^2(V_{\kappa \eps})} \leq \kappa^2 \eps^2 (1 + \beta) ||\nabla \phi||_{L^2(V_{\kappa \eps})}
\]
\end{lemma}
\noindent \Pf Take $\phi \in C^1(V_{\kappa \eps})$ and $\phi(s,0) \equiv 0$ since $C^1$ is dense in $H^1_{\eps}$ for $\eps$ fixed. Then we have 
\begin{align*}
|\phi(s,t)| &= \Big| \int_0^t \partial_r \phi(s,r) dr \Big|  \\
& \leq \int_0^t |\nabla^g \phi|(s,r) dr \\
& \leq \int_0^{\kappa \eps} |\nabla^g \phi|(s,r) dr \\
\implies |\phi(s,t)|^2 &\leq  \left( \int_0^{\kappa \eps} |\nabla^g \phi|(s,r) dr \right)^2 \\
& \leq \kappa \eps \int_0^{\kappa \eps} |\nabla^g \phi|^2 (s,r) dr
\end{align*}
The last line follows by a change of variable $t \to \frac{t}{\kappa \eps}$ and then Jensen's inequality. Now note that 
\begin{align*}
\sqrt{\det g(s,t)} dt ds &= [\sqrt{\det g|_Y(s)} + O(\eps)] dt ds \\
|\sqrt{\det g(s,t)} - \sqrt{\det g|_Y(s)}|& \leq K \eps \sqrt{\det g|_Y(s)} dt ds
\end{align*}
here, $\det g|_Y(s) ds = dA_Y$ is the volume form on $Y$ and the $O(\eps)$ expansion follows from being in fermi coordinates and $0 \leq t \leq \kappa \eps$. The second line comes from compactness of $Y$ and assuming that its riemannian (i.e. $\sqrt{\det g|_Y} \geq c_0 > 0$ for some $c_0$). This gives
\begin{align*}
\int_0^{\kappa \eps} |\phi(s,t)|^2 \sqrt{\det g(s,t)} dt &\leq [(\kappa \eps)^2 + 2K \eps^3] \int_0^{\kappa \eps} |\nabla^g \phi|^2 (s,r) \sqrt{\det g(s,r)} dr \\
\implies ||\phi||_{L^2(Y \times [0, \kappa \eps))}^2 & \leq [ \kappa^2 \eps^2 + 2K \eps^3] ||\nabla \phi||_{L^2(Y \times [0, \kappa \eps))}^2
\end{align*}
Taking $\eps$ sufficiently small we have the lemma. \qed \nl \nl
We are interested in applying lemma \ref{LocalPoincare} with  $\kappa < 1$. In particular, we note that 
\[
\kappa_0 := (3 g(t)^2 - 1)^{-1} \cap \R^+ = 0.93123\dots = 1 - \gamma, \qquad \gamma > 0
\]
i.e. the positive root of $W''(g) = 3 \tanh(t/\sqrt{2})^2 - 1$ is less than $1$. In particular, let 
\begin{align} \label{KappaDefinition}
\kappa &= 1 - \gamma/2 < 1 \\ \nonumber
\delta &= 3 g(\kappa)^2 - 1 > 0, \qquad \delta < 1
\end{align}
Then we see that for $\eps$ sufficiently small, because of uniform bounds on $||A_Y||$ and \eqref{uEpsExpansion4}, we have 
\begin{equation} \label{DeltaRange}
\frac{\delta}{2} \leq W''(u(s,\kappa \eps)) \leq 2 \delta
\end{equation}
\noindent \textbf{Proof of \ref{L2Bound}} Let $\{t > \kappa \eps\} := M^+ \backslash V_{\kappa \eps}$. We first assume that most of the mass is on $t > \kappa \eps$ in the following sense: let $0 < \alpha = \alpha(\delta) < 1$ to be determined and assume 
\begin{equation} \label{InitialMassAssumption}
||\phi||_{L^2(t > \kappa \eps)}^2 \geq (1 - \alpha) ||\phi||_{L^2(M^+)}^2
\end{equation}
Then we have 
\begin{equation} \label{IntegralBoundOne}
\frac{\delta}{2} ||\phi||_{L^2(t > \kappa \eps)}^2 \leq  \int_{t > \kappa \eps} W''(u) \phi^2 \leq 2 ||\phi||_{L^2(t > \kappa \eps)}^2
\end{equation}
since $t > \kappa \eps$ implies that  $W''(u) > \frac{\delta}{2}$. Note also that by equation \eqref{DeltaRange}
\[
\Big|\int_Y \int_0^{\kappa \eps}W''(u) \phi^2\Big| \leq ||\phi||_{L^2(V_{\kappa \eps})}^2 \leq \alpha ||\phi||_{L^2(M^+)}^2
\]
and so
\begin{equation} \label{IntegralBoundAwayFromBoundary}
\int_{M^+} W''(u) \phi^2 \geq \left(1 - \frac{2\alpha}{\delta(1 - \alpha)} \right)\int_{t > \kappa \eps} W''(u) \phi^2
\end{equation}
Now combining \eqref{InitialMassAssumption} and \eqref{IntegralBoundOne} to start, and then using \eqref{IntegralBoundAwayFromBoundary}, we get:
\begin{align*}
||\phi||_{L^2(M^+)}^2 &\leq \frac{2}{\delta(1 - \alpha)} \int_{t > \kappa \eps} W''(u) \phi^2 \\
\implies ||\phi||_{L^2(M^+)}^2 &\leq \frac{2}{\delta - \alpha \delta - 2\alpha} \int_{M^+} W''(u) \phi^2 \\
& = \frac{2}{\delta - \alpha \delta - 2\alpha} \left[- \int_{M^+} \phi \Leu(\phi) - \eps^2 \int_{M^+} |\nabla \phi|^2 \right] \\
& = \frac{2}{\delta - \alpha \delta - 2\alpha} \left[- \int_{M^+} \phi \Leu(\phi)\right] \\
& \leq \frac{2}{\delta - \alpha \delta - 2\alpha} \left[\frac{\lambda}{2} ||\phi||_{H^1_{\eps}(M^+)}^2 + \frac{1}{2\lambda} ||\Leu \phi||_{H^{-1}_{\eps}(M^+)}^2 \right] \\
\implies ||\phi||_{L^2(M^+)}^2 &\leq K ||\Leu(\phi)||_{H^{-1}_{\eps}(M^+)}^2
\end{align*}
In the last line, we noted that 
\[
||\phi||_{H^1_{\eps}(M^+)}^2 \leq K [||\phi||_{L^2(M^+)}^2 + ||\Leu(\phi)||_{H^{-1}_{\eps}(M^+)}^2]
\]
from corollary \ref{H1InitialBound}, and also chose $\alpha < \delta/4$ so that 
\[
\delta - \alpha (2 + \delta) > \frac{\delta}{2}
\]
and finally chose $\lambda$ sufficiently small independent of $\eps$ so that we can move the $K \lambda ||\phi||_{L^2(M)}$ term over to the left side. This tells us that if $||\phi||_{L^2(t > \kappa \eps)}^2 \geq (1 - \alpha) ||\phi||_{L^2(M)}^2$ for $\alpha < \delta/8$, then bound holds. \nl \nl
Now suppose the contrary:
\[
||\phi||_{L^2(Y \times [0,\kappa \eps)}^2 > \frac{\delta}{8} ||\phi||_{L^2(M)}^2
\]
We bound on $||\phi||_{L^2(Y \times [0, \kappa \eps))}$ for $\kappa$ as in \eqref{KappaDefinition}, using explicitly that $\kappa < 1$ and our bound in lemma \ref{LocalPoincare} 
\begin{align*}
||\phi||_{L^2(Y \times [0, \kappa \eps)}^2 &\leq \eps^2 \kappa^2 (1 + \beta) ||\nabla \phi||_{L^2(Y \times [0, \kappa \eps))}^2 \\
& \leq \eps^2 \kappa^2 (1 + \beta) ||\nabla \phi||_{L^2(M^+)}^2 \\
& \leq \kappa^2 (1 + \beta) \left[ - \int_M \phi \Leu(\phi) - \int_M W''(u) \phi^2 \right]
\end{align*}
explicitly using \eqref{IBPEquality} from proposition \ref{EllipticRegularityMimic}. Now we note that 
\[
- \int_M W''(u) \phi^2 \leq -\int_{t = 0}^{t = \kappa \eps} W''(u) \phi^2  \leq \int_{t = 0}^{t = \kappa \eps} \phi^2 = ||\phi||_{L^2(Y \times [0, \kappa \eps))}^2
\]
By definition of $W''(u)$ and since $W''(u) > 0$ on $t \geq \kappa \eps$ and $|W''(u)| \leq 1$ on $[0, \kappa \eps]$. Now we have 
\begin{align*}
||\phi||_{L^2(Y \times [0, \kappa \eps)}^2 (1 - \kappa^2(1 + \beta)) & \leq - \int_M \phi \Leu(\phi) \\
& \leq \frac{\lambda}{2} ||\phi||_{H^{1}_{\eps}}^2 + \frac{1}{2\lambda} ||\Leu(\phi)||_{H^{-1}_{\eps}}^2 
\end{align*}
Note that the value $\beta$ from lemma \ref{LocalPoincare} can be chosen sufficiently small so that $1 - \kappa^2(1 + \beta) > 0$. If $\lambda$ is sufficiently small (independent of $\eps$) and we use corollary \ref{H1InitialBound} again to compute
\begin{align*}
||\phi||_{L^2(Y \times [0, \kappa \eps))}^2 & \leq K ||\Leu(\phi)||_{H^{-1}_{\eps}(M)}^2  \\
\implies ||\phi||_{L^2(M)}^2 & \leq \frac{8 K}{\delta} ||\Leu(\phi)||_{H^{-1}_{\eps}(M)}^2 
\end{align*}
which is the desired bound. \qed
\begin{Remark}
Our proof relies heavily on the fact that $\kappa_0 < 1$, and hence, heavily on the choice of $W(x) = \frac{(1 - x^2)^2}{4}$ as our double well potential. We expect a similar statement about $\kappa_0 < 1$ to hold for all double well potentials with $W(0) = \frac{1}{4}$ and $W(\pm 1) = 0$. The authors currently lack a proof for this.
\end{Remark}
%
%
\section{Proof of \cref{SecondVarFormulaSmallEpsSobolevThm}} \label{SecondVarProofSection}

\subsection{Set up}
Recall that we have $Y^{n-1} \subseteq M^n$ a closed separating hypersurface with $M = M^+ \sqcup_Y M^-$. Let $u_{\eps}^{\pm}$ be the Dirichlet minimizer of Allen--Cahn on $M^{\pm}$ which is positive (resp. negative) away from $Y$ the boundary. Recall the linearized system for $\dot{u}$ as in \eqref{LinearizedSystem}: 
\begin{align*}
\eps^2 \Delta_g \dot{u} &= W''(u) \dot{u} \\
\dot{u} \Big|_Y &= - f(s) u_{\nu}(s) = h(s)
\end{align*}
where $s$ is a coordinate on $Y$. We posit 
\begin{equation} \label{dotUInitialDecomposition}
\dot{u} = h(s) \dot{\bg}_{\eps}(t) + \phi
\end{equation}
so that $\phi: M^+ \to \R$ and $\phi\Big|_Y \equiv 0$. We want to produce $H^1_{\eps}$ bounds on $\phi$. In the following sections, we'll be working with one-sided solutions and often write $M$ in place of $M^{\pm}$, noting that our results can be applied to either $M^+$ or $M^-$.
\subsection{Computing the bound} \label{ComputeBound}
In this section, we apply proposition \ref{NoLinearizedKernel} and show
\begin{lemma} \label{phiBoundByfLemma}
There exists an $\eps_0 > 0$ such that for all $\eps < \eps_0$ and $\phi: M \to \R$ as in equation \eqref{dotUInitialDecomposition}, we have
\begin{equation} \label{phiH1Bound}
||\phi||_{H^1_{\eps}(M)} \leq K \eps [||\nabla f||_{L^2(Y)} + \sqrt{\eps} ||f||_{L^2(Y)}]
\end{equation}
for $K$ independent of $\eps$. In particular $||\phi||_{H^1_{\eps}(M)} \leq K \eps ||f||_{H^1(Y)}$.
\end{lemma}
\noindent \Pf First note
\[
\Leu(\phi) = - \Leu(h \dot{\bg}_{\eps}(t))
\]
and 
\begin{equation} \label{hExpansion}
h(s) = -f(s) \left[\frac{1}{\eps \sqrt{2}} + O(1)\right] 
\end{equation}
from the expansion of $\partial_t u |_{t = 0} = u_{\nu}$ implicit in \eqref{uEpsExpansion4} for surfaces with bounded geometry. We compute bounds on $\Leu(\phi)$
\begin{align*}
||\Leu(\phi)||_{H^{-1}_{\eps}(M)} &= ||- \eps^2 \Delta_t(h) \dot{\bg}_{\eps} +  \eps H_t h \ddot{\bg}_{\eps} - [W''(u) - W''(\bg_{\eps}) + E] h(s) \dot{\bg}_{\eps} ||_{H^{-1}_{\eps}} \\
& \leq \Big[ \eps  ||\Delta_t(f) \dot{\bg}_{\eps} ||_{H^{-1}_{\eps}(M)}  +  ||H_t f \ddot{\bg}_{\eps}||_{L^2(M)} \\
& \quad \quad + \eps^{-1} ||[W''(u) - W''(\bg_{\eps}) + E] f \dot{\bg}_{\eps}||_{L^2(M)} \Big]
\end{align*}
by the expansion of $h(s)$ in \eqref{hExpansion}. Note that
\[
H_t = H_Y + \eps \dot{H}_Y (t/\eps) + O(\eps^2)
\]
Recall the expansion of $u$ as in \eqref{uEpsExpansion5}
\begin{align*}
u(s,t) &= \bg_{\eps}(t) + \eps H_Y(s) \bw_{\eps}(t) + \eps^2 \left[ \dot{H}_Y(s) \btau_{\eps}(t) + H_Y(s)^2 \brho_{\eps}(t) + \frac{1}{2} H_Y(s)^2 \bkappa_{\eps}(t) \right] + \phi(s,t) \\
||\phi||_{C^{2,\alpha}_{\eps}} &= O(\eps^3) 
\end{align*}
For critical points of $\BE$, we know that $u_{\nu} = u_{\nu}^c$ so from \eqref{MatchingNormal}, we have 
%
\begin{equation} \label{MCCritExpansion}
H_Y = O(\eps^2)
\end{equation}
We then have the following in $C^0$ and $L^2$:
\begin{align}  \nonumber
u(s,t) &= \bg_{\eps}(t) + \eps^2 \dot{H}_Y(s) \btau_{\eps}(t) + O(\eps^3) \\ \nonumber
W''(u) - W''(\bg_{\eps}) &= 3(u - \bg_{\eps})(u + \bg_{\eps}) \\ \nonumber
&= 3\eps^2 \dot{H}_Y \btau_{\eps}(2 \bg_{\eps}) + O(\eps^3) \\ \label{SecondDerivDiffExpansion}
&= 6 \eps^2 \dot{H}_Y \btau_{\eps} \bg_{\eps} + O(\eps^3)
\end{align}
Furthermore, since $H_Y = H_0 =O(\eps^2)$ and $W''(u) - W''(\bg_{\eps}) = O(\eps^2)$, we have that 
\begin{align*}
||H_t f \ddot{\bg}_{\eps}||_{L^2(M)}^2 & = \int_{Y} \int_{0}^{-2 \omega \eps \ln(\eps)} H_t^2 f^2 \ddot{\bg}_{\eps}^2 \\
&\leq K \eps^3 ||f||_{L^2(B_2)}^2 \\
||H_t f \ddot{\bg}_{\eps}||_{L^2(M)} & \leq K \eps^{3/2} ||f||_{L^2(Y)}
\end{align*}
where the extra factor of $\eps$ comes from integrating $\ddot{\bg}_{\eps}^2$. Similarly, we have
\begin{align*}
\frac{1}{\eps\sqrt{2}} ||[W''(u) - W''(\bg_{\eps}) + E] f \dot{\bg}_{\eps}||_{L^2(M)} & \leq K \eps^{3/2} ||f||_{L^2(B_2)}
\end{align*}
So it suffices to compute $||\Delta_t(f) \dot{\bg}_{\eps}||_{H^{-1}_{\eps}}$. Note that
\begin{align*}
v &\in H^1_{0, \eps}(M) \\
\int_{M} \Delta_t(f) \dot{\bg}_{\eps} v &= \int_{Y} \int_0^{-2\omega \eps \ln(\eps)} \Delta_t(f) \dot{\bg}_{\eps} v dA \\
&= \int_0^{-2 \omega \eps \ln(\eps)} \dot{\bg}_{\eps} \int_{Y} \Delta_t(f) v \sqrt{\det g(s,t)} ds dt \\
&= \int_0^{- 2\omega \eps \ln(\eps)} \dot{\bg}_{\eps} \left( \int_{Y} g|_Y\left(\nabla^t \left( v(s,t) \sqrt{ \det g(s,t)}\right), \nabla^t(f)\right) ds \right) dt \\
\Big| \int_M \Delta_t(f) \dot{\bg}_{\eps} v \Big| & \leq K \eps ||v||_{H^1} ||\nabla f||_{L^2} \\
& \leq K ||v||_{H^1_{\eps}} ||\nabla f||_{L^2}
\end{align*}
which tells us that 
\[
||\Leu(\phi)||_{H^{-1}_{\eps}(M)} \leq K \eps ||\nabla f||_{L^2(Y)} + K \eps^{3/2} ||f||_{L^2(Y)}
\]
Applying proposition \ref{NoLinearizedKernel} we have
\begin{equation} \label{BestPhiH1Bound}
||\phi||_{H^1_{\eps}(M)} \leq K \eps [||\nabla f||_{L^2(Y)} + \sqrt{\eps} ||f||_{L^2(Y)}]
\end{equation}
completing the lemma. \qed 
\subsection{Computing second variation}
We have that for $f \in C^2$
\[
\BE''(f \nu) \Big|_Y = \int_Y f u_{\nu}[ \dot{u}_{\nu} - \dot{u}_{\nu}^c]
\]
We've shown that 
\[
\dot{u}_{\nu}(s) = \partial_t \left( h(s) \dot{\bg}_{\eps}(t) + \phi \right)(s,0) = \phi_t(s,0)
\]
and also $u = \bg_{\eps}(t) + O(\eps^2)$ in $C^{1,\alpha}_{\eps}$ and so
\begin{align*}
\BE''(f \nu) \Big|_Y &= \int_Y f u_{\nu} [ \phi_t - \phi^c_t] \\
&= \int_Y f \left[ (\eps \sqrt{2})^{-1} + O(\eps) \right] [ \phi_t - \phi^c_t]
\end{align*}
We write
\begin{align} \label{NormalIntegration}
-\eps^2 \phi_t(s,0)\sqrt{\det g(s,0)} &= \sqrt{2}\int_0^{-2\omega \eps \ln(\eps)} (\eps^2 \partial_t^2 - W''(\bg_{\eps})) \phi \dot{\bg}_{\eps} \sqrt{\det g(s,t)} dt + R \\ \nonumber
&=  \sqrt{2}\int_0^{-2\omega \eps \ln(\eps)} \Leu(\phi) \dot{\bg}_{\eps} \sqrt{\det g(s,t)} dt \\ \nonumber
&\quad \quad - \sqrt{2}\int_0^{-2\omega \eps \ln(\eps)} [\eps^2 \Delta_t(\phi) - H_t \eps^2 \phi_t] \dot{\bg}_{\eps} \sqrt{\det g(s,t)} dt  \\ \nonumber
& \quad \quad  + \sqrt{2}\int_0^{-2\omega \eps \ln(\eps)} [W''(u) - W''(\bg_{\eps}) + E] \phi \dot{\bg}_{\eps} \sqrt{\det g(s,t)} dt + R
\end{align}
using that $\dot{g}(0) = 2^{-1/2}$ and  
\begin{equation} \label{ErrorIntegralBound}
\Big|\int_Y f R \Big| \leq K \eps^{3} ||f||_{H^1(Y)}^2
\end{equation}
(see appendix \S \ref{ReducingBE}). From before, we have 
\begin{align*}
\Leu(\phi) &= -\Leu(h(s) \dot{\bg}_{\eps}) \\
&= - \eps^2 \Delta_t(h) \dot{\bg}_{\eps} +  \eps H_t h \ddot{\bg}_{\eps} + [W''(u) - W''(\bg_{\eps}) + E] h(s) \dot{\bg}_{\eps} 
\end{align*}
So that from \eqref{hExpansion} and \eqref{NormalIntegration}
\begin{align*}
\int_Y f \phi_t &=  -\eps^{-2}\sqrt{2}\int_Y f \int_0^{-2\omega \eps \ln(\eps)} \Big(-\eps^2 \Delta_t(h) \dot{\bg}_{\eps}^2 + \eps H_t h \ddot{\bg}_{\eps} \dot{\bg}_{\eps} + [W''(u) - W''(\bg_{\eps}) + E] h(s) \dot{\bg}_{\eps}^2 \\
& \qquad \qquad -\sqrt{2}[\eps^2 \Delta_t(\phi) - H_t \eps^2 \phi_t - (W''(u) - W''(\bg_{\eps}) + E) \phi] \dot{\bg}_{\eps} \Big) \sqrt{\det g(s,t)} dt ds \\
&= \sqrt{2}\int_0^{-2\omega \eps \ln(\eps)} \int_{Y_t} \left[f \Delta_t(h) \dot{\bg}_{\eps}^2 - \dot{H}_0 f^2 (t/\eps) \ddot{\bg}_{\eps} \dot{\bg}_{\eps} - 6\dot{H}_0 f h \dot{\btau}_{\eps}(t) \bg_{\eps}(t) \dot{\bg}_{\eps}^2 + f \Delta_t (\phi) \dot{\bg}_{\eps} \right] dA_{Y_t} dt \\
& \quad \quad + \tilde{R}
\end{align*}
where
\[
|\tilde{R}| \leq K \eps ||f||_{H^1}^2
\]
by grouping the higher order terms in the expansion of $H_t$ and $W''(u) - W''(\bg_{\eps})$. In the appendix (\S \ref{ComputeBound}), we integrate by parts and compute integrals of our ODE solutions and show
\begin{align} \label{halfIntegralSecondVar}
\int_Y f u_{\nu} \dot{u}_{\nu} &= \frac{\sqrt{2}}{3} \int_Y  (|\nabla f|^2 - \dot{H}_Y f^2 ) dA_Y + E \\ \nonumber
|E| &\leq K \eps^{1/2} ||f||_{H^1(Y)}^2
\end{align}
Noting that $|\nabla f|^2$ and $\dot{H}_Y$ are invariant under the choice of normal, we see that 
\begin{align*}
\BE''(f \nu) \Big|_Y &= \int_Y f u_{\nu} [\dot{u}^+_{\nu} - \dot{u}^-_{\nu}] \\
&= \int_Y f u_{\nu} [\dot{u}^+_{\nu} + \dot{u}^-_{\nu^-}] \\
& = 2\frac{\sqrt{2}}{3} \int_Y  (|\nabla f|^2 - \dot{H}_Y f^2 ) dA_Y + E \\ 
|E| &\leq K \eps^{1/2} ||f||_{H^1(Y)}^2
\end{align*}
where $\nu^- = - \nu$ is the normal to $Y$ pointing towards $M^-$.
\begin{remark}
    Using \cref{SecondVarFormulaSmallEpsSobolevThm} it is straightfoward to see that there are no stable critical points of the $\BE_{\eps}$ functional on Manifolds with positive Ricci curvature. This is identical to the situation for the usual Allen--Cahn energies, as well as for the case of minimal surfaces. To see this suppose that 
    \[
    \Ric \geq C \textnormal{Id}
    \]
    for some $C > 0$.
    %
    %
    With  \cref{SecondVarFormulaSmallEpsSobolevThm}, we see that if we take $f \equiv 1$ then
    \[
    BE''(\nu) = -\frac{2 \sqrt{2}}{3} \int_Y \dot{H}_0 + O(\eps^{1/2} \text{Vol}(Y))
    \]
    We recall that
    \[
    \dot{H}_0 = |A_Y|^2 + \Ric_g(\nu, \nu)
    \]
    So we can say prove the following result, \cref{RicciStability}
    \RicciStability*
\end{remark}
\begin{remark}
Though we initially assumed $f \in C^2$, we can define the second variation formula for any $f \in H^1(Y)$ by using the final formula.
\end{remark}
\begin{remark} \label{HolderRemark}
One can also prove a H\"older version of this theorem and show that 
\begin{align*} 
\dot{u}^+_{\nu} - \dot{u}^-_{\nu} &= 2 \frac{\sqrt{2}}{3} [ \Delta_Y + \dot{H}_0](f) + E(f)  \\
||E(f)||_{C^{\alpha}(Y)} & \leq K \eps ||f||_{C^{2,\alpha}(Y)}
\end{align*} 
To prove this, one can use \eqref{uEpsExpansion5} to expand $u$ and fermi coordinates to expand $\eps^2 \Delta_g$. This is enough to expand $\dot{u}$ and compute the neumann difference to sufficiently high order in $\eps$
\end{remark}
\subsubsection{$\BE_{\eps}$ eigenvalues}
We clearly see that for $\eps$ sufficiently small,
\[
\frac{d^2}{dt^2} \BE_{\eps}(Y_t) \Big|_{t = 0} \geq \frac{\sqrt{2}}{3}||\nabla f||_{L^2(M)}^2 - K ||f||_{L^2(M)}^2
\]
with $K$ independent of $\eps$. With this, we can minimize
\[
\frac{d^2}{dt^2} \BE_{\eps}(Y_t) \Big|_{t = 0} = D^2 \BE_Y(f)
\]
with respect to $f$ for $f \in H^1(Y)$. This allows us to use the Rayleigh quotient to define eigenvalues for $\BE''_{\eps, Y}$ at a critical point $Y$:
\begin{align*}
\lambda_{i} &:= \inf_{\substack{f \in E \\ E \subset H^{1}(Y)}} \frac{\BE''_{\eps, Y}(f)}{||f||_{L^2(Y)}^2}, \\
\lambda_{1} &\leq \cdots \leq \lambda_{m} \leq 0 = \lambda_{m + 1} = \cdots = \lambda_{m + n} \leq \cdots.
\end{align*}
Here $E$ ranges over $i$-dimensional subspaces of $H^{1}(Y)$, $n$ is the $\BE_{\eps}$ nullity, and $m$ is the $\BE_{\eps}$ Morse index of $Y$. We recall from section \S \ref{EqualityOfIndices} that $m = \Ind_{AC_{\eps}}(u_{\eps, Y})$ and $n \leq \text{Null}_{AC_{\eps}}(u_{\eps, Y})$.
\section{Applications of Second Variation Formula} \label{ApplicationsOfSecondVariation}
\subsection{Allen--Cahn Index of symmetric solutions}
For $Y^{n-1} \subseteq M^n$ minimal, let $\text{Ind}_{\text{Min}}(Y)$ denote the index of $Y$ as a minimal surface. We prove the following corollary of \cref{SecondVarFormulaSmallEpsSobolevThm}, \cref{MinimalIndexEqualsACIndex}
%
\MinimalIndexEqualsACIndex*
\noindent \Pf The first equality is shown in \cref{EqualityOfIndicesTheorem}. For the third equality, we look at the formula:
\[
\BE_{\eps}''(f \nu) \Big|_Y = -\frac{2\sqrt{2}}{3}\int_Y f [J_Y + E](f)
\]
for $J_Y = \Delta_Y + \dot{H}_Y$. Let $V \subseteq C_c^{\infty}(Y)$ denote the finite dimensional vector space such that
\[
f \in V, \; f \neq 0 \implies D^2A(f \nu) \Big|_Y = -\int_Y f J_Y(f) < 0
\]
By self-adjointness of $J_Y$, we can diagonalize $V$ and get an orthogonal basis $\{f_1, \dots, f_m\}$. Without loss of generality, we can normalize each so that $||f_i||_{C^2(Y)} = 1$. Now we compute
\begin{align*}
\BE_{\eps}''(a^i f_i) &= -\frac{2\sqrt{2}}{3}\int_Y (a^i f_i) (J_Y + E)(a^j f_j) \\
&= -\frac{2\sqrt{2}}{3} \left[ \int_Y (a^i f_i) J_Y(a^j f_j) + (a^i f_i) E (a^j f_j) \right] \\
&= \frac{2\sqrt{2}}{3} D^2 A(a^i f_i) \Big|_Y - 2 \sigma_0 \int (a^i f_i) E(a^j f_j) \\
&= \frac{2\sqrt{2}}{3} \left[(a^i)^2 D^2 A\Big|_Y (f_i)\right] - 2 \sigma_0 \int (a^i f_i) E(a^j f_j)
\end{align*}
Now note that we can bound
\begin{align*}
\Big|\int (a^i f_i) E(a^j f_j) \Big| &= \Big| a^i a^j \int f_i E (f_j) \Big| \\ 
& \leq \frac{(a^i)^2 + (a^j)^2}{2} \Big| \int f_i E(f_j) \Big| \\
& \leq \frac{(a^i)^2 + (a^j)^2}{2}  \int |f_i|  \; |E(f_j)| \\
& \leq \frac{(a^i)^2 + (a^j)^2}{2}  K \eps 
\end{align*}
for $K$ independent of $i$ and $\eps$. Summing over both $i$ and $j$, we get
\begin{align*}
\BE_{\eps}''(a^i f_i) &\leq \frac{2\sqrt{2}}{3} \sum_i (a^i)^2 D^2 A\Big|_Y (f_i) + K \eps \sum_i (a^i)^2 \\
& \leq \frac{2\sqrt{2}}{3} \sum_i (a^i)^2 [D^2 A |_Y(f_i) + K \eps] 
\end{align*}
Define
\[
K_i := -D^2 A(f_i \nu) > 0
\]
\noindent Note that $K_i$ are independent of $\eps$, so as long as
\[
\eps < \frac{1}{2K} \sup_i K_i
\]
we have that
\begin{align*}
\BE_{\eps}''(a^i f_i) &\leq \frac{\sqrt{2}}{3} \sum_i (a^i)^2 D^2 A |_Y(f_i) < 0
\end{align*}
this finishes the proof. \qed \nl \nl
%
%
This inequality can be used to reprove the index of a family of symmetric solutions to Allen--Cahn vanishing on the clifford torus in $S^3$
\begin{restatable}{corr}{cajuHiesmayrThm} \label{cajuHiesmayrThm}
Let $Y^2 \subseteq S^3$ be the clifford torus and let $u_{\eps}: S^2 \to \R$ be the sequence of solutions vanishing along $Y$ constructed in \cite{caju2020ground}, Prop 9.1. There exists an $\eps_0 > 0$ such that for all $\eps < \eps_0$
\[
\Ind_{\text{AC}_{\eps}}(u_{\eps})  = 5
\]
\end{restatable}
\begin{Remark}
This result, and a stronger rigidity result, was originally proved by Hiesmayr in theorem $1$ \cite{hiesmayr2020rigidity}. A separate proof was given in \cite{caju2019solutions}, theorem 1.1
\end{Remark}
\noindent \Pf Note that the clifford torus, $\mathcal{T}$, is minimal inside of $S^3$, and that the solutions $\{u_{\eps}\}$ are equal to the energy minimizers on either side of $\mathcal{T}$ inside of $S^3$ (in fact, such solutions are unique by uniqueness of solutions with the dirichlet condition). This tells us that 
\[
\Ind_{AC_\eps}(u_{\eps}) = \Ind_{\BE_{\eps}}(\mathcal{T})
\]
Now apply corollary \ref{MinimalIndexEqualsACIndex} and Urbano's original proof that 
\[
\Ind_{\text{Min}}(\mathcal{T}) = 5
\]
(see \cite{urbano1990minimal}, Theorem) to conclude that 
\[
\Ind_{AC_{\eps}}(u_{\eps}) \geq 5
\]
To bound from above, we apply Chodosh-Mantoulidis Theorem 5.11 \cite{chodosh2020minimal} to get that 
\begin{equation} \label{ChodoshMantoulidis}
\Ind_{AC_{\eps}}(u_{\eps}) + \text{Null}_{AC_{\eps}}(u_{\eps}) \leq \Ind_{\text{Min}}(\mathcal{T}) + \text{Null}_{\text{Min}}(\mathcal{T})
\end{equation}
and note that $\text{Null}_{\text{Min}}(\mathcal{T}) = 4$ by \cite{hsiang1971minimal}. In fact, the nullity of $\mathcal{T}$ is generated by the subspace $V \subseteq \mathfrak{s}\mathfrak{o}(4)$ of dimension $4$ which \textit{does not} fix $\mathcal{T}$. The same subspace of killing fields give rise to diffeomorphisms of $S^3$ which preserve $E_{\eps}(u_{\eps}, S^3) = \BE_{\eps}(\mathcal{T})$ and so 
$\text{Null}_{AC_{\eps}}(u_{\eps}) \geq 4$. Immediately by equation \eqref{ChodoshMantoulidis}, we conclude $\Ind_{AC_{\eps}}(u_{\eps}) = 5$. \qed 
%
%
\subsection{Application: Allen--Cahn Index of $D_{2p}$-symmetric solutions on $S^1$} \label{S1MorseIndex}
\noindent Identify $S^1 \cong [0,1]$. Let $p \in \N^{+}$. There exists an $\eps_p > 0$, such that for any $\eps_p > \eps > 0$, there exists a positive energy minimizer of \eqref{ACEnergy} on $[0, 1/(2p)]$ with dirichlet boundary conditions, $u_{\eps, 2p}^+$. Now define
\begin{align*}
u_{\eps, 2p}(x) &= (-1)^i u_{\eps, 2p}^+(x \mod 1/(2p)) \\
x & \in [i/2p, (i+1)/2p]
\end{align*}
i.e. we reflect and repeat $u_{\eps, 2p}^+$ at each subsequent interval of length $1/(2p)$. By symmetry, this is a solution to Allen--Cahn. We recall that all solutions to Allen--Cahn on $S^1$ with a finite number of zeros are of this form. This is done in \cite{mantoulidis2022double}, proposition 2.1, among other places
\begin{theorem}[\cite{mantoulidis2022double}, Prop 2.1]\label{OneDimPeriodicClassification}
Let $\eps > 0$. Given $u$ a solution to \eqref{ACEquation} on $S^1$ with $|u^{-1}(0)| > 0$ and finite, $u = u_{\eps, 2p}$ for some $p \geq 0$
\end{theorem}
\noindent \Pf Lift $u$ to a solution on $\R$. Using the Poincar\'e-Bendixson theorem (\S 10.5 in \cite{hirsch2012differential}) on the dynamical system given by $(u, u_x)$, we see that $u$ is either constant, periodic, or heteroclinic. If $u$ is constant, then $p = 0$ and $|u^{-1}(0)| = 0$. If $|u^{-1}(0)| > 0$, then $u$ cannot be heteroclinic, as this is not $1$-periodic and $u$ is lifted from $S^1$ with at least $1$ zero. Thus $u$ must be periodic on $\R$. \nl \nl
Without loss of generality, $u(0) = 0$ and let $\alpha > 0$ be the first positive $0$ of $u$. Now construct
\begin{align*} 
\tilde{u}(x) &:= (-1)^i u(x \; \text{mod} \; \alpha) \\
x &\in [i \alpha, (i+1) \alpha]
\end{align*}
The two functions agree on $[0, \alpha]$, so by existence and uniqueness, they are equal everywhere. Now note that because $u$ is lifted from $S^1$, we have $u(0) = u(1) = 0$, which means that $n \alpha = 1$ for some $n$, i.e. $\alpha = 1/n$. Finally, note that $u$ changes sign at $k \alpha$, and $u_{x} \Big|_{x = 0}$ is only defined if $n$ is even. Thus $\alpha = 1/(2p)$ for some $p \in \N^+$. \qed \nl \nl
With this, our goal is to compute the Morse index of all solutions on $S^1$ with $|u^{-1}(0)|$ finite. We'll first show \cref{CircleSecondVarProp} 
\CircleSecondVarProp*
\noindent and then conclude \cref{S1MorseIndexTheorem}
\CircleMorseIndexTheorem*
\begin{Remark}
By \cref{EqualityOfIndicesTheorem} (adapted for a one dimensional manifold), theorem \ref{S1MorseIndexTheorem} shows that the Allen--Cahn index of $u_{\eps, 2p}$ equals $2p-1$ for all $\eps > 0$ sufficiently small.
\end{Remark}
\subsubsection{Reformulating second variation}
For general $p$, we reframe
\begin{align} \label{YtDefinition}
q_i &= \frac{i}{2p} + f\left(\frac{i}{2p}\right) t
\end{align}
so that we can write the second variation in a more convenient way
\begin{lemma}
Let $u_{\eps, 2p}: S^1 \to \R$ be the solution to \eqref{ACEquation} with $2p$ symmetric nodes and let $\dot{u}_{\eps, i}$ be the solution to \eqref{LinearizedSystem} on each subinterval of $\left[ \frac{i}{2p}, \frac{i+1}{2p} \right]$. Then
\begin{align*}
BE'' &= \eps c \sum_{i = 0}^{2p-1} \left[ f\left(\frac{i}{2p} \right) \dot{u}_{i,x}\left(\frac{i}{2p}\right) + f\left(\frac{i+1}{2p}\right) \dot{u}_{i,x}\left(\frac{i+1}{2p}\right) \right]
\end{align*}
\end{lemma}
\noindent \Pf Using \eqref{SecondVariationFormula} and letting $\dot{u}_x^{\pm}$ denote the $x$ derivative from the $\pm$ direction, we compute:
\begin{align*}
BE''(f) &= \eps \sum_{i = 0}^{2p-1} (-1)^i f\left(\frac{i}{2p}\right) u_{x}\left(\frac{i}{2p}\right) \left[ \dot{u}^+_{x}\left( \frac{i}{2p} \right) - \dot{u}^-_{x}\left(\frac{i}{2p}\right)\right] \\
&= \eps \sum_{i = 0}^{2p-1} (-1)^i\left[ f\left(\frac{i}{2p} \right)  u_x\left(\frac{i}{2p}\right) \dot{u}_{i,x}\left(\frac{i}{2p}\right) + f\left(\frac{i+1}{2p}\right) u_x\left(\frac{i+1}{2p}\right)  \dot{u}_{i,x}\left(\frac{i+1}{2p}\right) \right]
\end{align*}
where on the interval of $\left[\frac{i}{2p}, \frac{i+1}{2p} \right]$, we have defined 
\begin{align*}
\eps^2 \dot{u}_{i,xx} &= W''(u) \dot{u} \\
\dot{u}_i\left( \frac{i}{2p} \right) &= f\left(\frac{i}{2p} \right) u_x\left( \frac{i}{2p} \right)\\ 
\dot{u}_i\left( \frac{i+1}{2p} \right) &= f\left(\frac{i+1}{2p} \right) u_x\left( \frac{i+1}{2p} \right)
\end{align*}
Letting $c = u_x(0) > 0$, and noting that 
\[
u_x\left(\frac{i}{2p} \right) = (-1)^i c
\]
our second variation formula becomes 
\begin{align*}
BE'' &= \eps c \sum_{i = 0}^{2p-1} \left[ f\left(\frac{i}{2p} \right) \dot{u}_{i,x}\left(\frac{i}{2p}\right) + f\left(\frac{i+1}{2p}\right) \dot{u}_{i,x}\left(\frac{i+1}{2p}\right) \right]
\end{align*}
and our system for $\dot{u}_i$ becomes
\begin{align} \label{IntervalSystem}
\dot{u}_i:&\left[ \frac{i}{2p}, \frac{i+1}{2p} \right] \to \R \\
\eps^2 \dot{u}_{i,xx} &= W''(u) \dot{u}_i \\ \nonumber
\dot{u}_i\left( \frac{i}{2p} \right) &= c f\left(\frac{i}{2p} \right)\\  \nonumber
\dot{u}_i\left( \frac{i+1}{2p} \right) &= -c f\left(\frac{i+1}{2p} \right) \nonumber
\end{align}
proving the lemma. \qed 
\subsubsection{Isometric reduction of Allen--Cahn solutions on $S^1$}
We now observe that we can shift the boundary conditions of each individual subinterval $[i/2p, (i+1)/2p]$ and get the same energy, and hence variations of energy. If we let $Y_t = \{q_i(t)\}$ as in \eqref{YtDefinition}, then we can write
\[
\BE(Y_t) = \sum_{i = 0}^{2p-1} \text{E}_{\eps}([q_i(t), q_{i+1}(t)])
\]
where $q_{2p}(t) = q_0(t)$. Here, $\text{E}([a,b])$ denotes the energy of the minimizer of \eqref{ACEnergy} on $[a,b]$ at scale $\eps$. Note that for $a(t) \leq b(t)$, we have 
\[
\text{E}_{\eps}([a(t), b(t)]) = \text{E}_{\eps}([0, b(t) - a(t)]) = \text{E}_{\eps}([a(t) - b(t), 0])
\]
This can be thought of as a change of reference frame to one of the endpoints. See picture \ref{fig:rightendpointframe}.
\begin{figure}[h!]
\centering
\includegraphics[scale=0.3]{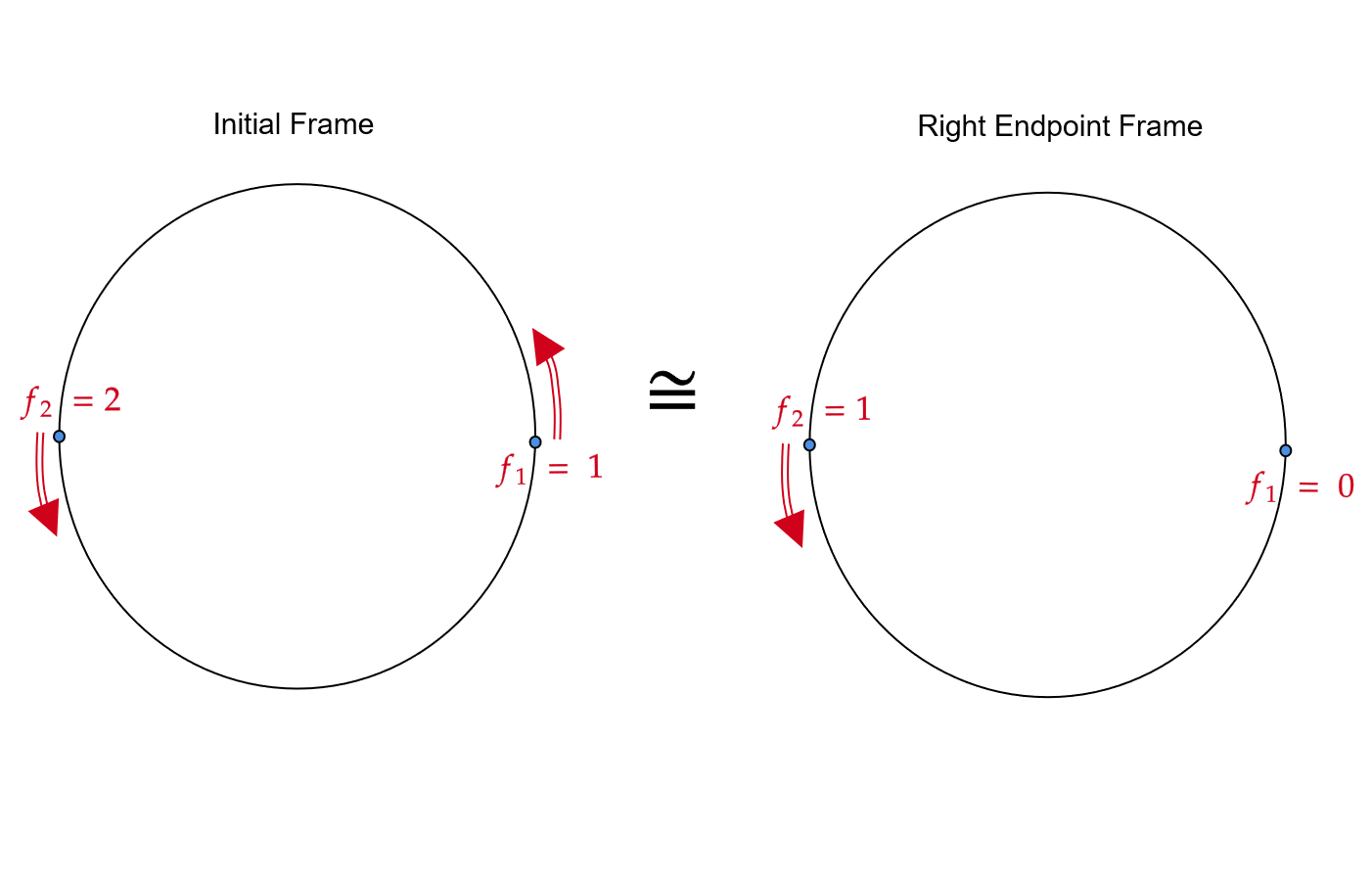}
\caption{}
\label{fig:rightendpointframe}
\end{figure}
We can also take the ``center of momentum frame" and get 
\[
\text{E}_{\eps}([a(t), b(t)]) = \text{E}_{\eps}\left(\left[\frac{a(t) - b(t)}{2}, \frac{b(t) - a(t)}{2}\right]\right)
\]
See picture \ref{fig:comframe}. 
\begin{figure}[h!]
\centering
\includegraphics[scale=0.3]{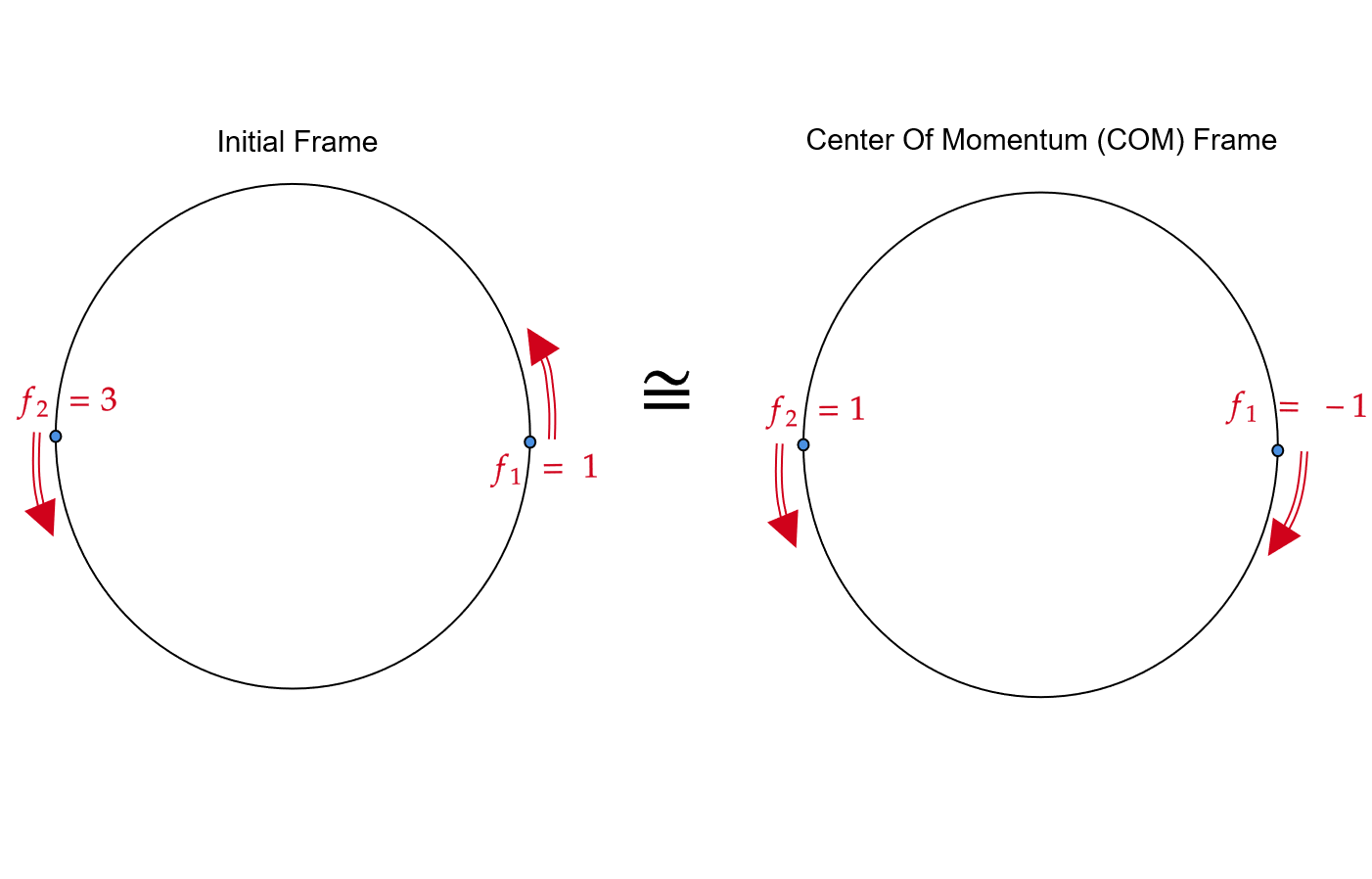}
\caption{}
\label{fig:comframe}
\end{figure}
If we further define
\[
\tilde{f}(0) = \frac{\dot{q}_0 - \dot{q}_1}{2}, \qquad \tilde{f}(1/2p) = \frac{\dot{q}_1 - \dot{q}_0}{2} = - \tilde{f}(0)
\]
Then it is clear that 
\begin{align} \nonumber 
\text{E}_{\eps}([q_0(t), q_1(t)]) &= \text{E}_{\eps}\left(\left[\frac{q_0(t) - q_1(t)}{2}, \frac{q_1(t) - q_0(t)}{2}\right]\right)\\ \nonumber
\frac{d^2}{dt^2} \text{E}_{\eps}([q_0(t), q_1(t)])\Big|_{t = 0} &= \frac{d^2}{dt^2} \text{E}_{\eps}\left(\left[\frac{q_0(t) - q_1(t)}{2}, \frac{q_1(t) - q_0(t)}{2}\right]\right)\Big|_{t = 0} \\ \label{SecondVarEnergyReduction}
&= \eps c \left[\tilde{f}(0) \dot{\tilde{u}}_x(0) + \tilde{f}(1/2p) \dot{\tilde{u}}_x(1/2p) \right]
\end{align}
where instead of the system \eqref{IntervalSystem}, we have
\begin{align} \label{COMIntervalSystem}
\dot{\tilde{u}}:&\left[0, \frac{1}{2p} \right] \to \R \\ \nonumber
\eps^2 \dot{u}_{xx} &= W''(u) \dot{\tilde{u}} \\ \nonumber
\dot{\tilde{u}}\left( 0 \right) &= c \tilde{f}\left(0 \right)\\  \nonumber
\dot{\tilde{u}}_i\left( \frac{1}{2p} \right) &= -c \tilde{f}\left(\frac{1}{2p} \right)  \\ \nonumber
&= c \tilde{f}(0)
\end{align}
In the appendix (\S \ref{NeumannCircleSection}), we show that for the system \eqref{COMIntervalSystem}, we can compute
\begin{equation} \label{DotUNeumannEq}
\dot{\tilde{u}}_x(0) = c \tilde{f}(0) v, \qquad v = v(\eps) < 0
\end{equation}
With this, we can now prove proposition \cref{CircleSecondVarProp} 
\CircleSecondVarProp*
\noindent \Pf Breaking up the energy into all of its intervals and using the center-of-momentum frame, we have 
\begin{align*}
\BE''(t) \Big|_{t = 0} &= \sum_{i = 0}^{2p-1} \frac{d^2}{dt^2}\text{E}_{\eps}\left(\left[q_i(t), q_{i+1}(t)\right] \right) \\
&= \sum_{i = 0}^{2p-1} \frac{d^2}{dt^2}\text{E}_{\eps}\left(\left[\frac{q_{i}(t) - q_{i+1}(t)}{2}, \frac{q_{i+1}(t) - q_{i}(t)}{2}\right] \right)\Big|_{t = 0} \\
& = \eps c \sum_{i = 0}^{2p-1} \frac{\dot{q}_i - \dot{q}_{i+1}}{2}\left[ \dot{\tilde{u}}_{i,x} \left( \frac{i}{2p} \right) - \dot{\tilde{u}}_{i,x}\left( \frac{i+1}{2p}\right) \right]
\end{align*}
from \eqref{SecondVarEnergyReduction}. By equation \eqref{DotUNeumannEq}, and the symmetry of the system \ref{COMIntervalSystem}, we have that 
\[
\dot{\tilde{u}}_{i,x}\left( \frac{i}{2p} \right) = - \dot{\tilde{u}}_{i,x} \left( \frac{i+1}{2p} \right) = c v \left( \frac{\dot{q}_i - \dot{q}_{i+1}}{2} \right)
\]
so that 
\begin{align*}
\BE''(t) \Big|_{t = 0} &= \eps c^2 v(\eps) \sum_{i = 0}^{2p-1} 2\left(\frac{\dot{q}_i - \dot{q}_{i+1}}{2} \right)^2  \\
&= \eps c^2 v(\eps) \sum_{i = 0}^{2p-1} 2\left(\frac{f(i/(2p)) - f((i+1)/2p)}{2} \right)^2
\end{align*}
This finishes the proof of \cref{CircleSecondVarProp}. \qed\nl \nl
The formula for second variation shows: 
\begin{align*}
BE'' &\geq  0 \iff |f(i/(2p)) -  f((i+1)/2p)| = 0 \qquad \forall i = 0, \dots, 2p - 1 
\end{align*}
which geometrically means that our variation either fixes the $D_{2p}$ nodal sets, or rotates all of them at the same rate. This gives the following \textbf{rigidity} result
\begin{corollary}
Any variation of the $u_{2p,\eps}$ solution produces a non-positive second variation of $\BE$, i.e.
\[
\frac{d^2}{dt^2} BE(t) \Big|_{t = 0}  \leq 0
\]
With equality if and only if the variation rotates or fixes all nodes.
\end{corollary}
\noindent As a result, we also prove \cref{S1MorseIndexTheorem} 
\CircleMorseIndexTheorem*
\noindent \Pf Without loss of generality, the space of variations can be chosen to fix $q = 0$ and hence is spanned by $\{f_i\}$ with $f_i$ moving the point $\frac{i}{2p}$ for $i = 1, \dots, 2p-1$. Clearly these are linearly independent and this space is dimension $2p-1$. Moreover, any element of their span is of the form: 
\[
V = \sum_{i = 1}^{2p-1} a_i f_i
\]
and gives a negative variation by equation \eqref{S1SecondVarRevamped}. If one of the $a_i$ are non-trivial, then this is a non-trivial variation among all of the points $\{ \frac{i}{2p} \}_{i = 1}^{2p-1}$ that fixes $q = 0$, and hence is not a rotation of all points. \qed
\subsection{Classification of stable $\BE$ critical points in noncompact $M^3$}
In this section we prove a result analogous to the following theorems of Fischer-Colbrie-Schoen (\cite{fischer1980structure}, Theorem 3)
\begin{theorem}[Fischer-Colbrie-Schoen]
Let $M^3$ be a complete oriented $3$-manifold of non-negative scalar curvature. Let $Y^2$ be an oriented complete stable minimal surface in $M$. Let $S$ denote the scalar curvature of $M$ at any point. Then
\begin{enumerate}
\item If $Y$ compact, then $Y$ is conformally equivalent to $S^2$ or $Y$ is a totally geodesic flat torus $T^2$. If $S > 0$ on $M$, then $Y$ is conformally equivalent to $S^2$. 

\item If $Y$ is not compact, then $Y$ is conformally equivalent to the complex plane $\C$, or the cylinder $A$. If $Y$ is a cylinder, then $Y$ is flat and totally geodesic. If the scalar curvature of $M$ is everywhere positive, then $Y$ cannot be a cylinder with finite total curvature  \label{SecondCondition}
\end{enumerate}
If $\Ric_M \geq 0$, then the assumption of finite total curvature in \ref{SecondCondition} can be removed
\end{theorem}
\noindent Recall that for $M$ non-compact, a closed critical point of $\BE_{\eps}$, $Y$, has first variation $0$, as given by the integral over $Y$ in equation \eqref{SecondVariationFormula}. We prove the following analogous to the first half of \cref{FischerColbrieMimic}
\FischerColbrieMimic*
\noindent \rmk \; The authors hope that the ``almost" descriptor in this theorem be removed and we can get obtain a rigidity statement. The authors also hope to mimic the non-compact part of the original theorem. \nl \nl
\Pf Again using the asymptotic expansion of the second variation formula:
\begin{align*}
BE''(f\nu) \Big|_Y &= \frac{2 \sqrt{2}}{3}\int_Y  [|\nabla f|^2 - \dot{H}_0 f^2] + E \\
&= \frac{2 \sqrt{2}}{3}\int_Y \left[|\nabla f|^2 - \left( S - K + \frac{1}{2} \sum_{i,j = 1}^2 A_{ij}^2 \right)f^2 + E\right] 
\end{align*}
where $||E||_{L^{\infty}} \leq K_0 \eps^{1/2}$ for some $K_0 > 0$ and we've decomposed 
\begin{align*}
\dot{H}_0 &= |A_Y|^2 + \Ric_g(\nu,\nu) \\
&= S - K + \frac{1}{2} \sum_{i,j = 1}^2 A_{ij}^2 
\end{align*}
here $S$ is the ambient scalar curvature of $M$ restricted to $M$ and $K$ is the sectional/Gaussian curvature of $Y \subseteq M$. Using stability, taking $f = 1$, we have 
\begin{align*}
\int_Y K dA &\geq \int_Y \left( S + \frac{1}{2} \sum_{i,j=1}^2 A_{i,j}^2 + E \right) \\
& \geq -K \eps
\end{align*}
Gauss bonnet for a closed manifold now tells us that 
\[
\int_Y K dA = 2 \pi \chi(Y)
\]
Now our lower bound gives $\chi(Y) \in \{0,2\}$. If $Y$ is a torus (i.e. $\chi(Y) = 0$) then we see that 
\begin{align*} 
0 &\geq \int_Y S + \frac{1}{2} \sum_{i,j} A_{ij}^2 + E  \\
\implies 0 & \geq \int_Y \sum_{i,j} A_{ij}^2 + O(\eps) 
\end{align*}
Since $S \geq 0$ everywhere. This tells us that 
\[
||A_Y||_{L^2}^2 \leq K \eps 
\]
i.e. $Y$ is ``almost" geodesic in an $L^2$-sense. If $Y$ is a topological sphere, then uniformization tells us that it is conformally equivalent. \nl \nl
If $M$ is PSC and $Y$ has volume bounded from below by a positive constant, then we see that
\[
\int_Y S + \frac{1}{2} \sum_{i,j} A_{ij}^2 + E > 0
\]
ruling out the torus case. \qed 

\section{Appendix}

%
%
%
%
%
%
%
\subsection{Smooth dependence of solutions on domains} \label{SmoothDependenceDomains}
In this section, we show that for two domains $\Omega$ and $\varphi(\Omega)$, the solutions are continuous under pullback to $\Omega$. \nl \nl   
Throughout, we will assume that $\Omega$ is such that it supports a positive Dirichlet phase transition. Let $k > 2$ be a given integer. Suppose that $\Omega$ is a domain with $C^{k, \alpha}$ boundary (i.e., the inclusion map $i_{\Omega} \in C^{k, \alpha}(\Omega; M)$ is $C^{k, \alpha}$), we take a neighborhood $U \subset C^{k, \alpha}(\Omega; M)$ consisting only of $C^{k, \alpha}$ embeddings of $\Omega$. Suppose that $\varphi \in U$ and let $v$ be the unique positive solution, afforded by theorem \ref{BO}, to the problem
\begin{equation} \label{ACDomainEq}
\begin{cases}
\eps^2 \Delta v = W'(v) & \text{in } \varphi(\Omega),\\
v = 0 & \text{on } \partial \varphi(\Omega).
\end{cases}
\end{equation}
\noindent The $\varphi^{\ast} v : \Omega \to \mathbf{R}$ satisfies $\varphi^{\ast}v = 0$ along $\partial \Omega$, and moreover $\varphi^{\ast}v$ satisfies the PDE

\[
(\varphi^{\ast} \Delta (\varphi^{-1})^{\ast}) \varphi^{\ast}v = \varphi^{\ast} \Delta v = \varphi^{\ast} W'(v) = W'(\varphi^{\ast}v).
\]
\noindent In other words, given $v$ as in \eqref{ACDomainEq}, and possibly shrinking $U$, we have for every $\varphi \in U$ a unique solution to the auxiliary problem
\[
\begin{cases}
\eps^2 (\varphi^{\ast} \Delta (\varphi^{-1})^{\ast}) u = W'(u) & \text{in } \Omega,\\
u = 0 & \text{on } \partial \Omega.
\end{cases}
\]
\noindent Define the operator $L_{\varphi, \eps }: C^{k, \alpha}_{0}(\Omega) \to C^{k - 2, \alpha}(\Omega)$ by $L_{\varphi, \eps}u \coloneqq \eps^2 (\varphi^{\ast} \Delta (\varphi^{-1})^{\ast}) u - W'(u)$. We may now define $\Phi: U \times C^{k, \alpha}_{0}(\Omega) \to C^{k - 2,\alpha}(\Omega)$ by
\[
\Phi(\varphi, u) \coloneqq L_{\varphi, \eps}u.
\]
\begin{proposition} \label{D2Iso}
If $u$ is the positive Dirichlet phase transition on $\Omega$, then the map $D_{2}\Phi(i_{\Omega}, u):C^{k, \alpha}_{0}(\Omega) \to C^{k - 2, \alpha}(\Omega)$ is an isomorphism.
\end{proposition}

\begin{proof}
The map in question is the linearized Allen--Cahn operator. That is, if $f \in C^{k, \alpha}_{0}(\Omega)$
\[
D_{2}\Phi(i_{\Omega}, u)f = \eps^2 \Delta f - W''(u)f.
\]
But this operator is invertible by \cref{NoLinearizedKernel} in the Sobolev setting or \cite{marx2023dirichlet}, Theorem 1.4 in the H\"older setting.
\end{proof}
\noindent After possibly shrinking $U$ once more, the implicit function theorem \cite{lang1985differential} and proposition \ref{D2Iso} together imply that there a unique $C^{k}$ map $\varphi \mapsto u_{\varphi}$ satisfying
\[
\Phi(\varphi, u_{\varphi}) = 0.
\]
\noindent In other words, the phase transitions $u^{+}_{\Omega}$ depend on the domain $\Omega$ in an appropriately smooth way. 
\subsection{Computing $\dot{u}_x(0)$ on $S^1$} \label{NeumannCircleSection}
In this section, we consider 
\begin{align} \label{SimplifiedCircleSystem}
u_{\eps}&: [0,1] \to \R \\ \nonumber
u_{\eps}(0) &= u_{\eps}(1) = 0 \\ \nonumber
\eps^2 u_{\eps,xx} &= W'(u_{\eps}) \\ \nonumber
\dot{u}_{\eps}&: [0,1] \to \R \\ \nonumber
\dot{u}_{\eps}(0) &= \dot{u}_{\eps}(1) = 1 \\ \nonumber
\eps^2 \dot{u}_{\eps,xx} &= W'(u_{\eps})
\end{align}
and the goal is to show 
\begin{theorem}
For $\dot{u}$ as in the system \ref{SimplifiedCircleSystem}, there exists an $\eps_0 > 0$ such that
\[
\forall \eps_0 > \eps > 0, \qquad \text{sgn}(\dot{u}_{\eps,x}(0)) < 0
\]
\end{theorem}
\noindent \rmk \; While we've shown that the geometry of the level set determines the neumann data (see remark \ref{HolderRemark}) when it is at least $1$ dimensional, there is no geometry for a $0$ dimensional level set. Thus, this theorem requires a more careful ODE analysis
\subsubsection{Expanding $u = u_{\lambda}(t)$}
First note that $u$ and $\dot{u}$ as in \ref{SimplifiedCircleSystem} can be extended to functions on $\R$ by odd (resp. even) reflection about the boundary points of $0$ and $1$, and then repeating this at every $n \in \mathbb{Z}$. Note that $\dot{u}$ is $C^0$ but \textit{not} $C^1$ on $\R$ a priori. From here, we define 
\begin{align} \label{CircleACSystem}
\tilde{u}_{\eps} &: \R \to \R \\ \nonumber
\tilde{u}_{\eps}(x) & := u(\eps^{-1} x) \\ \nonumber
\tilde{u}(0) &= \tilde{u}(\eps^{-1}) = 0 \\ \nonumber
\tilde{u}_{\eps, xx} &= W'(\tilde{u}) \\ \nonumber
\dot{\tilde{u}} &: \R \to \R \\ \nonumber
\dot{\tilde{u}}(x) &= \dot{\tilde{u}}(\eps^{-1} x) \\ \nonumber
\dot{\tilde{u}}(0) &= \dot{\tilde{u}}(\eps^{-1}) = 1 \\ \nonumber
\dot{\tilde{u}}_{xx} &= W''(\tilde{u}) \dot{\tilde{u}} 
\end{align}
We first establish a lemma about the shape of $\tilde{u}_{\eps}$:
\begin{lemma} \label{ACSolutionShapeLemma}
Let $\tilde{u} = \tilde{u}_{\eps}$ be a solution as in \eqref{CircleACSystem}. Then on $[0, \eps^{-1}]$, $\tilde{u}$ is concave down with a unique critical point and maximum at $x = \eps^{-1}/2$
\end{lemma}
\noindent $\Pf$ The fact that $\tilde{u}$ is concave down follows because $0 < \tilde{u} < 1$ on $(0, \eps^{-1})$ by the maximum principle compared to $v = 1$. By considering the symmetric reflection of $u$ about $x = \eps^{-1}/2$, we see that $\tilde{u}$ is even about $x = \eps^{-1}/2$ due to its uniqueness from energy minimization, so clearly there is a local max there by concavity. To see that there are no other critical points, we use concavity. If there exists $x^* \in (0, \eps^{-1}/2)$, a critical point, then by concavity, it is another local max. This implies there exists a local min $x^{**} \in (x^*, \eps^{-1}/2)$, so that $\tilde{u}_{xx}(x^{**}) \geq 0$ at that point, a contradiction to the Allen--Cahn equation and bounds on $\tilde{u}$. Thus the only critical point is the maximum that $\tilde{u}$ achieves at $x = \eps^{-1}/2$. \qed \nl \nl
By multiplying \eqref{ACEquation} by $\tilde{u}_x$ and integrating, every solution with $\eps = 1$ on $\R$ satisfies 
\begin{equation} \label{EnergyConservation}
\frac{\tilde{u}_x^2}{2} = W(\tilde{u}) - \lambda
\end{equation}
for some $\lambda \in [0,1/4]$. We now show that for all $\eps < \eps_0$ sufficiently small, there is a one to one correspondence between $\eps$ and $\lambda$. 
We now recall continuity of the $\{ \tilde{u}_{\eps}\}$ with respect to $\eps$: 
\begin{lemma} \label{epsUContLemma}
The family $\{\tilde{u}\} = \{\tilde{u}_{\eps}\}$ are continuous with respect to $\eps$, and in fact 
\begin{equation} \label{ContinuityLemmaEq}
\lim_{\eps \to \eps^*} ||\tilde{u}_{\eps} - \tilde{u}_{\eps^*}||_{C^0(\R)} = 0
\end{equation}
for all $\eps^* < \eps_0$ postive with $\eps_0$ sufficiently small.
\end{lemma}
\noindent \Pf We present $2$ proofs. First, choose $\eps_0$ sufficiently small so that $\tilde{u}_{\eps}$ is non-zero and periodic for all $\eps < \eps_0$. Suppose equation \eqref{ContinuityLemmaEq} did not hold for some $\eps^* < \eps_0$. In particular, there exists a $\delta > 0$ and (by periodicity) a sequence of points and scales, $\{x_i, \eps_i\}$ with $\eps_i^{-1} < 2 (\eps^*)^{-1}$ and $x_i \in [0, 2 (\eps^*)^{-1}]$ such that 
\begin{equation} \label{pointpickcounter}
|\tilde{u}_{\eps_i}(x_i) - \tilde{u}_{\eps^*}(x_i)| > \delta
\end{equation}
However, by the Allen--Cahn equation, we have uniform $C^{2}$, and hence $C^{1,\alpha}$ estimates on $\{\tilde{u}_{\eps_i}\}$ up to a subsequence. This gives convergence in $C^1$ to $\overline{u}$ along a subsequence. Note that since $\eps_i \to \eps^*$, we also have that 
\[
0 = \tilde{u}_{\eps_i}(\eps_i^{-1}) \rightarrow \overline{u}((\eps^*)^{-1})
\]
so $\overline{u}$ vanishes at $x = (\eps^*)^{-1}$. By lemma \ref{OneDimPeriodicClassification}, it is periodic with period $(\eps^*)^{-1}$. But now by using the uniqueness of periodic solutions to Allen--Cahn on $\R$ with respect to their period (a consequence of rescaling back to $[0,1]$ and uniqueness of energy minimizers on $[0,1]$), we have that $\overline{u} = \tilde{u}_{\eps^*}$, contradicting \eqref{pointpickcounter}. \qed \nl \nl
\noindent \textbf{2nd Proof}: from \S \ref{EllipticFunctions}, we know that 
\begin{equation} \label{EllipticFunctionSolution}
\tilde{u}_{\eps} = g(x, k(\eps)) = k(\eps) \sqrt{\frac{2}{1 + k(\eps)^2}} \sn\left( \frac{x}{\sqrt{1 + k(\eps)^2}} \; \Big| \; k(\eps) \right)
\end{equation}
where $\sn(x \; | \; k)$ is the Jacobi elliptic integral as defined in \cite{byrd2013handbook}. This function is a solution to Allen--Cahn with fundamental period 
\[
\tau = 2 K'(k) = 2 \int_0^{\pi/2} \frac{d \theta}{\sqrt{1 - (1 - k^2) \sin^2 \theta}}
\]
setting this equal to $\eps^{-1}$ and differentiating, one can see that $K'$ is locally invertible for $k < 1$, corresponding to $\eps > 0$. Now $\tilde{u}_{\eps}$ is a continuous function of $k(\eps)$ and hence of $\eps$, so the uniform convergence holds by compactness and periodicity. \qed \nl \nl
Now we prove a correspondence between $\lambda$ and $\eps$. First note that by evaluating at $x = \eps^{-1}/2$, we have that 
\begin{equation} \label{EpsLambdaMap}
\lambda = W(\tilde{u}_{\eps}(\eps^{-1}/2)) = W(||\tilde{u}_{\eps}||_{C^0})
\end{equation}
\begin{proposition} \label{EpsLambdaCorrespondenceProp}
The map from $\eps \to \lambda$ is injective and continuous (and hence monotone) for all $\eps < \eps_0$ sufficiently small. It is surjective from $(0, \eps_0)$ onto $(0, \delta_0)$ for some $\delta_0 = \delta_0(\eps_0)$
\end{proposition}
\noindent \Pf \; The map from $\eps \to \lambda$ is clearly defined via equation \eqref{EpsLambdaMap}. To show that it is injective, suppose $\tilde{u}_{\eps_1}$ and $\tilde{u}_{\eps_2}$ had the same $\lambda$. Then looking at $x = 0$ and \eqref{EnergyConservation}, we have 
\[
\tilde{u}_{\eps_i}(0) = 0 \implies \tilde{u}_{\eps_i, x} = \sqrt{\frac{1}{4} - \lambda}
\]
i.e. $\tilde{u}_{\eps_1}, \tilde{u}_{\eps_2}$ have the same dirichlet and neumann data at $x = 0$, and they're both solutions to Allen--Cahn with $\eps = 1$ on $\R$. By ODE theory, we have $\tilde{u}_{\eps_1} = \tilde{u}_{\eps_2}$, which by period considerations means that $\eps_1 = \eps_2$. Continuity of this map follows from lemma \ref{epsUContLemma} since 
\[
|\lambda(\eps_1) - \lambda(\eps_2)| = |W(\tilde{u}_{\eps_1}(\eps_1^{-1}/2) - W(\tilde{u}_{\eps_2}(\eps_2^{-1}/2)|
\]
And $\tilde{u}_{\eps_1}(\eps_1^{-1}/2) \to \tilde{u}_{\eps_2}(\eps_2^{-1}/2)$ as $\eps_1 \to \eps_2$ by lemma \ref{epsUContLemma}. To see surjectivity, note that $\lim_{\eps \to 0} ||\tilde{u}_{\eps}||_{C^0(\R)} = 1$, which can be seen by convergence of the family of functions to the heteroclinic on $\R$. But then 
\[
\lim_{\eps \to 0} \lambda(\eps) = \lim_{\eps \to 0} W(\tilde{u}_{\eps}(\eps^{-1}/2) = \lim_{\eps \to 0} W(||\tilde{u}_{\eps}||_{C^0(\R)}) = 0
\]
since $||\tilde{u}_{\eps}|| = \tilde{u}_{\eps}(\eps^{-1}/2)$. By monotonicity (from injectivity and continuity), we get the surjectivity result. \qed \nl \nl
%
%
We can now call each $\tilde{u}_{\eps}: \R \to \R$ a $u_{\lambda(\eps)} = u_{\lambda}: \R \to \R$, and similarly $\dot{\tilde{u}}_{\eps} = \dot{\tilde{u}}_{\lambda}$
\subsubsection{Expansion in $\lambda$}
In this section, we remove the $\tilde{\cdot}$ notation for convenience. From equation \eqref{EnergyConservation}, we have
\[
\int_0^{u_{\lambda}(t)} \frac{dx}{\sqrt{W(x) - \lambda}} = t
\]
we can differentiate \textbf{with respect to $\lambda$} and get 
\[
(\partial_{\lambda} u_{\lambda})(t) \frac{1}{\sqrt{W(u_{\lambda}) - \lambda}} = - \frac{1}{2}\int_0^{u_{\lambda}(t)} \frac{dx}{[W(x) - \lambda]^{3/2}}
\]
evaluating at $\lambda = 0$, we have that 
\begin{align*}
(\partial_{\lambda} u_{\lambda}|_{\lambda = 0})(t) &= - 4(1 - g^2) \int_0^g \frac{dx}{(1 - x^2)^3} \\
&= - \frac{1 - g^2}{4} \left[\frac{2g(5 - 3g^2)}{(g^2 - 1)^2} + 3 \log\left( \frac{1 + g}{1 - g} \right) \right] \\
&=: \kappa(t)
\end{align*}
where $g = g(t)$ is the heteroclinic and we've used that 
\[
\int \frac{dx }{(1  -x^2)^3} = \frac{1}{16} \left(\frac{10x - 6x^3}{(1 - x^2)^2}  +3 \log\left(\frac{1 + x}{1 - x}\right)\right)
\]
Note that $\kappa(t) < 0$ for all $t > 0$. Let $L = \partial_t^2 - W''(g)$. We now want to solve
\[
L(\tau(t)) = \kappa(t)
\]
For 
\[
u_{\lambda}(t) = g(t) + \lambda \kappa(t) + \phi
\]
We compute 
\begin{align*}
u_{\lambda,tt} &= \ddot{g} + \lambda \ddot{\kappa} + \phi_{tt} \\
W'(u_{\lambda}) &= W'(g) + W''(g)[\lambda \kappa + \phi] + O(\lambda^2, \phi^2, \lambda \phi) \\
u_{\lambda, tt} - W'(u_{\lambda}) &= \phi_{tt} - W''(g) \phi + O(\lambda^2, \phi^2, \lambda \phi) \\
L(\phi) &= O(\lambda^2, \phi^2, \lambda \phi)
\end{align*}
%
%
\noindent We now show:
\begin{lemma} \label{1DEstimate}
For $\delta = \delta(\lambda)$ sufficiently small, we have 
\[
||\phi||_{C^{2,\alpha}([0, \delta))} \leq K \lambda^2
\]
for $K$ independent of $\lambda$ (and hence $\eps$)
\end{lemma}
\noindent \Pf Note that $\phi(0) = 0$ by nature of the Dirichlet condition on $u_{\lambda}(t)$ and definition of $\kappa(t)$. Let $\delta = \delta(\lambda)$ be such that 
\[
||\phi||_{C^0([0, \delta])} \leq \lambda^2
\]
Then we have that 
\begin{align*}
\phi_{tt} &= W''(g) \phi + (3g) [\lambda^2 \kappa^2 + 2 \lambda \kappa \phi + \phi^2] \\
& + [\lambda^3 \kappa^3 + 3 \lambda^2 \kappa^2 \phi + 3 \lambda \kappa \phi^2 + \phi^3] \\
||\phi_{tt}||_{C^0} & \leq 2 \lambda^2 + 3[\lambda^2 K_0^2 + 2 \lambda^3 K_0 + \lambda^4] + [\lambda^3 K_0^3 + 3 K_0^2 \lambda^4 + 3 K_0 \lambda^5 + \lambda^6]
\end{align*}
where
\[
K_0 = ||\kappa(t)||_{C^0(\R^+)}
\]
Then for $\lambda$ sufficiently small, we see that
\[
||\phi_{tt}||_{C^0([0, \delta])} \leq 3(1 + K_0)^2 \lambda^2
\]
and so by interpolation
\[
||\phi||_{C^2([0, \delta])} \leq 3(1 + K_0)^2 \lambda^2
\]
\qed \nl
%
%
\noindent Similarly, we now want to show there's a next order expansion for $\dot{u}$ on $[0, \delta(\lambda)]$ at least. We write 
\begin{align*}
W''(u) &= 3 (g(t) + \lambda \tau(t) + \phi)^2 - 1 \\
&= 3g^2 + 6\lambda g(t) \tau(t) - 1 + O(\lambda^2) \\
&= W''(g) + 6 \lambda g(t) \tau(t) + O(\lambda^2) \\
f(t) &:= 6 g(t) \tau(t)
\end{align*}
%
where the $O(\lambda^2)$ error terms hold in $C^2([0, \delta))$. Take $\omega(t)$ the solution to 
\begin{align*}
\omega&: [0, \infty) \to \R \\
\omega(0) &= \lim_{t \to \infty} \omega(t) = 0 \\
L(\omega(t)) &= f(t) \dot{g}(t)
\end{align*}
and expand
\[
\dot{u} = \dot{g}(t) + \lambda \omega(t) + \varphi
\]
Using our prior 
\begin{align*}
[\partial_t^2 - W''(u)] \dot{u} &= [\partial_t^2 - W''(g) - \lambda f(t) + O(\lambda^2)] [ \dot{g} + \lambda \omega + \varphi] \\
&= L(\lambda \omega) +  L(\varphi) - \lambda f \dot{g} + O(\lambda \varphi, \lambda^2) \\
&= 0 \\
\implies L(\varphi) &= O(\lambda \varphi, \lambda^2)
\end{align*}
By the same argument as in \ref{1DEstimate}, we have 
\[
||\varphi||_{C^{2}([0, \delta^*))} \leq K \lambda^2
\]
for some $\delta^* = \delta^*(\lambda)$ and $K$ independent of $\lambda$. Furthermore, it is clear that 
\begin{align*}
\dot{u}_t(0) &= \ddot{g}(0) + \lambda \omega_t(0) + \varphi_t(0) \\
&= \lambda \omega_t(0) + O(\lambda^2)
\end{align*}
so 
\[
\text{sign}(\dot{u}_t(0)) = \text{sign}(\omega_t(0))
\]
and it suffices to compute $\omega_t(0)$.
\subsubsection{Solving for $\omega_t(0)$}
Recall that one can solve for $\omega(t) = v(t) \dot{g}(t)$ as in (\cite{marx2023dirichlet}, \S 7.6 or \cite{mantoulidis2022variational}), so that 
\begin{align*}
w'(0) &= v'(0) \dot{g}(0) \\
v'(0) &= \frac{a_0}{\dot{g}(0)^2} \\
a_0 &= - \int_0^{\infty} (f \dot{g}) \dot{g}
\end{align*}
We check the sign of $f$. Write
\begin{align*}
f(t) &= 6 g(t) \tau(t) \\
\tau(t) & \; \st \; L(\tau) = \kappa(t) \\
\kappa(t) &= - \frac{1 - g^2}{4} \left[ \frac{2g(5 - 3g^2)}{(1 - g^2)^2} + 3 \log\left( \frac{1 + g}{1 - g} \right) \right]
\end{align*}
So $\kappa < 0$ for all $t > 0$ and vanishes at $t = 0$. Recall that 
\begin{align*}
\tau(t) &= r(t) \dot{g}(t) \\
r(t) & = \int_0^t \dot{g}(s)^{-2} \left[ a_0 + \int_0^s \kappa(u) \dot{g}(u) du \right] ds \\
a_0 &= - \int_0^{\infty} \kappa(u) \dot{g}(u) du \\
& >0 \\
\implies r(t) & > 0 \qquad \forall t > 0 \\
\implies \tau(t) & > 0 \qquad \forall t > 0 \\
\implies f(t) & > 0 \qquad \forall t > 0 \\
\implies w'(0) & < 0
\end{align*}
i.e. 
\begin{align*}
\dot{u} &= \dot{g}(t) + \lambda \omega(t) + \varphi(t) \\
\dot{u}'(0) &= \lambda w'(0) + O(\lambda^2) \\
& < 0
\end{align*}
for all $\lambda$, and hence all $\eps$, sufficiently small. This is our desired result.
\subsection{Elliptic Functions} \label{EllipticFunctions}
This section provides a brief introduction to elliptic functions, which describe the solutions to Allen-Cahn on $\R$. Let 
\[
u(y_1, k) = \int_0^{y_1} \frac{dt}{\sqrt{(1 - t^2)(1 - k^2 t^2)}} = \int_0^{\sin^{-1}(y_1)} \frac{d \theta}{\sqrt{1 - k^2 \sin^2 \theta}}  = F(\sin^{-1}(y_1), k)
\]
be the incomplete elliptic integral of the first kind (see \cite{byrd2013handbook}), and let
\[
y_1 = \sn(u,k)
\]
be the $\sin$ amplitude function. From hereon, we'll notate the right as $\sn(x,k)$ and think of $\{\sn(\cdot, k)\}$ as a family of functions on $\R$. \nl \nl
We also recall the definitions and relations
\begin{align*}
\cn(x,k) &= \cos(\sin^{-1}(y_1)) =\sqrt{1 - \sn(x,k)^2} \\
\dn(x,k) &= \sqrt{1 - k^2 \sn(x,k)^2} \\
\sn_x &= \cn \cdot \dn \\
\cn_x &= - \sn \cdot \dn \\
\dn_x &= - k^2 \sn \cdot \cn
\end{align*}
One can compute
\[
\sn_{xx}(x,k) = -(1 + k^2) \sn + 2k^2 \sn^3
\]
and define
\begin{equation}\label{ACFamilyEquation}
g(x,k): = k \sqrt{\frac{2}{1 + k^2}} \sn\left(\frac{x}{\sqrt{1 + k^2}} \; \Big| \; k \right)
\end{equation}
Each of the $g(\cdot, k)$ are solutions to \eqref{ACEquation} and from \cite{byrd2013handbook} we see that $\sn(u, k) = \sn(u)$ has period equal to $4 K(k) = 4 F(\pi/2, k) $ which tends to $\infty$ as $k \to 1$. We know that any $\tilde{u}_{\eps}: \R \to \R$, a solution to \eqref{ACEquation} with period $\tau = \eps^{-1}$, is of the form $\tilde{u}_{\eps}(x) = g(x,k)$ for $k$ so that
\[
\eps^{-1} = 4K(k) \sqrt{1 + k^2}
\]
From the integral expression for $K$, it is clear that as $k \to 1$, $\eps^{-1}(k)$ is monotone and increasing to infinity and moreover differentiable so we can write 
\[
k = F(\eps)
\]
for some $\eps$. We are interested in the regime of when $k = 1$ since
\begin{align*}
u(y_1, 1) &= \int_0^{y_1} \frac{dt}{1 - t^2} = \tanh^{-1}(y_1) \\
y_1 &= \tanh(u) = \sn(u, 1)
\end{align*}
and so 
\[
g(x,k=1) = \sn(x/\sqrt{2}, k= 1) = \tanh(x/\sqrt{2}) = g(x)
\]
the heteroclinic. 
\subsection{Reducing $\BE''(f, f)$} \label{ReducingBE}
In this section, we first prove
\begin{lemma}
For $\phi: M \to \R$ as in section \S \ref{ComputeBound}, we have
\begin{align*}
-\int_Y f \phi_t &= \sigma^{-1} \int_Y f \int_0^{-2\omega \eps \ln(\eps)} - \Delta_t(h) \dot{\bg}_{\eps}^2 + \eps^{-1} H_t h \ddot{\bg}_{\eps} \dot{\bg}_{\eps} + \eps^{-2}[W''(u) - W''(\bg_{\eps})] h \dot{\bg}_{\eps}^2 \\
& \qquad \qquad \qquad - [\Delta_t(\phi) - H_t  \phi_t + \eps^{-2}(W''(u) - W''(\bg_{\eps}) + R) \phi] \dot{\bg}_{\eps} \Big) \sqrt{\det g(s,t)} dt ds
\end{align*}
where 
\[
\int_Y \int_0^{-2 \omega \eps \ln(\eps)} f R \phi \dot{\bg}_{\eps} \sqrt{\det g(s,t)} dt s \leq K \eps^{7/2} ||f||_{H^1(Y)}^2
\]
\end{lemma}
\noindent \Pf We first note that by integration by parts
\begin{equation} \label{IBPEq1}
-\phi_t \sqrt{\det g(s,0)} = \sigma^{-1} \int_0^{-2\omega \eps \ln(\eps)} (\partial_t^2 \phi) \dot{\bg}_{\eps} \sqrt{\det g(s,t)} - \sigma^{-1} \int_0^{-2\omega \eps \ln(\eps)} \phi \partial_t^2(\dot{\bg}_{\eps} \sqrt{\det g(s,t)}) dt
\end{equation}
for $\sigma = \sqrt{2}^{-1}$ as in \eqref{constants} having noted that 
\[
\phi(s,0) = 0 = \partial_t^k(\bg_{\eps})(-2 \omega \eps \ln(\eps)) \qquad \forall k \geq 1
\]
so that all of the remaining boundary terms vanish in \eqref{IBPEq1}. This tells us that 
\begin{align*} 
-\int_Y f \phi_t \sqrt{\det g(s,0)} &= \sigma^{-1} \int_Y f\int_0^{-2\omega \eps \ln(\eps)} \left[ (\partial_t^2 \phi) \dot{\bg}_{\eps} \sqrt{\det g(s,t)} - \phi \partial_t^2(\dot{\bg}_{\eps} \sqrt{\det g(s,t)})\right] dt \\
&= \sigma^{-1} \eps^{-2} \int_Y f \int_0^{-2\omega \eps \ln(\eps)} [\eps^2\partial_t^2 - W''(\bg_{\eps}) + E](\phi) \dot{\bg}_{\eps} \sqrt{\det g(s,t)} - \phi [2\eps\ddot{\bg}_{\eps}  \partial_t\sqrt{\det g(s,t)} \\
& \qquad \qquad - \eps^2 \dot{\bg}_{\eps}(t) \partial_t^2 \sqrt{\det g(s,t)}] - E \phi \sqrt{\det g(s,t)} dt ds \\
&= \sigma^{-1} \eps^{-2} \int_Y f\int_0^{-2 \omega \eps \ln(\eps)} \Leu(\phi) \dot{\bg}_{\eps} \sqrt{\det g(s,t)} \\
& \qquad \qquad - \phi \left(E +  2 \eps \ddot{\bg}_{\eps} \partial_t \sqrt{\det g(s,t)} + \eps^2 \dot{\bg}_{\eps}(t) \partial_t^2 \sqrt{\det g(s,t)}\right) \\
&\qquad \qquad  + [- \eps^2 \Delta_t(\phi) + \eps^2 H_t \phi_t + [W''(u) - W''(\bg_{\eps}) ]\phi ]\dot{\bg}_{\eps} \sqrt{\det g(s,t)} dt ds \\
&= \sigma^{-1} \eps^{-2} \int_Y f \int_0^{-2\omega \eps \ln(\eps)} -\Leu(h \dot{\bg}_{\eps}) \dot{\bg}_{\eps} \sqrt{\det g(s,t)} \\
&\qquad \qquad  + [- \eps^2 \Delta_t(\phi) + H_t \phi_t + [W''(u) - W''(\bg_{\eps}) + R] \phi]\dot{\bg}_{\eps} \sqrt{\det g(s,t)} dt ds \\
&= \sigma^{-1} \int_Y f \int_0^{-2\omega \eps \ln(\eps)} - \Delta_t(h) \dot{\bg}_{\eps}^2 + \eps^{-1} H_t h \ddot{\bg}_{\eps} \dot{\bg}_{\eps} + \eps^{-2}[W''(u) - W''(\bg_{\eps})] h \dot{\bg}_{\eps}^2 \\
& \qquad \qquad \qquad - [\Delta_t(\phi) - \eps^{-2} H_t  \phi_t + \eps^{-2}(W''(u) - W''(\bg_{\eps}) + R) \phi] \dot{\bg}_{\eps} \Big) \sqrt{\det g(s,t)} dt ds
\end{align*}
for $E$ as in \eqref{ApproxHeteroclinic} and having used $\Leu(\phi) = - \Leu(h \dot{\bg}_{\eps})$. Here 
\[
R = -[E + 2 \eps \ddot{\bg}_{\eps} (\partial_t \sqrt{\det g(s,t)}) + \eps^2 \dot{\bg}_{\eps} (\partial_t^2 \sqrt{\det g(s,t)})]
\]
In particular, note that 
\begin{equation} \label{RC0Bound}
|R| \leq K \eps^2 
\end{equation}
holds pointwise. This follows because $\partial_t \sqrt{\det g(s,t)} = H_Y + O(t)$ so that $|\ddot{\bg}_{\eps} \partial_t \sqrt{\det g(s,t)}| = O(\eps)$ because $H_Y = O(\eps)$ by nature of being a critical point for $\BE_{\eps}$. Now using  \ref{phiBoundByfLemma}, we get
\begin{align*}
\Big|\int_Y f \int_0^{-2 \omega \eps \ln(\eps)} R \phi \dot{\bg}_{\eps} \sqrt{\det g(s,t)} dt ds \Big| &\leq ||f \dot{\bg}_{\eps} R||_{L^2(M)} ||\phi||_{L^2(M)} \\
& \leq K \eps^{5/2}  ||f||_{L^2(Y)} ||\phi||_{L^2(M)} \\
& K \eps^{7/2} ||f||_{H^1(Y)}^2
\end{align*}
by equation \eqref{BestPhiH1Bound} and hence verifying equation \eqref{ErrorIntegralBound}. This ends the proof of the lemma. \qed \nl \nl
\noindent We now verify equation \eqref{halfIntegralSecondVar}:
\[
\int_Y f u_{\nu} \dot{u}_{\nu} = \int_Y f u_{\nu} \phi_t = \frac{\sqrt{2}}{3} \int_Y  (|\nabla f|^2 - \dot{H}_Y f^2 ) dA_Y + E 
\]
where $|E| \leq K \eps^{1/2} ||f||_{H^1(Y)}$. We analyze the components in our expression for $\int f \phi_t$:
\[
\int f \phi_t = I_1 + I_2 + I_3 + I_4 + I_5 + I_6
\]
In the below we often employ proposition \ref{NoLinearizedKernel}. We'll first bound the lower order error terms:
\begin{align*}
I_4&= - \sigma^{-1} \int_0^{-2\omega \eps \ln(\eps)} \dot{\bg}_{\eps} \int_{Y} \Delta_t(\phi) f dA_{Y_t} \\
&= \sigma^{-1} \int_0^{-2\omega \eps \ln(\eps)} \dot{\bg}_{\eps} \int_{Y_t} g(\nabla^t \phi, \nabla^t f)  \\
& =  \sigma^{-1}\int_M g(\nabla^t \phi, - \dot{\bg}_{\eps} \nabla^t f) \\
|I_4| & \leq K \eps^{1/2} ||f||_{H^1(Y)}^2
\end{align*}
%
having used \eqref{BestPhiH1Bound} in the last line. We can also bound
\begin{align*}
I_5&= -\sigma^{-1} \int_{Y} f \int_0^{-2\omega \eps \ln(\eps)} H_t \phi_t \dot{\bg}_{\eps} \sqrt{\det g(s,t)} dA_{Y_t} dt \\
&= - \sigma^{-1} \int_M \phi_t f H_t \dot{\bg}_{\eps} dA \\
|I_5| & \leq K \eps^{3/2} ||f||_{L^2} ||\nabla \phi||_{L^2} \\
&\leq K \eps^{3/2} ||f||_{H^1(Y)}^2
\end{align*}
We also have
\begin{align*}
I_6 &= \sigma^{-1} \int_Y f \int_0^{-2\omega \eps \ln(\eps)} \phi [W''(u) - W''(\bg_{\eps}) + R] \dot{\bg}_{\eps}  \sqrt{\det g(s,t)} dt ds	\\
&= \sigma^{-1} \int_M \phi [W''(u) - W''(\bg_{\eps}) + R] f \dot{\bg}_{\eps} dA \\
|I_6| & \leq K ||\phi||_{L^2(M)} \Big| \Big| [W''(u) - W''(\bg_{\eps}) + R] f \dot{\bg}_{\eps} \Big| \Big|_{L^2(M)} \\
& \leq K \eps^{7/2} ||f||_{H^1(Y)}^2
\end{align*}
from equation \eqref{BestPhiH1Bound} and our $C^0$ bound, equation \eqref{RC0Bound}, on $R$. \nl \nl
For the remaining terms, we get something on the order of $O(1)$:
\begin{align*}
I_1 &= \sigma^{-1} \int_{Y} \int_0^{-2\omega \eps \ln(\eps)} f \Delta_t(h) \dot{\bg}_{\eps}^2 \sqrt{\det g(s,t)} ds dt \\
&= -\sigma^{-1}\int_{Y}  \int_0^{-2\omega \eps \ln(\eps)} g(\nabla^t f, \nabla^t h) \dot{\bg}_{\eps}^2 \sqrt{\det g(s,t)} ds dt \\
&= \eps^{-1} (1 + O(\eps)) \int_{Y}  \int_0^{-2\omega \eps \ln(\eps)} |\nabla^t f|^2 \dot{\bg}_{\eps}^2 \sqrt{\det g(s,t)} ds dt \\
&= \sigma_0(1 + O(\eps)) \int_Y  |\nabla f|^2
\end{align*}
having used the expansion for $h$ as in \eqref{hExpansion} and also noted that $|\nabla^t f|^2 = |\nabla^Y f|^2 + E$ where $|E| \leq t |\nabla^Y f|^2$. We further used $\sigma_0, \; \sigma$ as in \eqref{constants}. We also have 
\begin{align*}
I_2 &= -\sigma^{-1} \eps^{-1}\int_Y \int_0^{-2\omega \eps \ln(\eps)} H_t f h \ddot{\bg}_{\eps} \dot{\bg}_{\eps} \sqrt{\det g(s,t)} dt ds\\
&= \eps^{-2}(1 + O(\eps)) \Big[\int_Y f^2 \eps \dot{H}_Y \left(\int_0^{-2\omega \eps \ln(\eps)} (t/\eps) \dot{\bg}_{\eps} \ddot{\bg}_{\eps}  \sqrt{\det g(s,t)} dt ds + O(\eps^2) \right) \Big]\\
&= [ \sigma_1 + O(\eps)] \int_Y f^2 \dot{H}_Y \sqrt{\det g(s,0)} ds \\
\sigma_1 &:= \int_0^{\infty} t \dot{g}(t) \ddot{g}(t) dt
\end{align*}
having used $H_Y = O(\eps^2)$ as in \ref{MCCritExpansion}. We  also have 
\begin{align*}
I_3 &= \eps^{-3} (1 + O(\eps))  \int_Y f^2 \int_{0}^{-2\omega \eps \ln(\eps)} [W''(u) - W''(\bg_{\eps})] \dot{\bg}_{\eps}^2 \sqrt{\det g(s,t)} dt ds \\
&= 6 \eps^{-1} \int_Y f^2 \int_0^{-2\omega \eps \ln(\eps)} [\dot{H}_Y \btau_{\eps} \bg_{\eps} + O(\eps)]\dot{\bg}_{\eps}^2 \sqrt{\det g(s,t)} dt ds \\
&= \sigma_2\int_Y f^2 [\dot{H}_Y  + O(\eps)] \sqrt{\det g(s,0)} ds \\
\sigma_2 &:= 6 \int_0^{\infty}  \tau g \dot{g}^2 dt 
\end{align*}
using the expansion in equation \eqref{SecondDerivDiffExpansion}. In sum
\begin{equation} \label{AlmostJacobiEq}
\int_Y f \phi_t dA = \int_Y \left(\sigma_0 (1 + O(\eps)) |\nabla f|^2 + [(\sigma_1 + \sigma_2)\dot{H}_Y + O(\eps)] f^2 \right) dA_Y
\end{equation}
We now compute the constants
\subsubsection{Computing the constants}
Recall that
\[
\sigma_0 = \int_0^{\infty} \dot{g}^2 dt = \frac{\sqrt{2}}{3}
\]
\noindent We compute $\sigma_1$ 
\begin{align*}
\sigma_1 &= \int_0^{\infty} t \dot{g} \ddot{g} = - \frac{1}{3 \sqrt{2}} 
\end{align*}
\noindent For $\sigma_2$, we now compute
\[
\int_0^{\infty} \tau g \dot{g}^2
\]
To do this, we note that we can solve
\begin{align*}
\alpha &: [0, \infty) \to \R	 \\
\alpha(0) &= \lim_{t \to \infty} \mu(t) = 0 \\
[\partial_t^2 - W''(g)] \alpha &= g \dot{g}^2
\end{align*}
In particular,
\[
\alpha = - \frac{1}{3\sqrt{2}} g(t) \dot{g}(t)
\]
Moreover, $L = \partial_t^2 - W''(g)$ is self-adjoint on the space of smooth functions on $\R^+$ with exponential decay and dirichlet condition at $0$ and $\infty$. This means that 
\begin{align*}
\int_0^{\infty} \tau g \dot{g}^2 &= \int \tau L(\alpha) \\
&= -\frac{1}{3\sqrt{2}} \int L(\tau) g \dot{g} \\
&= - \frac{1}{3 \sqrt{2}} \int t g \dot{g}^2 \\
&= -\frac{1}{18 \sqrt{2}}
\end{align*}
so that 
\begin{align*}
\sigma_2 &= 6 \int_0^{\infty} \tau g \dot{g}^2 =  - \frac{1}{3\sqrt{2}}
\end{align*}
and so 
\[
\sigma_1 + \sigma_2 = -\frac{\sqrt{2}}{3} = -\sigma_0
\]
and so we conclude that \eqref{AlmostJacobiEq} becomes
\begin{align} \label{PartialSecondVarEq}
\int_Y f \phi_t dA &= \sigma_0 \int_Y \left((1 + O(\eps)) |\nabla f|^2 + \left([-1 + O(\eps)]\dot{H}_Y + O(\eps) \right) f^2 \right) dA_Y \\ \nonumber
&= \sigma_0 [D^2A(f) \Big|_Y + R] \\ \nonumber
|R| &\leq K \eps^{1/2} ||f||_{H^1}^2
\end{align}
Note that this expression is quadratic in $f$ and that $\dot{H}_Y$ does not depend on the choice of normal. We see that we'll get the same contribution from $-\int_Y f \phi_t^-$, meaning that the second variation of $\BE_{\eps}$ will be twice of \eqref{PartialSecondVarEq}. \nl \nl
Thus if $Y$ is a critical point, the second variation of $\BE$ corresponds to the second variation of area plus an $O(\eps)$ error term. 
\subsection{Absolute Minimizers of $\BE_{\eps}$} \label{NonExistenceSection}
\subsubsection{$M^n$ compact for $n \geq 2$}
\noindent In this section, we show that absolute minimizers of $\BE$ do not exist on closed riemannian manifolds with $n \geq 2$:
\begin{restatable}{thmm}{NoAbsoluteMinimizergeqTwo}\label{NoAbsoluteMinimizergeq2}
For $M^n$ a compact riemannian manifold with $n \geq 2$, and any $\eps > 0$ 
\[
\inf_{\substack{Y \; \text{separating} \\ \text{hypersurface}}} \BE_{\eps}(Y) = 0
\]
in particular, there  is no $Y^{n-1} \subseteq M^n$ separating hypersurface achieving the infinum.
\end{restatable}
%
%
\noindent To prove this, we construct a $W^{1,2}$ function vanishing on a small sphere about a point, with energy bounded above by the $n-1$-dimensional hausdorff area of the sphere. In the small $\eps$ limit, this converges to $0$ for each $r_0 > 0$ (see \ref{fig:absmincounter}).\nl \nl
\begin{figure}[h!]
\centering
\includegraphics[scale=0.2]{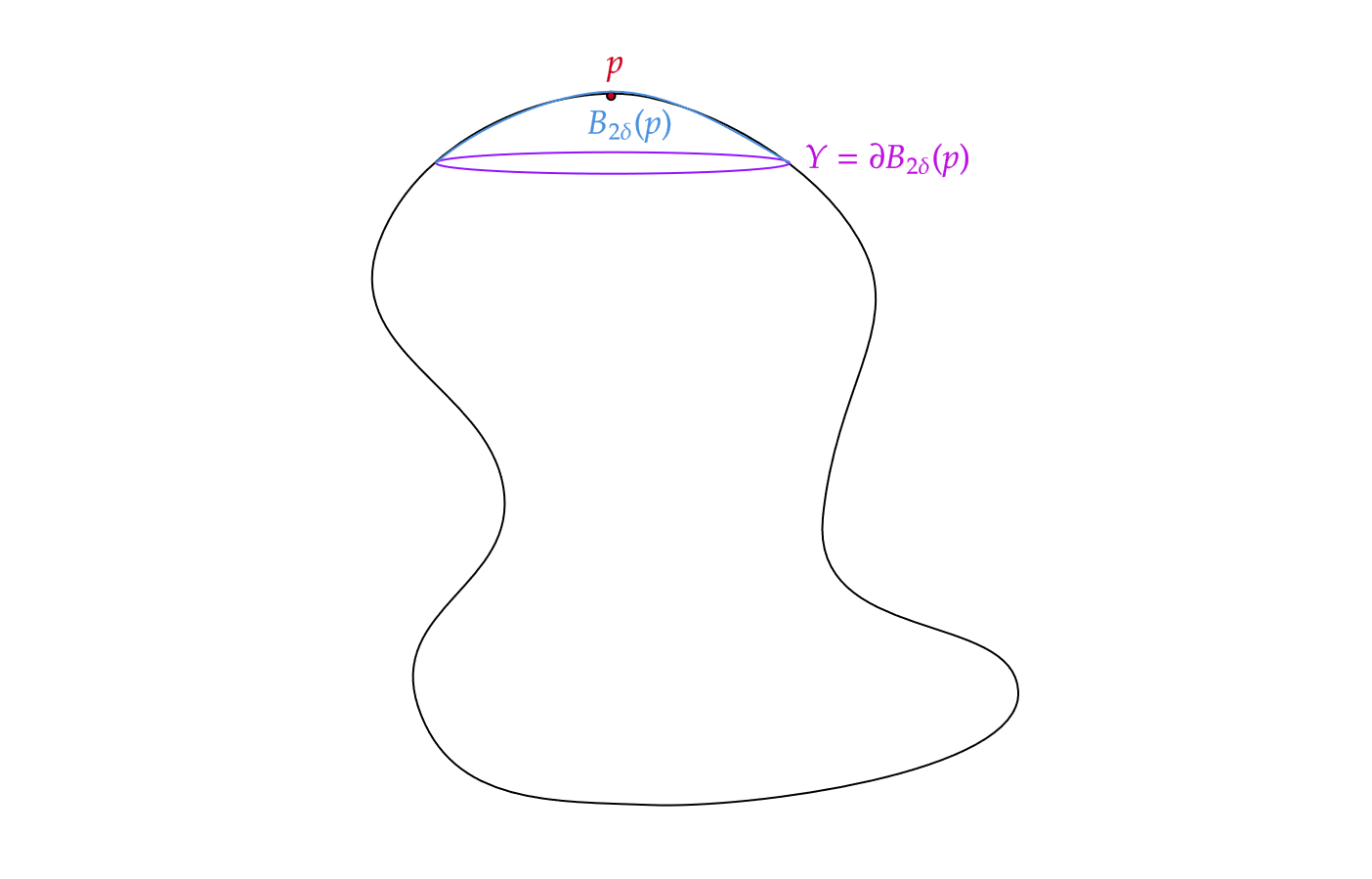}
\caption{}
\label{fig:absmincounter}
\end{figure}
\noindent \Pf Let $\eps$ fixed, $p \in M^n$, and $B_{\delta}(p)$, $B_{2\delta}(p)$ for $\delta > 0$ small. Consider 
\[
\psi: B_{2 \delta}(p) \to B_{2\delta}(0) \subseteq \R^n
\]
a diffeomorphism for $\delta$ sufficienty small such that the pull back metric on $B_{2\delta}(p)$ is approximately the pull back metric on $\R^n$. Now define 
\[
f(r) = \begin{cases}
	0 & 0 \leq r \leq \delta \\
	\frac{\ln(r/\delta)}{\ln(k)} & \delta \leq r \leq k \delta \\
	1 & r \geq k \delta
	\end{cases}
\]
Then we have that 
\[
f'(r) = \begin{cases}
	0 & 0 \leq r \leq \delta \\
	\frac{1}{r \ln(k)} & \delta \leq r \leq k \delta \\
	0 & r \geq k \delta
	\end{cases}
\]
And so 
\begin{align*}
E_{\eps}(f) &= \int_{M} \frac{\eps}{2} |\nabla f|^2 + \frac{1}{\eps} W(f) dA \\
&= \int_{B_{2\delta}(p)} \frac{\eps}{2} |\nabla f|^2 + \frac{1}{\eps} W(f) dA \\
&\leq \int_{B_{2\delta}(p)} \frac{\eps}{2} |f'(r)|^2 dA + 2\text{Vol}(B_{2\delta}(p)) \cdot \frac{1}{4 \eps} \\
&= \eps\pi \int_{\delta}^{k\delta} \frac{1}{r^2 \ln(k)^2} r dr + \frac{C \delta^n}{\eps} \\
&\leq \frac{\eps \tilde{C}}{\ln(k)^2} \int_{\delta}^{k \delta} \frac{1}{r} dr + \frac{C \delta^n}{\eps} \\
&= \begin{cases}
\frac{\eps \tilde{C}}{\ln(k)} + \frac{C \delta}{\eps} & n = 2 \\
\frac{\eps \tilde{C} \delta^{n-2}}{\ln(k)^2} + \frac{C \delta^n}{\eps} & n \geq 3
\end{cases}
\end{align*}
here, $C$ and $\tilde{C}$ can be taken to be $2$ times the volume of the balls of radius $1$ in $\R^n$ and surface area of the sphere of radius $1$ in $\R^n$ respectively.  \nl \nl
If $n \geq 3$, then we can simply send $\delta \to 0$ to prove the theorem. If $n = 2$, set $\delta = (k \ln(k))^{-1}$ (so that $\delta k$ is bounded in the definition of $f$) and send $k \to \infty$ so that
\[
\inf_Y \BE_{\eps}(Y) \leq \frac{\eps \tilde{C}}{\ln(k)} + \frac{C}{\eps k \ln(k)} \to 0
\]
which tells us that \textbf{on $M$ compact, no absolute minimizer of $\BE$ exists}. \nl \nl
For manifolds which are $2$ dimensional or larger, this leads us to look into minimizers which are only local critical points, in particular critical points which are Morse index $1$ or larger.
\subsubsection{Absolute minimizers on $S^1$}
%
Identify $S^1 \cong [0,1]$ with $0 \sim 1$ identified. For any $\eps > 0$, let $u_{0,\eps}$ denote the absolute minimizer of \eqref{ACEnergy} on $[0,1]$ vanishing at $0$ and $1$. In this section, we'll prove the following analogous theorem:
\begin{restatable}{thmm}{NoAbsoluteMinimizerCircle}\label{NoAbsoluteMinimizerS1}
For any $\eps > 0$ 
\[
\inf_{\substack{p \neq 0 \in S^1}} \BE_{\eps}(\{0, p\}) = E_{\eps}(u_{0,\eps}) > 0
\]
but there is no $2$-node solution which achieves the infinum.
\end{restatable}
%
%
\noindent To prove this, we first prove the following lemma
\begin{lemma} \label{MinEnergyContinuousLemma}
Let $N$ be a closed manifold with boundary $\partial Y = N$. Let $\{u_{\eps}\}$ be the minimizers of \eqref{ACEnergy} on $N$ restricted to functions in $W^{1,2}_0(N)$. Then there exists an $\eps_0 > 0$ such that for all $0 < \eps < \eps_0$, the function
\[
g(\eps) = E_{\eps}(u_{\eps}) = \inf_{v \in W^{1,2}_0(N)} E_{\eps}(v)
\]
is continuous in $\eps$. In fact, it is lipschitz away from $\eps = 0$.
\end{lemma}
\noindent \Pf Let $\gamma, \delta < \eps_0$ arbitrary and positive. Let 
\[
e_{\eps}(p) := \frac{\eps}{2} |\nabla u_{\eps}(p)|^2 + \frac{1}{\eps} W(u_{\eps}(p))
\]
Then
\begin{align*}
g(\gamma) - g(\delta) & =  \int_N [e_{\gamma} - e_{\delta}]
\end{align*}
We rewrite
\begin{align*}
e_{\gamma} &= \left[\frac{\delta}{2} |\nabla u_{\gamma}(p)|^2 + \frac{1}{\delta} W(u_{\gamma}(p)) \right] + \left[\frac{\gamma - \delta}{2} |\nabla u_{\gamma}(p)|^2 + (\gamma^{-1} - \delta^{-1}) W(u_{\gamma}(p))\right] \\
E_{\gamma}(u_{\gamma}) &= E_{\delta}(u_{\gamma}) + \int_N \frac{\gamma - \delta}{2} |\nabla u_{\gamma}(p)|^2 + (\gamma^{-1} - \delta^{-1}) W(u_{\gamma}(p)) \\
& \geq E_{\delta}(u_{\gamma}) - \left(\frac{|\gamma - \delta|}{\gamma} + \frac{|\delta - \gamma|}{\delta} \right) E_{\gamma}(u_{\gamma}) \\
& \geq E_{\delta}(u_{\gamma}) - K |\gamma - \delta| \\
& \geq E_{\delta}(u_{\delta}) - K|\gamma - \delta|
\end{align*}
Since $u_{\delta}$ is the minimizer of $E_{\delta}$ on $N$. Here $K$ depends on $E_{\gamma}(u_{\gamma})$ which we know is bounded as long as $\gamma > c_0 > 0$ -- take $0$ as a competitor to get a bound of $\text{Vol}(N)/4\gamma < \text{Vol}(N)/4c_0$ as a competing energy. This shows that 
\[
E_{\gamma}(u_{\gamma}) - E_{\delta}(u_{\delta}) \geq - K|\gamma - \delta|
\]
but by the same argument with $\gamma, \; \delta$ switched, we get 
\[
\Big| E_{\gamma}(u_{\gamma}) - E_{\delta}(u_{\delta}) \Big| \leq K |\gamma - \delta|
\]
This gives lipschitzness away from $0$, and hence continuity. \qed \nl \nl
We now prove theorem \ref{NoAbsoluteMinimizerS1}: \nl 
\noindent \Pf The argument is as follows: Take $\{p_i\}$ such that 
\[
\lim_{i \to \infty} \BE(\{0, p_i\}) = \inf_{\substack{p \neq 0 \in S^1}} \BE_{\eps}(\{0, p\})
\]
Then, potentially up to a subsequence, we have $\{p_i\}$ converges to some $p \in [0,1]$. If $p \neq 0, 1$, then $E_{\eps}(u) > E_{\eps}(u_{0,\eps})$ as $u_{0,\eps}$ is the unique absolute minimizer with dirichlet conditions on $\{0,1\}$. However, let $u_{\delta, \eps}$ be the function
\[
u_{\delta, \eps }(x) = \begin{cases}
u_{[0,1-\delta], \eps }(x) & x \in [0, 1 - \delta] \\
0 & x \in [1 - \delta, 1]
\end{cases}
\]
where $u_{[1-\delta], \eps }$ is the minimizer of \eqref{ACEnergy} on $[0,1-\delta]$ with dirichlet conditions at the endpoints. Then we have 
\[
\lim_{\delta \to 0} E_{\eps}(u_{\delta, \eps}, [0,1]) = \lim_{\delta \to 0} E_{\eps}(u_{[0,1-\delta], \eps}, [0,1-\delta]) + \frac{\delta}{4\eps}
\]
where we know that 
\[
\int_{1-\delta}^1 e_{\eps}(u_{\delta, \eps}) = \int_{1-\delta}^1 \left(0 + \frac{1}{4\eps} \right) = \frac{\delta}{4\eps} 
\]
Moreover for $\delta$ sufficiently small, we have that 
\[
\tilde{u}_{\eps, \delta}(x) := u_{[0,1-\delta],\eps}((1 - \delta) x)
\]
is a well defined solution to \eqref{ACEquation} with $\tilde{\eps} = \frac{\eps}{1 - \delta}$, and hence the unique solution to \eqref{ACEquation} corresponding to $\tilde{\eps}$. In particular, it is the minimizer of $E_{\tilde{\eps}}$ on $[0,1]$, and so by lemma \ref{MinEnergyContinuousLemma}, we have that 
\[
\lim_{\delta \to 0} E_{\eps}(u_{[0,1-\delta], \eps}, [0,1-\delta]) = \lim_{\delta \to 0} E_{\eps/(1-\delta)}(\tilde{u}_{\eps,\delta}) = E_{\eps}(u_{0, \eps})
\]
%
In particular, this tells us that for $\delta$ sufficiently small, we have
\[
E_{\eps}(u) > E_{\eps}(u_{\delta,\eps}) > \BE_{\eps}(\{0, 1 - \delta\})
\]
by our construction and the definition of $\BE_{\eps}$ on hypersurfaces. Thus, $u$ cannot be an absolute minimizer when $|Y| = 2$ since we've found a competitor. This is a contradiction, and tells us that $p = 0, 1$, so our minimizing solution must be $u_{0,\eps}$, which is only a $1$-node function on $S^1$.

\printbibliography

\end{document}